\documentclass[a4paper,10.5pt]{article}
\usepackage{amsmath,amssymb,amsthm}
\theoremstyle{definition}
 \newtheorem{dfn}{Definition}
 \newtheorem{remark}[dfn]{Remark}

\theoremstyle{plain}
 \newtheorem{thm}[dfn]{Theorem}
 \newtheorem{prop}[dfn]{Proposition}
 \newtheorem{lem}[dfn]{Lemma}
 
 \newcommand{\eq}[1]{\begin{equation}
\begin{split}
#1
\end{split}
\end{equation}}
\newcommand{\eqh}[1]{\begin{equation*}
\begin{split}
#1
\end{split}
\end{equation*}}

\numberwithin{equation}{section}
\newcommand{\bn}{{\bold n}}
\newcommand{\ba}{{\bold a}}
\newcommand{\bb}{{\bold b}}
\newcommand{\bk}{{\bold k}}
\newcommand{\bu}{{\bold u}}
\newcommand{\bv}{{\bold v}}
\newcommand{\bw}{{\bold w}}
\newcommand{\bff}{{\bold f}}

\newcommand{\bg}{{\bold g}}
\newcommand{\bh}{{\bold h}}

\newcommand{\bC}{{\bold C}}
\newcommand{\bD}{{\bold D}}

\newcommand{\bF}{{\bold F}}

\newcommand{\bI}{{\bold I}}

\newcommand{\bS}{{\bold S}}

\newcommand{\bU}{{\bold U}}
\newcommand{\bV}{{\bold V}}

\newcommand{\dv}{{\rm div}\,}

\newcommand{\BR}{{\Bbb R}}
\newcommand{\BC}{{\Bbb C}}

\newcommand{\BN}{{\Bbb N}}

\newcommand{\CA}{{\mathcal A}}
\newcommand{\CB}{{\mathcal B}}
\newcommand{\CC}{{\mathcal C}}
\newcommand{\CD}{{\mathcal D}}
\newcommand{\CE}{{\mathcal E}}
\newcommand{\CF}{{\mathcal F}}
\newcommand{\CI}{{\mathcal I}}

\newcommand{\CL}{{\mathcal L}}

\newcommand{\CR}{{\mathcal R}}
\newcommand{\CS}{{\mathcal S}}
\newcommand{\CT}{{\mathcal T}}
\newcommand{\CH}{{\mathcal H}}

\newcommand{\CP}{{\mathcal P}}

\newcommand{\CU}{{\mathcal U}}
\newcommand{\CV}{{\mathcal V}}

\newcommand{\CX}{{\mathcal X}}
\newcommand{\CY}{{\mathcal Y}}

\newcommand{\fp}{{\frak p}}

\newcommand{\pd}{\partial}

\newcommand{\Hol}{{\rm Hol}\,}

\newcommand{\kv}{\bk_\bv}
\newcommand{\vt}{\vartheta}
\newcommand{\Grad}{\nabla}
\newcommand{\lr}[1]{\left( #1 \right)}
\newcommand{\bd}{{\bf d}}
\newenvironment{cases*}%
{%
\left\{
\begin{array}{@{}r@{\;}l@{\quad}l@{}}
}%
{\end{array}\right.}
\begin{document}
\title{On strong dynamics of compressible two-component mixture flow}
\author{Tomasz PIASECKI
\thanks{Institute of Applied Mathematics and Mechanics, University of Warsaw, Banacha 2, 02-097 Warsaw, Poland.  
\endgraf
E-mail address: tpiasecki@mimuw.edu.pl
\endgraf
Partially supported by Top Global University Project and polish National Science Centre Harmonia project UMO-2014/14/M/ST1/00108.},
\quad 
Yoshihiro SHIBATA
\thanks{Department of Mathematics and Research Institute of 
Science and Engineering, 
Waseda University, 
Ohkubo 3-4-1, 
\endgraf
Shinjuku-ku, Tokyo 169-8555, Japan. 
\quad 
E-mail address: yshibata@waseda.jp 
\endgraf
Adjunct faculty member in the Department of Mechanical 
Engineering and Materias Science, University of Pittsburgh.
\endgraf 
Partially supported by  
JSPS Grant-in-aid for Scientific Research (A) 17H0109
and Top Global University Project.},
\enskip 
 and \enskip Ewelina ZATORSKA
\thanks{Department of Mathematics, University College London, Gower Street,  London WC1E 6BT, United Kingdom. 
\endgraf 
E-mail address: e.zatorska@ucl.ac.uk 
\endgraf E.Z. has been supported by the Polish Government MNiSW research grant 2016-2019 "Iuventus Plus"  no. IP2015 088874.}}
\date{}
\maketitle
\begin{abstract}
We investigate a system describing the flow of a compressible two-component mixture.
The system is composed of the compressible Navier-Stokes equations coupled with non-symmetric reaction-diffusion equations
describing the evolution of fractional masses. 
We show the local existence and, under certain smallness assumptions, also the global existence of unique strong 
solutions in $L_p-L_q$ framework. 
Our approach is based on so called entropic variables which enable to rewrite the system  
in a symmetric form. Then, applying Lagrangian coordinates, we show the local existence of solutions 
applying the $L_p$-$L_q$ maximal regularity estimate. Next, applying exponential decay estimate we 
show that the solution exists globally in time provided the initial data is sufficiently 
close to some constants. The nonlinear estimates impose restrictions $2<p<\infty,\   3<q<\infty$.
However, for the purpose of generality we show the linear estimates for wider range of $p$ and $q$.

\vskip5mm


\smallskip

\end{abstract}

\section{Introduction}
The Navier-Stokes-Maxwell-Stefan equations provide a description of the multicomponent reactive flows. The system consists of compressible Navier-Stokes equations for the barycentric velocity and total density as well as the convection-diffusion equations for the constituents of the mixture. 
The two subsystems are coupled by the form of the pressure in the momentum equation and the form of the fluxes in the species equations. The relation between the diffusion deriving forces for the constituents and the diffusion fluxes is called the Maxwell-Stefan equations. 

In this paper we are interested in analysis of a simple two-component mixture model with neglect of the heat-conduction and reactivity. The associated system of PDEs reads as follows
\begin{equation}\label{mf:1}\left\{
\begin{aligned}
\pd_t\rho +\dv(\rho\bu) = 0 \qquad
&\text{in $\Omega\times(0, T)$}, \\
\pd_t(\rho\bu) + \dv(\rho\bu\otimes\bu)-\dv \bS
+\nabla\fp=0 \qquad&\text{in $\Omega\times(0, T)$}, \\
\pd_t\rho_k + \dv(\rho_k\bu)
+ \dv \bF_k = 0\qquad
&\text{in $\Omega\times(0, T)$}, 
\end{aligned}
\right. \end{equation}
where $\rho$ denotes the total density of the flow and is a sum of partial densities of the species $\rho=\rho_1+\rho_2$, $\bu$ denotes 
the velocity vector field, $\fp$ denotes the pressure, 
$\bF_1,\ \bF_2$ the diffusion fluxes for both species and
$\bS$ denotes the stress tensor given by  
\begin{equation}\label{st:1} \bS=\mu\bD(\bu) +(\nu-\mu)\dv\bu\bI
\end{equation}
where $\bD(\bu) = \nabla\bu + {}^\top\nabla\bu$ is the doubled deformation 
tensor.
We assume the system \eqref{mf:1} is supplied with the initial and boundary conditions 
\begin{equation}\label{icbc}\left\{
\begin{aligned}
\bu=0, \quad \bF_1\cdot\bn = \bF_2\cdot\bn=0
\qquad&\text{on $\Gamma\times(0, T)$}, \\
(\bu, \rho_1, \rho_2)|_{t=0}=(\bu_0, \rho_{10}, 
\rho_{20})
\qquad&\text{in $\Omega$}.
\end{aligned}
\right. \end{equation}
Note that assuming the constraint on the diffusion fluxes $\bF_1+\bF_2=0$, the species equation, when summed, give the continuity equation. Therefore wenhave $\rho_1=\rho-\rho_2$, and so,
the unknowns of the system are $\rho,\ \bu$, and one of the partial densities $\rho_1$ or $\rho_2$. For the derivation of system \eqref{mf:1} from the kinetic theory of gases in the general multi-component, heat-conducting and reactive case we refer to the monograph of Giovangigli \cite{VG}. 

In this paper we consider the mixture of ideal gases, therefore the  internal pressure of the mixture is determined through the Boyle law
\begin{equation} \label{mf:2}
\fp = \frac{\rho_1}{m_1} + \frac{\rho_2}{m_2}. 
\end{equation}
Above, $m_{k}$ denotes the molar mass of the species $k$ and for simplicity, we set the gaseous constant  equal to 1. We are interested in the case when the pressure essentially depends on the densities of different species, therefore we assume
$$m_1\neq m_2.$$
The simplest form of the diffusion fluxes widely used in particular applications is the Fick approximation $\bF_k\approx-c\Grad\lr{\frac{\rho_k}{\rho}}$, see \cite{FTP}. The Fick law states that the flux of a species is proportional to the
gradient of the concentration of this species, and does not take into account the presence of all other components. However, in the real-word applications the cross-diffusion effects cannot be neglected, see for example \cite{BGS12, Wal62,T2000, BE}.
This issue can be solved by considering the so-called Maxwell-Stefan equations for multicomponent diffusion. These equations relate  the diffusion velocities ${\bold V}_i$ defined as $\bF_i=\rho_i{\bold V}_i$ and the  molar and the mass fractions respectively
$$X_i=\frac{\fp_i}{\fp},\quad Y_i=\frac{\rho_i}{\rho},$$
where $\fp_k=\frac{\rho_k}{m_k}$,  in the implicit way:
\begin{equation}\label{SM}
\underbrace{\Grad X_i-{\left(Y_i-X_i\right)\Grad \log \fp}}_{:={\bold d}_i}=\sum_{{j=1} \atop {j\neq i}}^{n}\left(\frac{X_i X_j}{{\mathcal D}_{ij}}\right)\left({\bold V}_j-{\bold V}_i\right),
\end{equation}
where  ${\mathcal D}_{ij}> 0$  denotes the binary diffusion coefficient, ${\mathcal D}_{ij}={\mathcal D}_{ji}$. 
The Maxwell-Stefan system \eqref{SM} was first treated by  Giovangigli \cite{Gio90, Gio91}, who used iterative methods to solve these equations, i.e. to find the inverse matrix that allows to characterize the fluxes as the functions of gradients of concentrations. It was proven that for positive concentrations Maxwell-Stefan relations lead to the following form of the fluxes
	\begin{equation}\label{eq:diff}
	\bF_{k}=-\sum_{l=1}^{n} C_{kl}{\bd}_l, \quad k=1,...n,
	\end{equation}
where $C_{kl}$ are multicomponent flux diffusion coefficients and $\bd_l=(d_{l}^{1},d_{l}^{2},d_{l}^{3})$ is the species $l$ diffusion force 
	\begin{equation}\label{eq:}
	d_{l}^{i}=\nabla_{x_{i}}\left({\fp_{l}\over \fp}\right)+\left({\fp_{l}\over \fp}-{\rho_{k}\over \rho}\right)\nabla_{x_{i}} \log{\fp}= \frac{1}{\fp}\lr{\nabla_{x_i} \fp_l-\frac{\rho_l}{\rho}\nabla_{x_i}\fp},
	\end{equation}
appearing in the Maxwell-Stefan equations \eqref{SM}. The main properties of the flux diffusion matrix $C$ discussed in \cite[Chapter 7]{VG}  are
\begin{equation} \label{prop_C}
C{\cal Y}={\cal Y}C^{T},\quad
   N(C)=\mbox{lin}\{\vec Y\},\quad
  R(C)={U}^{\bot},
\end{equation}
where ${\cal Y}=\mbox{diag}(Y_{1},\ldots,Y_{N})$, $\vec Y=(Y_1,\ldots,Y_n)^t$, 
$N(C)$ is the nullspace of $C$,  $R(C)$ is the range of $C$, 
$\vec U=(1,\ldots,1)^{T},$ and  ${U}^{\bot}$ is the orthogonal complement of $\mbox{lin}\{\vec U\}$.

In this paper we will use the explicit form \eqref{eq:diff}. In case of two components it reduces to  
\begin{equation} \label{mf:3}\
F_1= -\frac{1}{\fp}\left(\frac{\rho_2}{\rho}\nabla\left(\frac{\rho_1}{m_1}\right)
- \frac{\rho_1}{\rho}\nabla\left(\frac{\rho_2}{m_2}\right)\right),
\quad F_2 = -F_1. 
\end{equation}

Under the assumption \eqref{eq:diff}, global in time strong (unique) solutions around the constant
equilibrium for the Cauchy problem was proven by Giovangigli in \cite{VG}. 
He introduced the entropic and normal variables to symmetrize the system \eqref{mf:1} and applied the 
Kawashima and Shizuta theory \cite{K84, KS88} for symmetric hyperbolic-parabolic systems of conservation laws.
For the local in time existence result to the species mass balances equations in the isobaric, isothermal case we refer to \cite{B2010}, see also \cite{HMPW13}.  
Later on, J\"ungel and Stelzer generalized this result and combined it with the  entropy dissipation method to prove the global in time existence of weak solutions \cite{JS13}, still in the case of constant pressure and temperature. The detailed description of the method and its applicability for a range of models we refer to \cite{Jungel}.
For the qualitative and quantitative analysis of the ternary gaseous system together with numerical simulations we refer to \cite{BGS12}.
One should note that constant pressure assumption in \eqref{SM} not only significantly simplifies the cross-diffusion equations but basically decouples the fluid and the reaction-diffusion parts of the system \eqref{mf:1}. 
Stationary problems for compressible mixtures were considered in \cite{EZ} under the assumption of Fick law 
and later in \cite{GPZ, PP1, PP2} with cross diffusion, however for different molar masses.  
Existence of weak solutions for the mixture of non-newtonian fluids has been shown in \cite{BH15}.  
Let us also mention results on multi-phase systems \cite{FGM05b, KMP12} and incompressible mixtures \cite{MT13,CJ13, BP2017}. 
We would also like to mention the theoretical results for the systems describing the compressible reacting electrolytes \cite{DDGG}, where the authors prove the existence of global in time weak solutions to the Nernst--Planck--Poisson model originating from the modelling approach developed by Bothe and Dreyer in the previous paper \cite{BD2015}. 
The classical mixture models in the sense of \cite{VG} were studied in the series of papers  \cite{EZ2, EZ3, MPZ, MPZ1,MPZ2}, where the global in time existence of weak solutions was proved without any simplification of \eqref{eq:}. This was possible thanks to postulate of the so-called Bresch-Desjardins condition for the viscosity coefficients, which provides an extra estimate of the density gradient and a special form of the pressure.  The last restriction was recently removed by Xi and Xie \cite{XiXie}.

The global well-posedness in the framework of strong solutions for the compressible Navier-Stokes(-Fourier)
system under smallness assuptions on the data is already well investigated, see among others \cite{MaNi} in $L_2$ framework, 
\cite{WZ} in $L_p$ setting with slip boundary condition or \cite{S16, EBS} for a free boudary problem. 
However, for the system coupled with reaction-diffusion equations admitting cross-diffusion the issue of global well posedness of 
initial-boundary value problems has remained open. 

The purpose of this work is to prove the global in time existence of strong solutions to the system \eqref{mf:1}. 
Our basic observation is that this system enjoys some smoothing effect when written in terms of entropic 
variables \cite{VG}. Its symmetric structure enables us apply $L_p$-$L_q$ maximal regularity estimate to 
show the local well posedness and exponential decay estimate to show the global well-posedness under additional 
smallness assumptions. The linear estimates are based on the theory of $\CR$-bounded operators 
(see for instance \cite{DHP},\cite{Murata},\cite{SSchade}).  
The symmetrized system is derived in the next section. Afterwards we formulate our main 
results and discuss the structure of the remaining sections.

\section{Symmetrization and main Results.}
Since $\bF_1$ and $\bF_2$ are not independent, we  reduce
two diffusion equations to one diffusion equation introducing the normal form, see \cite[Chapter 8]{VG}.  Let 
\begin{equation} \label{trans:1}
(h,\rho) = \left(\frac{1}{m_2}\log \rho_2 - \frac{1}{m_1}\log \rho_1~,~~\rho_1 + \rho_2\right):=\Psi(\rho_1,\rho_2).
\end{equation}
Notice that $\Psi:\BR_+\times \BR_+ \to \BR\times \BR_+$ is a bijection, let us denote its inverse by $\Phi$.
Computing $\nabla h, \nabla \rho$ from \eqref{trans:1} and solving the resulting linear system for $\nabla \rho_1, \nabla \rho_2$
we get
\begin{equation}\label{norm:2}\begin{split}
\nabla\rho_1 & = \frac{m_1\rho_1}{m_1\rho_1 + m_2\rho_2}\nabla\rho
-\frac{m_1\rho_1m_2\rho_2}{m_1\rho_1 + m_2\rho_2}\nabla h,
\quad 
\nabla\rho_2  = \frac{m_2\rho_2}{m_1\rho_1 + m_2\rho_2}\nabla\rho
+\frac{m_1\rho_1m_2\rho_2}{m_1\rho_1 + m_2\rho_2}\nabla h.
\end{split}\end{equation}
From \eqref{norm:2} and the third equations in Eq. \eqref{mf:1},
we have
\begin{align*}
\pd_t h + \bu\cdot\nabla h 
& = \frac{1}{m_2\rho_2}\pd_t \rho_2
-\frac{1}{m_1\rho_1}\pd_t \rho_1
+ \frac{1}{m_2\rho_2}\bu\cdot\nabla\rho_2 - 
\frac{1}{m_1\rho_1}\bu\cdot\nabla\rho_1\\
& = \frac{1}{m_2\rho_2}\lr{-\rho_2\dv\bu - \dv\bF_2}
- \frac{1}{m_1\rho_1}\lr{-\rho_1\dv\bu - \dv\bF_1} \\
& = -\Bigl(\frac{1}{m_2}- \frac{1}{m_1}\Bigr)\dv\bu
- \frac{1}{m_2\rho_2}\dv \bF_2 + 
\frac{1}{m_1\rho_1}\dv\bF_1.
\end{align*}
Since $\bF_1 =- \bF_2 $, we have 
$$\pd_t h + \bu\cdot\nabla h
= -\Bigl(\frac{1}{m_1\rho_1} + \frac{1}{m_2\rho_2}\Bigr)
\dv\bF_2 - \left( \frac{1}{m_2}-\frac{1}{m_1} \right)\dv\bu,
$$
which leads to 
\begin{equation} \label{norm:3a}
\frac{m_1m_2\rho_1\rho_2}{m_1\rho_1 + m_2\rho_2}
(\pd_th + \bu\cdot\nabla h)
+\frac{(m_1-m_2)\rho_1\rho_2}{m_1\rho_1 + m_2\rho_2}
= -\dv \bF_2.
\end{equation}
Moreover, noting that $m_1$ and $m_2$ are positive constants, 
by \eqref{mf:3} and \eqref{norm:2} we have
\eq{\label{norm:3b}
-\bF_2 = \bF_1 = &\frac{1}{\fp}\Bigl(\frac{\rho_1}{\rho m_2}\nabla\rho_2
-\frac{\rho_2}{\rho m_1}\nabla\rho_1\Bigr)\\
=& \frac{1}{\fp}\Bigl\{\Bigl(\frac{\rho_1\rho_2}{\rho(m_1\rho_1 + m_2\rho_2)}
-\frac{\rho_1\rho_2}{\rho(m_1\rho_1 + m_2\rho_2)}\Bigr)\nabla\rho
+ \frac{m_1\rho_1^2\rho_2 + m_2\rho_1\rho_2^2}{\rho(m_1\rho_1 + m_2\rho_2)}
\nabla h\Bigr\} \\
=& \frac{\rho_1\rho_2}{\fp\rho}\nabla h.
}
Combining \eqref{norm:3a} and \eqref{norm:3b} formulas gives 
\begin{equation}\label{norm:3}
\frac{m_1m_2\rho_1\rho_2}{m_1\rho_1 + m_2\rho_2}
(\pd_t h + \bu\cdot\nabla h) + 
\frac{(m_1-m_2)\rho_1\rho_2}{m_1\rho_1 + m_2\rho_2}\dv\bu 
= \dv\Bigl(\frac{\rho_1\rho_2}{\fp\rho}\nabla h\Bigr).
\end{equation}
By \eqref{mf:2} and \eqref{norm:2}, we have
\begin{equation*}
\nabla \fp = \frac{1}{m_1}\nabla\rho_1 + \frac{1}{m_2}\nabla\rho_2 
= \frac{\rho}{m_1\rho_1 + m_2\rho_2}\nabla\rho
+ \frac{\rho_1\rho_2(m_1-m_2)}{m_1\rho_1 + m_2\rho_2}\nabla h,
\end{equation*}
Inserting this formula into the second equation in Eq. \eqref{mf:1},
we obtain
\begin{equation}\label{norm:4}
\rho(\pd_t\bu + \bu\cdot\nabla\bu) -\dv\bS 
+ \frac{\rho}{m_1\rho_1 + m_2\rho_2}\nabla\rho 
+ \frac{\rho_1\rho_2(m_1-m_2)}{m_1\rho_1+m_2\rho_2}\nabla h = 0.
\end{equation}
Concerning the boundary conditions, 
by \eqref{norm:3b} the condition $\nabla\cdot\bF_1 = 0$ is transformed to
$(\nabla h)\cdot\bn = 0$.
Thus, setting 
$$\Sigma_\rho = m_1\rho_1 + m_2\rho_2,
\quad\rho_0 = \rho_{10} + \rho_{20},
\quad h_0 = \frac{1}{m_2}\log\rho_{20} 
-\frac{1}{m_1}\log\rho_{10},
$$
by \eqref{norm:3} and \eqref{norm:4} 
we have the following equations for $\rho$, $\bu$ and $h$: 
\begin{equation}\label{neweq:1}\left\{
\begin{aligned}
\pd_t\rho + \dv(\rho\bu) &= 0&\quad&\text{in $\Omega\times(0, T)$}, \\
%%%%%%%%%%%%%%%%%
\rho(\pd_t\bu + \bu\cdot\nabla\bu) -\dv\bS 
+\frac{\rho}{\Sigma_\rho}\nabla\rho + 
\frac{(m_1-m_2)\rho_1\rho_2}{\Sigma_\rho}\nabla h &= 0
&\quad&\text{in $\Omega\times(0, T)$}, \\
%%%%%%%%%%%%%%%%%%%%%%
\frac{m_1m_2\rho_1\rho_2}{\Sigma_\rho}
(\pd_t h + \bu\cdot\nabla h) + 
\frac{(m_1-m_2)\rho_1\rho_2}{\Sigma_\rho}\dv\bu 
&= \dv\Bigl(\frac{\rho_1\rho_2}{p\rho}\nabla h\Bigr)
&\quad&\text{in $\Omega\times(0, T)$}, \\
\bu=0, \quad(\nabla h)\cdot\bn & = 0
&\quad&\text{on $\Gamma \times(0, T)$}, \\
(\rho, \bu, h)|_{t=0} & =(\rho_0, \bu_0, h_0)
&\quad&\text{in $\Omega$}.
\end{aligned}\right.\end{equation}
To solve Eq. \eqref{neweq:1} in the maximal $L_p$-$L_q$ regularity class, 
we introduce Lagrange coordinates $\{y\}$. 
Let $\bv(y, t)$ be the 
velocity field in the Lagrange coordinates and we consider the transformation:
\begin{equation}\label{lag:1}
x = y + \int^t_0\bv(y, s)\,ds.
\end{equation}
Then for any differentiable function $f$ we have 
\begin{equation} \label{dt_lag}
\pd_t f(t,\phi(t,y))=\pd_t f+\bv \cdot \nabla_x f.
\end{equation}
Moreover, since
\begin{equation}\label{lag:2}
\frac{\pd x_i}{\pd y_j} = \delta_{ij} 
+ \int^t_0\frac{\pd v_i}{\pd y_j}(y, s)\,ds
\end{equation}
where $\delta_{ij}$ are Kronecker's delta symbols, assuming
\begin{equation}\label{assump:1}
\sup_{t \in (0,T)}\int^t_0\|\nabla\bv(\cdot, s)\|_{L_\infty(\Omega)}\,ds
\leq \delta
\end{equation}
with some small positive constant $\delta$, the $N\times N$ matrix
$\pd x/\pd y = (\pd x_i/\pd y_j)$ has the inverse 
\begin{equation}\label{lag:3}
\Bigr(\frac{\pd x_i}{\pd y_j}\Bigr)^{-1} = \bI + \bV^0(\bk_\bv)
\end{equation}
where $\bk_\bv = \int^t_0\nabla\bv(y, s)\,ds$, $\bI$ is the $N\times N$
identity matrix, and $\bV^0(\bk)$ is the $N\times N$ matrix of 
smooth functions with respect to $\bk=(k_{ij} \mid i,j=1\ldots, N)
\in \BR^{N^2}$ defined on $|\bk| <\delta$ with $\bV^0(0) = 0$,
where $\bk$ is independent variables corresponding to 
$\bk_\bv$. 
Let $V^0_{ij}(\bk)$ be $(i, j)^{\rm th}$ components of $\bV^0(\bk)$, 
and then $V^0_{ij}(\bk)$ are smooth functions with respect to 
$\bk \in B_{\delta, N^2}$ with $V^0_{ij}(0) = 0$, where
$B_{\delta, N^2}$ denotes the ball of radius $\delta$ centered at the origin
in $\BR^{N^2}$. We have
\begin{equation}\label{lag:4}
\nabla_x = (\bI + \bV^0(\bk_\bv))\nabla_y, 
\quad \frac{\pd}{\pd x_i} = \sum_{j=1}^N (\delta_{ij} + V^0_{ij}(\bk_\bv))
\frac{\pd}{\pd y_j}.
\end{equation}
Moreover, as was seen in Str\"omer \cite{St1}, the map:
$x = \Phi(y, t)$ is bijection from $\Omega$ onto
$\Omega$, and so setting
\begin{equation}\label{lag:5}
\bv(y, t) = \bu(x, t),
\quad \eta(y, t) = \rho(x, t), \quad
\vartheta(y, t) = h(x, t)
\end{equation}
we see that Eq. \eqref{neweq:1} is transformed to the following equations:
\begin{equation}\label{neweq:2}\left\{
\begin{aligned}
\pd_t\eta + \eta\dv\bv &= R_1(U)&\quad&\text{in $\Omega\times(0, T)$}, \\
%%%%%%%%%%%%%%%%%
\eta\pd_t\bv - \mu\Delta\bv - \nu\nabla\dv\bv 
+\frac{\eta}{\Sigma_\rho}\nabla\eta + 
\frac{(m_1-m_2)\rho_1\rho_2}{\Sigma_\rho}\nabla \vartheta &= R_2(U)
&\quad&\text{in $\Omega\times(0, T)$}, \\
%%%%%%%%%%%%%%%%%%%%%%
\frac{m_1m_2\rho_1\rho_2}{\Sigma_\rho}\pd_t \vartheta + 
\frac{(m_1-m_2)\rho_1\rho_2}{\Sigma_\rho}\dv\bv
-\dv\Bigl(\frac{\rho_1\rho_2}{\fp\rho}\nabla \vartheta\Bigr)
&=R_3(U)
&\quad&\text{in $\Omega\times(0, T)$}, \\
\bv=0, \quad(\nabla \vartheta)\cdot\bn & = R_4(U)
&\quad&\text{on $\Gamma \times(0, T)$}, \\
(\eta, \bv, \vartheta)|_{t=0} & =(\rho_0, \bu_0, h_0)
&\quad&\text{in $\Omega$}.
\end{aligned}\right. \end{equation}
Here, $R_1(U)$, $R_2(U)$, $R_3(U)$ and $R_4(U)$ are nonlinear functions
with respect to $U = (\eta, \bv, \vartheta)$, which are  
given in Sect. \ref{sec:2} below. 

Our main results are the following two theorems. The first one concerns the local well-posedness.
\begin{thm}\label{thm:main1}
Let $2 < p < \infty$, $3 < q < \infty$ and $L > 0$. Assume that 
$2/p + 3/q < 1$ and that $\Omega$ is a uniform $C^3$ domain in
$\BR^N$ ($N \geq 2$).  Let $\rho_{10}(x)$, $\rho_{20}(x)$, and 
$\bu_0(x)$ be initial data for Eq. \eqref{mf:1}. Assume that there exist
positive numbers $a_1$ and $a_2$ for which
\begin{equation}\label{initial:0}
a_1 \leq \rho_{10}(x), \enskip \rho_{20}(x) \leq a_2
\quad\text{for any $x \in \overline{\Omega}$}.
\end{equation}
Let $(h_0(x), \rho_0(x)) = \Psi(\rho_{10}(x), \rho_{20}(x))$. 
 Then, there exists a time $T>0$ depending on
$a_1$, $a_2$ and $L$ such that if 
$\rho_{10}$, $\rho_{20}$, $\bu_0$ and $h_0$ satisfy the condition:
\begin{equation}\label{initial:1}
\|\nabla(\rho_{10}, \rho_{20})\|_{L_q(\Omega)}
+ \|\bu_0\|_{B^{2(1-1/p)}_{q,p}(\Omega)} 
+ \|h_0\|_{B^{2(1-1/p)}_{q,p}(\Omega)}
\leq L
\end{equation}
and the compatibility condition:
\begin{equation}\label{initial:2}
\bu_0|_\Gamma=0, \quad (\nabla h_0)\cdot\bn|_\Gamma = 0,
\end{equation}
then problem \eqref{neweq:2} admits a unique solution 
$(\eta, \bv, \vartheta)$ with
\begin{gather*}
\eta - \rho_0 \in H^1_p((0, T), H^1_q(\Omega)),
\quad \bv \in H^1_p((0, T), L_q(\Omega)^3) \cap L_p((0, T), H^2_q(\Omega)^3),\\
\vartheta \in H^1_p((0, T), L_q(\Omega)) \cap L_p((0, T), H^2_q(\Omega))
\end{gather*}
possessing the estimates:
\begin{gather*}
\|\eta-\rho_0\|_{H^1_p((0, T), H^1_q(\Omega))}
+ \|\pd_t(\bv, \vartheta)\|_{L_p((0, T), L_q(\Omega))}
+ \|(\bv, \vartheta)\|_{L_p((0, T), H^2_q(\Omega))}
\leq CL, \\ a_1 \leq \rho(x,t) \leq 2a_2+a_1
\quad\text{for $(x, t) \in \Omega\times(0, T)$}, \quad 
 \int^T_0\|\nabla\bv(\cdot, s)\|_{L_\infty(\Omega)}
\leq \delta. 
\end{gather*}
Here, $C$ is some constant independent of $L$. 
\end{thm}
The second main result gives the global well-posedness:
\begin{thm}\label{thm:main2}
Let $2 < p < \infty$, $3 < q < \infty$ and $L > 0$. Assume that 
$2/p + 3/q < 1$ and that $\Omega$ is a bounded domain whose
boundary $\Gamma$ is a  compact $C^3$ hypersurface.  
Let $\rho_{1*}$ and $\rho_{2*}$ be
any positive numbers and set $(h_*, \rho_*) = \Psi(\rho_{1*}, \rho_{2*})
\in \BR\times \BR_+$.  Then, there exists a small number $\epsilon>0$
 depending on
$\rho_{1*}$, $\rho_{2*}$  such that if the initial data
$(\rho_0, \bu_0, h_0)$ satisfy the smallness condition:
\begin{equation}\label{initial:3}
\|(\rho_{10}-\rho_{1*}, \rho_{20}-\rho_{2*})\|_{H^1_q(\Omega)}
+ \|\bu_0\|_{B^{2(1-1/p)}_{q,p}(\Omega)} 
+ \|h_0-h_*\|_{B^{2(1-1/p)}_{q,p}(\Omega)}
\leq \epsilon
\end{equation}
and the compatibility condition \eqref{initial:2}
then problem \eqref{neweq:2} with $T = \infty$ admits a unique solution 
$(\eta, \bv, \vartheta)$ with
\begin{gather*}
\eta \in H^1_p((0, \infty), H^1_q(\Omega)),
\quad \bv \in H^1_p((0, T), L_q(\Omega)^N) 
\cap L_p((0, \infty), H^2_q(\Omega)^N),\\
\vartheta \in H^1_p((0, \infty), L_q(\Omega)) \cap L_p((0, T), H^2_q(\Omega))
\end{gather*}
possessing the estimates:
\begin{align*}
&\|e^{\gamma t}\nabla\eta\|_{L_p((0, \infty), L_q(\Omega))}
+ \|e^{\gamma t}\pd_t\eta\|_{L_p((0, \infty), H^1_q(\Omega))}
+ \|e^{\gamma t}\pd_t(\bv, \vartheta)\|_{L_p((0, \infty), L_q(\Omega))} 
+ \|e^{\gamma t}\bv\|_{L_p((0, \infty), H^2_q(\Omega))}\\
&\quad
+ \|e^{\gamma t}\nabla\vartheta\|_{L_p((0, \infty), H^1_q(\Omega))}
+ \|(\rho_1, \rho_2) - (\rho_{1*}, \rho_{2*})
\|_{L_\infty((0, \infty), H^1_q(\Omega))}
\leq C\epsilon, \\
&\rho_{i*}/4\leq \rho_i(x, t) \leq 4\rho_{i*}
\quad\text{in $(x, t) \in \Omega\times(0, \infty)$
for $i =1,2$}, \quad 
 \int^T_0\|\nabla\bv(\cdot, s)\|_{L_\infty(\Omega)}
\leq \delta
\end{align*}
for some constant $C > 0$ independent of $\epsilon$. 
\end{thm}
The rest of the paper is organized as follows.  In Sect.3, we derive
the formulas $R_i(U)$ ($i=1, \ldots, 4$) in the right side of Eq. 
\eqref{neweq:2}. In Sect.4, assuming the maximal $L_p$-$L_q$ theory for the 
linearized equations, we prove Theorem \ref{thm:main1}. In Sect.5, assuming 
the decay properties of solutions of the linearized equations, we prove
Theorem \ref{thm:main2}.  In Sect.6, we prove the maximal $L_p$-$L_q$
regularity for the linearized equations and in Sect.7 we prove the decay theorem
for the linearized equations.

%%%%%%%%% notation 
\begin{center}
{\bf Notation} 
\end{center}
We conclude this section by summarizing the symbols used throughout
the paper. 
We denote the sets of all complex numbers, real numbers
 and natural numbers by
$\BC$, $\BR$, and $\BN$, respectively, and $\BN_0 = \BN \cup\{0\}.$
$\bI$ stands for the $N\times N$ identity matrix or the identity
operator.
For any multi-index $\alpha = (\alpha_1, \ldots, \alpha_N) 
\in \BN_0^N$ we set 
$\partial_x^\alpha h = 
\partial_1^{\alpha_1}\cdots \partial_N^{\alpha_N} h$
with $\partial_i = \partial/\partial x_i$.  In particular, 
for scalar functions $\theta$, vector functions  
$\bu = {}^\top(u_1, \ldots, u_N)$ and 
$n \in \BN_0$ we set  
$\nabla^n\theta = (\partial_x^\alpha\theta \mid |\alpha|=n)$ and 
$\nabla^n \bu = (\nabla^n u_j \mid j = 1, \ldots, N)$.
In particular, $\nabla^0 = \bI$, 
$\nabla^1 = \nabla$.
 For any $N$-vectors $\ba={}^\top(a_1, \ldots, a_N)$ and $\bb
={}^\top(b_1, \ldots, b_N)$, 
let $\ba\cdot \bb =<\ba, \bb>= \sum_{j=1}^N\ba_j\bb_j$.  
For any domain $G$ in $\BR^N$, let $L_q(G)$, $H^m_q(G)$,  and 
$B^s_{q,p}(G)$  
be the standard Lebesgue, Sobolev, and 
Besov spaces on $G$, 
and let  $\|\cdot\|_{L_q(G)}$, 
$\|\cdot\|_{H^m_q(G)}$, and  
$\|\cdot\|_{B^s_{q,p}(G)}$ denote their respective norms.
Let $(\cdot, \cdot)_{\theta, p}$ and $(\cdot, \cdot)_{[\theta]}$ denote
the real interpolation functor and complex interpolation functor,
respectively.  Note that $B^{m+\theta}_{q,p}(G) = (
H^m_q(G), H^{m+1}_q(G))_{\theta, p}$. 
For a Banach space $X$ with norm $\|\cdot\|_X$, 
let $X^d = \{(f_1, \ldots, f_d) \mid f_i \in X \enskip
 (i=1, \ldots, d)\}$, 
and write the norm of $X^d$ as simply  $\|\cdot\|_X$, which is 
defined by    
$\|f\|_X = \sum_{j=1}^d \|f_j\|_X$ for 
$f = (f_1, \ldots, f_d)
\in X^d$.  
Let 
\begin{align*}
&\CH_q(G) = \{F = (f_1, \bff_2, f_3) \mid f_1 \in H^1_q(G), \enskip 
\bff_2 \in L_q(G)^N, \enskip f_3 \in L_q(G)\}, \\
&\|F\|_{\CH_q(G)} = \|f_1\|_{H^1_q(G)} + \|(\bff_2, f_3)\|_{L_q(G)}
\quad\text{for $F = (f_1, \bff_2, f_3) \in \CH_q(G)$}, \\
&D_q(G)  = \{U = (\zeta, \bv, \vartheta) \mid \zeta \in H^1_q(G), \enskip
\bv \in H^2_q(G)^N, \enskip \vartheta \in H^2_q(G)\}, \\ 
&\|U\|_{D_q(\Omega)}  = \|\zeta\|_{H^1_q(G)} + \|(\bv, \vartheta)\|_{H^2_q(G)}
\quad\text{for $U=(\zeta, \bv, \vartheta) \in D_q(G)$}, \\
&D_{p,q}(G) = \{U_0 = (\zeta_0, \bv_0, \vartheta_0) \mid 
\zeta_0 \in H^1_q(G), \enskip \bv_0 \in B^{2(1-1/p)}_{q,p}(G)^N, \enskip 
\vartheta_0 \in B^{2(1-1/p)}_{q,p}(G)\}, \\ 
&\|U_0\|_{D_{p,q}(G)}  = \|\zeta_0\|_{H^1_q(G)} +  
\|(\bv_0, \vartheta_0)\|_{B^{2(1-1/p)}_{q,p}(G)}
\quad\text{for $U_0 = (\zeta_0, \bv_0, \theta_0) \in D_{p,q}(G)$}, \\
&Y_q(G) =\{(h_1, h_2) \mid h_1 \in L_q(G), 
\quad h_2 \in H^1_q(G)\}, \quad 
\|(h_1, h_2)\|_{Y_q(G)} = \|h_1\|_{L_q(G)} \\
&\CY_q(G)  = \{(F_1, F_2, F_3) \mid F_1, F_2 \in L_q(\Omega),
\quad F_3 \in H^1_q(\Omega)\},  \\
&\|(F_1, F_2, F_3)\|_{\CY_q(G)} = \|(F_1, F_2)\|_{L_q(G)} + \|F_3\|_{H^1_q(G)}.
\end{align*}
Let 
$(\bu, \bv)_G = \int_G\bu\cdot\bv\,dx$ and  let 
$(\bu,\bv)_{\partial G} = \int_{\partial G} \bu\cdot \bv\,d\omega$, 
where $d\omega$ denotes the surface element on $\partial G$.  
%%%%%%%%%%%%%%%%%%%%%%%
For $1 \leq p \leq \infty$, $L_p((a, b), X)$ and 
$H^m_p((a, b), X)$ denote the standard  Lebesgue and Sobolev spaces of
$X$-valued functions defined on an interval $(a, b)$, 
and $\|\cdot\|_{L_p((a, b), X)}$,  
$\|\cdot\|_{H^m_p((a, b), X)}$ denote their respective norms. 
Let $H^s_p(\BR, X)$ be the standard $X$-valued Bessel potential space and
$\|\cdot\|_{H^s_p(\BR, X)}$ its norm. 
Let $C^\infty_0(G)$ be the set of all $C^\infty$ functions
whose supports are compact and contained in $G$. 
For two Banach spaces $X$ and $Y$, $X+Y = \{x + y \mid x \in X, y\in Y\}$,
$\CL(X, Y)$ denotes the 
set of all bounded linear operators from $X$ into $Y$ and 
$\CL(X, X)$ is written simply as $\CL(X)$. 
%%%%%%%%%%%%%%%%%%%%%%%
For a domain $U$ in $\BC$, $\Hol(U, \CL(X, Y))$
 denotes the set of all $\CL(X, Y)$-valued holomorphic 
functions defined on $U$. 
Let $\Sigma_\epsilon = \{\lambda \in \BC\setminus\{0\} \mid
|\arg\lambda| \leq \pi-\epsilon\}$ and $\Sigma_{\epsilon, \lambda_0}
= \{\lambda \in \Sigma_\epsilon \mid |\lambda| \geq \lambda_0\}$.
%%%%%%%%%%%%%%%%%%%%%%%%  
Moreover, 
the letter $C$ denotes a
generic constant and $C_{a,b,c,\cdots}$ denotes that the 
constant $C_{a,b,c,\cdots}$ depends 
on $a$, $b$, $c, \cdots$.  The value of 
$C$ and $C_{a,b,c,\cdots}$ may change from line to line.

%%%%%%%%%%%%%%% derivation
\section{Lagrange transformation}\label{sec:2}
In this section we rewrite all necessary differential operators under the
Lagrange transformation \eqref{lag:1} under the assumption \eqref{assump:1}. 
This way e obtain exact form of the right hand side of \eqref{neweq:2}.
We have  
\begin{equation}\label{lag:div}
\dv_x = \dv_y + \sum_{i,j=1}^n\bV^0_{ij}(\kv)\frac{\pd v_i}{\pd y_j},
\end{equation} 
therefore by \eqref{dt_lag},\eqref{lag:4} and \eqref{lag:5}, we obtain \eqref{neweq:2}$_1$ with
\begin{equation}\label{lag:6}
R_1(U) = -\eta\sum_{i,j=1}^N V^0_{ij}(\bk_\bv)\frac{\pd v_i}{\pd y_j}.
\end{equation}
Here and in the following, we set $U = (\eta, \bv, \vartheta)$.
Now we have to transform second order operators.
By \eqref{lag:4}, we have
$$
\Delta \bu = \sum_{k=1}^3\frac{\pd}{\pd x_k}\lr{\frac{\pd \bu}{\pd x_k}}
= \sum_{k,\ell,m=1}^3\lr{\delta_{k\ell} + V^0_{kl}(\bk_\bv)}
\frac{\pd}{\pd y_\ell}
\lr{(\delta_{km} + V^0_{km}(\bk_\bv))\frac{\pd \bv}{\pd y_m}},
$$
and so 
setting 
\begin{align*}
A_{2\Delta}(\bk)\nabla^2\bv &= 2\sum_{\ell, m=1}V^0_{k\ell}(\bk)
\frac{\pd^2\bv}{\pd y_\ell\pd y_m}
+ \sum_{k,\ell, m=1}^N
V^0_{k\ell}(\bk)V^0_{km}(\bk)
\frac{\pd^2\bv}{\pd y_\ell \pd y_m}, \\
A_{1\Delta}(\bk)\nabla\bv & = \sum_{\ell, m=1}^3(\nabla_\bk V^0_{\ell m})(\bk)
\int^t_0(\pd_l\nabla\bv)\,ds \frac{\pd \bv}{\pd y_m}
+ \sum_{k, \ell, m=1}^3
V^0_{k\ell}(\bk) (\nabla_\bk V^0_{km})(\bk)
\int^t_0\pd_\ell\nabla\bv\,ds\frac{\pd \bv}{\pd y_m}
\end{align*}
we have
$$\Delta \bu = \Delta \bv + A_{2\Delta}(\bk_\bv)\nabla^2\bv
+ A_{1\Delta}(\bk_\bv)\nabla\bv.
$$
And also, by \eqref{lag:4}, we have
$$\frac{\pd}{\pd x_j}\dv\bu 
= \sum_{k=1}^3(\delta_{jk} + V^0_{jk}(\bk_\bv))\frac{\pd}{\pd y_k}
\lr{\dv\bv + \sum_{\ell, m=1}^3 V^0_{\ell m}(\bk_\bv)\frac{\pd v_\ell}{\pd y_m}},
$$
and so setting
\begin{align*}
A_{2\dv,j}(\bk)\nabla^2\bv
& = \sum_{\ell, m=1}^3V^0_{\ell m}(\bk)\frac{\pd^2 v_\ell}{\pd y_m\pd y_j}
+ \sum_{k=1}^3 V^0_{jk}(\bk)\frac{\pd}{\pd y_k}\dv\bv
+ \sum_{k, \ell=1}^3V^0_{jk}(\bk)V^0_{\ell m}(\bk)
\frac{\pd^2v_\ell}{\pd y_k\pd y_m}, \\
A_{1\dv, j}(\bk)\nabla\bv
& = \sum_{\ell, m=1}^3(\nabla_\bk V^0_{\ell m})(\bk)
\int^t_0\pd_j\nabla\bv\,ds\frac{\pd v_\ell}{\pd y_m} 
+ \sum_{k,\ell, m=1}^3V^0_{jk}(\bk)(\nabla_\bk V^0_{\ell m})(\bk)
\int^t_0\pd_k\nabla\bv\,ds\frac{\pd v_\ell}{\pd y_m},
\end{align*}
we have
$$\frac{\pd}{\pd x_j}\dv\bu
= \frac{\pd}{\pd y_j}\dv\bv + A_{2\dv, j}(\bk_\bv)\nabla^2\bv
+ A_{1\dv, j}(\bk_\bv)\nabla\bv.
$$
By \eqref{lag:4}, we have
\begin{align*}
\frac{\rho}{\Sigma_\rho}\nabla\rho+ 
\frac{(m_1-m_2)\rho_1\rho_2}{\Sigma_\rho}\nabla h
= \frac{\eta}{\Sigma_\rho}(\nabla\eta + V^0(\bk_\bv)\nabla\eta)+
\frac{(m_1-m_2)\rho_1\rho_2}{\Sigma_\rho}\lr{\nabla \vartheta
+ V^0(\bk_\bv)\nabla\vartheta}.
\end{align*}
Thus, noting that $\pd_t\bu + \bu\cdot\nabla\bu = \pd_t\bv$
and setting 
\begin{equation}\label{lag:7}\begin{split}
R_2(U) &= \mu A_{2\Delta}(\bk_\bv)\nabla^2\bv 
+ \mu A_{1\Delta}(\bk_\bv)\nabla\bv
+ \nu A_{2\dv}(\bk_\bv)\nabla^2\bv + \nu A_{1\dv}(\bk_\bv)\nabla\bv \\
&- \frac{\eta}{\Sigma_\rho}V^0(\bk_\bv)\nabla\eta
- \frac{(m_1-m_2)\rho_1\rho_2}{\Sigma_\rho}V^0(\bk_\bv)\nabla\vartheta,
\end{split}\end{equation}
where $A_{i\dv}(\bk)\nabla^i\bv = {}^\top(A_{i\dv,1}(\bk)\nabla^i\bv,
\ldots, A_{i\dv,N}(\bk)\nabla^i\bv)$ ($\nabla^1 = \nabla$), we have
$$\eta\pd_t\bv - \mu\Delta\bv - \nu\nabla\bv 
+ \frac{\eta}{\Sigma_\rho}\nabla\eta 
+ \frac{(m_1-m_2)\rho_1\rho_2}{\Sigma_\rho}\nabla\vt =R_2(U)
\quad\text{in $\Omega\times(0, T)$}.
$$
By \eqref{lag:4}, we have
\begin{align*}
\dv_x\Bigl(\frac{\rho_1\rho_2}{\fp\rho}\nabla h\Bigr)
&= \frac{\rho_1\rho_2}{\fp\rho}(\Delta \vartheta 
+ A_{2\Delta}\nabla^2(\bk_\bv)\vartheta + A_{1\Delta}(\bk_\bv)\nabla\vartheta)
+ \nabla_x\Bigl(\frac{\rho_1\rho_2}{\fp\rho}\Bigr)\cdot(\nabla\vartheta
+ V^0(\bk_\bv)\nabla\vartheta) \\
&= \dv_y\Bigl(\frac{\rho_1\rho_2}{\fp\rho}\nabla \vartheta\Bigr)
+ \frac{\rho_1\rho_2}{\fp\rho}
(A_{2\Delta}(\bk_\bv)\nabla^2 \vartheta + A_{1\Delta}(\bk_\bv)\nabla\vartheta)\\
&\quad+  \lr{2V^0(\bk_\bv)+(V^0(\bk_\bv))^2} \nabla_y\Bigl(\frac{\rho_1\rho_2}{\fp\rho}\Bigr)\nabla\vartheta.
\end{align*}
Thus, noting that $\pd_th + \bu\cdot\nabla h = \pd_t\vartheta$ and 
setting
\eq{\label{lag:8}
R_3(U) = &\frac{\rho_1\rho_2}{\fp\rho}
(A_{2\Delta}(\bk_\bv)\nabla^2 \vartheta + A_{1\Delta}(\bk_\bv)\nabla\vartheta)
+ \nabla\Bigl(\frac{\rho_1\rho_2}{\fp\rho}\Bigr)V^0(\bk_\bv)\nabla\vartheta\\
&-\frac{(m_1-m_2)\rho_1\rho_2}{\Sigma_\rho}
\sum_{j,k=1}^3V^0_{jk}(\bk_\bv)\frac{\pd v_j}{\pd y_k},
}
we obtain \eqref{neweq:2}$_3$.  

Finally, by the Taylor formula we have 
$$\bn(x) = \bn\lr{y + \int^t_0\bv(y, s)\,ds}
= \bn(y) + \int^1_0(\nabla\bn)
\lr{y + \tau\int^t_0\bv(y, s)\,ds}\,d\tau
\int^t_0\bv(y, s)\,ds,
$$
and so setting
\eq{\label{lag:9}
R_4(U)=& -\bn\lr{y + \int^t_0\bv(y, s)\,ds)\cdot (\bV^0(\bk_\bv)\nabla \vartheta}\\
&- \lr{\int^1_0(\nabla\bn)
(y + \tau\int^t_0\bv(y, s)\,ds)\,d\tau
\int^t_0\bv(y, s)\,ds} \cdot\nabla\vartheta,
}
we obtain \eqref{neweq:2}.

%%%%%%%%%%%%%% section
\section{Linearized problem for the local well-posedness}\label{sec:3}

Let $\rho_{10}(x)$, $\rho_{20}(x)$ and $\bu_0(x)$ be initial 
data for Eq. \eqref{mf:1}. Let $\alpha_1$ and $\alpha_2$ be postive numbers 
for which we assume that
\begin{equation}\label{3.1} 
\alpha_1 \leq \rho_{10}(x), \enskip \rho_{20}(x) \leq \alpha_2
 \enskip\text{for any $x \in \overline{\Omega}$},
\quad \|\nabla(\rho_{10}, \rho_{20})\|_{L_r(\Omega)} \leq \alpha_2
\end{equation}
where $\alpha_1$ and $\alpha_2$  are some positive constants and 
$3< r < \infty$. 
Let $(h_0(x), \rho_0(x)) 
= \Psi(\rho_{10}(x),
\rho_{20}(x))$, where $\Psi$ is defined in \eqref{trans:1}. 
Obviously, since $\rho_0(x)=\rho_{10}(x)+\rho_{20}(x), $we have 
\begin{equation}\label{3.3} 2\alpha_1 \leq \rho_0(x) \leq 2\alpha_2,
\quad |h_0(x)| \leq \alpha_3
\end{equation}
where $\alpha_3= (\frac{1}{m_1}+\frac{1}{m_2})(|\log \alpha_1|+|\log \alpha_2|)$. 
We linearize Eq. \eqref{neweq:2} at $(\rho_{10}(x), \rho_{20}(x), 0)$.
Let  
\begin{equation}\label{sub:1}\begin{split}
&\rho = \rho_0(x) + \zeta, \quad \Sigma_\rho^0(x) 
= m_1\rho_{10}(x) + m_2\rho_{20}(x), 
\quad \gamma_1(x)= \frac{\rho_0(x)}{\Sigma_\rho^0(x)}, \\
&\gamma_2(x) = \frac{(m_1-m_2)\rho_{10}(x)\rho_{20}(x)}{\Sigma_\rho^0(x)},
\quad 
\gamma_3(x) = \frac{m_1m_2\rho_{10}(x)\rho_{20}(x)}{\Sigma^0_\rho(x)}, \\
&\gamma_4(x) = \frac{\rho_{10}(x)\rho_{20}(x)}{\fp_0(x)\rho_0(x)},
\quad \fp_0(x) = \frac{\rho_{10}(x)}{m_1} + \frac{\rho_{20}(x)}{m_2}.
\end{split}\end{equation}
We then write Eq. \eqref{neweq:2} as
\begin{equation}\label{neweq:3}\left\{\begin{aligned}
\pd_t\zeta + \rho_0\dv\bv = f_1(U)& &\quad&\text{in $\Omega\times(0, T)$}, \\
\rho_0\pd_t\bv - \mu\Delta\bv - \nu\nabla\dv\bv
+\gamma_1\nabla\zeta+ \gamma_2\nabla\vartheta  = \bff_2(U)&
&\quad&\text{in $\Omega\times(0, T)$}, \\
\gamma_3\pd_t\vartheta +\gamma_2\dv\bv- \dv(\gamma_4\nabla\vartheta)  = f_3(U)&
&\quad&\text{in $\Omega\times(0, T)$}, \\
\bv=0, \quad (\nabla \vartheta)\cdot\bn  = g(U)&
&\quad&\text{on $\Gamma\times(0, T)$}, \\
(\zeta, \bv, \vartheta)|_{t=0}  = (0, \bu_0, h_0)&
&\quad&\text{in $\Omega$},
\end{aligned}\right.\end{equation}
where we have set $U=(\rho, \bv, \vartheta)$, $\rho = \rho_0(x) + \zeta$, 
and 
\allowdisplaybreaks{
\begin{align}
f_1(U) & = R_1(U) - \zeta\dv\bv, \nonumber \\
\bff_2(U) & = R_2(U) -\zeta\pd_t\bv 
- (\rho_0+\zeta)\lr{\frac{1}{\Sigma_\rho}-
\frac{1}{\Sigma_\rho^0}}\nabla(\rho_0+\zeta)
- \frac{\rho_0+\zeta}{\Sigma_\rho^0}\nabla(\rho_0) 
\\ 
&\quad -\frac{\zeta}{\Sigma_\rho^0}\nabla\zeta \nonumber 
-(m_1-m_2)\lr{\frac{\rho_1\rho_2}{\Sigma_\rho} -
 \frac{\rho_{10}\rho_{20}}{\Sigma^0_\rho}
}\nabla\vartheta, \nonumber \\
f_3(U) & = R_3(U) -m_1m_2\lr{\frac{\rho_1\rho_2}{\Sigma_\rho}
-\frac{\rho_{10}\rho_{20}}{\Sigma_\rho^0}}\pd_t\vartheta
-(m_1-m_2)\lr{\frac{\rho_1\rho_2}{\Sigma_\rho}
-\frac{\rho_{10}\rho_{20}}{\Sigma_\rho^0}}\dv\bv\nonumber \\
&\quad+\dv\lr{\lr{\frac{\rho_1\rho_2}{\fp\rho}
-\frac{\rho_{10}\rho_{20}}{\fp_0\rho_0}}\nabla\vartheta},
\nonumber\\
g(U) & = R_4(U). \label{neweq:4}
\end{align}
}
In the next section we solve the system \eqref{neweq:3} by the Banach fixed point theorem. 
%%%%%%%%%%%%%%%%%%%
\section{Local well-posedness -- proof of Theorem \ref{thm:main1}}
\label{sec:5} 

To prove the local well-posedness, we use the maximal regularity result for
the following equations:
\begin{equation}\label{linear:1}\left\{
\begin{aligned}
\pd_t\zeta + \rho_0(x)\dv\bv  = f_1& &\quad&\text{in $\Omega\times(0, T)$}, \\
\rho_0(x)\pd_t\bv - \mu\Delta\bv - \nu\nabla\dv\bv + \gamma_1(x)\nabla\zeta
+ \gamma_2(x)\nabla\vartheta
 = \bff_2&&\quad&\text{in $\Omega\times(0, T)$}, \\
\gamma_3(x)\pd_t\vartheta 
+ \gamma_2(x)\dv\bv-\dv(\gamma_4(x)\nabla\vartheta) = f_3&
&\quad&\text{in $\Omega\times(0, T)$}, \\
\bv|_\Gamma = 0, \quad (\nabla\vartheta)\cdot\bn = g&
&\quad&\text{on $\Gamma\times(0, T)$}, \\
(\zeta, \bv, \vartheta)|_{t=0} = (\zeta_0, \bv_0, \vartheta_0)&
&\quad&\text{in $\Omega$}.
\end{aligned}\right. 
\end{equation}
Here $\gamma_1(x)$, 
$\gamma_2(x)$, $\gamma_3(x)$ and $\gamma_4(x)$ have been given in 
\eqref{sub:1}. 
We assume that $\rho_{10}(x)$, $\rho_{20}(x)$ are
uniformly continous functions defined on $\overline{\Omega}$
satisfying \eqref{3.1}. Then we see immediately that
there exist positive constants $\alpha_3 < \alpha_4$ depending on 
$\alpha_1$ and $\alpha_2$ for which 
\begin{gather}
\alpha_3 \leq \rho_0(x), \gamma_1(x), 
\gamma_3(x), \gamma_4(x) 
\leq \alpha_4
\enskip \text{for $x \in \overline{\Omega}$}, \nonumber \\
\|\nabla(\rho_0, \gamma_1, \gamma_2, \gamma_3, \gamma_4)\|_{L_r(\Omega)}
\leq \alpha_4. \label{cond:4*}
\end{gather}
For a Banach space $X$ with norm $\|\cdot\|_X$, 
let $H^s_p(\BR, X)$ be a $X$ valued Bessel potential space of 
order $s \in (0, 1)$ defined by 
$$H^s_p(\BR, X) = \{f \in L_p(\BR, X) \mid \|f\|_{H^s_p(\BR, X)} < \infty\},
\quad 
\|f\|_{H^s_p(\BR, X)}
 = \|\CF^{-1}[(1+\tau^2)^{s/2}\CF[f](\tau)]\|_{L_p(\BR, X)},
$$
where $\CF$ and $\CF^{-1}$ denote the Fourier transform and its inverse formula.
The following theorems gives maximal $L_p-L_q$ regularity estimate for the system \eqref{linear:1}.
\begin{thm}\label{thm:linear:1}
Let $1 < p, q < \infty$, $2/p + 1/q \not=2$ and $2/p + 1/q \not=1$. 
Assume that $\Omega$ is a uniformly $C^2$ domain. Then, 
there exists a constant $\gamma_0$ for which the following
assertion holds. 
Let 
\begin{gather*}
\zeta_0 \in H^1_q(\Omega), \enskip \bv_0 \in B^{2(1-1/p)}_{q,p}(\Omega)^3, 
\enskip \vartheta_0 \in B^{2(1-1/p)}_{q,p}(\Omega), \\
\enskip  
f _1 \in L_p((0, T), H^1_q(\Omega)), \enskip \bff_2 \in 
L_p((0, T), L_q(\Omega)^3), \quad f_3 \in L_p((0, T), L_q(\Omega)), \\
e^{-\gamma t}g \in L_p(\BR, H^1_q(\Omega)) \cap H^{1/2}_p(\BR, L_q(\Omega))
\end{gather*}
for any $\gamma \geq \gamma_0$. 
Assume that $\bv_0$ and $\vartheta_0$ satisfy the 
compatibility conditions: 
$$\bv_0|_\Gamma = 0 \enskip \text{on $\Gamma$ for 
$2/p + 1/q < 2$}, \quad 
(\nabla \vartheta_0)\cdot\bn= g|_{t=0}\enskip\text{on $\Gamma$ 
 for 
$2/p + 1/q < 1$}.
$$
 Then, problem \eqref{linear:1} admits unique solutions
$\zeta$, $\bv$ and $\vartheta$ with
\begin{gather*}
\zeta \in H^1_p((0, T), H^1_q(\Omega)),
\quad \bv \in H^1_p((0, T), L_q(\Omega)^3) \cap L_p((0, T), H^2_q(\Omega)^3), \\\vartheta \in H^1_p((0, T), L_q(\Omega)) \cap L_p((0, T), H^2_q(\Omega))
\end{gather*}
possessing the estimate:
\begin{align*}
&\|\zeta\|_{H^1_p((0, T), H^1_q(\Omega))} 
+ \|\pd_t(\bv, \vartheta)\|_{L_p((0, T), L_q(\Omega)^N)} 
+ \|(\bv, \vartheta)\|_{L_p((0, T), H^2_q(\Omega))} \\
&\quad \leq C_\gamma e^{\gamma T}\{\|\rho_0\|_{H^1_q(\Omega)} 
+ \|(\bv_0, \vartheta_0)\|_{B^{2(1-1/p)}_{q,p}(\Omega)}
+ \|(f_1, \bff_2, f_3)\|_{L_p((0, T), L_q(\Omega))}\\
&\phantom{\quad \leq Ce^{\gamma T}\{\|\rho_0\|_{H^1_q(\Omega)}}\,\,
+ \|e^{-\gamma t}g\|_{L_p(\BR, H^1_q(\Omega))}
+ \|e^{-\gamma t}g\|_{H^{1/2}_p(\BR, L_q(\Omega))}\} 
\end{align*}
for any $\gamma \geq \gamma_0$, where $C$ is a constant 
depending on $\gamma$.
\end{thm}
\begin{remark} All the constants appearing in Theorem \ref{thm:linear:1}
depend on $\alpha_1$ and $\alpha_2$.
\end{remark}
%%%%%%%%%%%%%% local well-posedness

Postponing the proof of Theorem \ref{thm:linear:1}, we prove Theorem 
\ref{thm:main1}.  Let $\CH_{T,M}$ be the underlying space for
our fixed point argument, which is defined by 
\begin{equation}\label{5.1}\begin{split}
\CH_{T,M} & = \{(\zeta, \bv, \vartheta) \mid 
\zeta \in H^1_p((0, T), H^1_q(\Omega)), \quad
\bv \in H^1_p((0, T), L_q(\Omega)^3) \cap L_p((0, T), H^2_q(\Omega)^3),
\\
& \vartheta \in H^1_p((0, T), L_q(\Omega)) \cap L_p((0, T), H^2_q(\Omega)),
\quad (\zeta, \bv, \vartheta)|_{t=0} = (0, \bu_0, h_0) \quad\text{in $\Omega$}, \\
&[\zeta,\bv,\vartheta]_T =
\|\zeta\|_{H^1_p((0, T), H^1_q(\Omega))}
+ \|\pd_t(\bv, \vartheta)\|_{L_p((0, T), L_q(\Omega))}
+ \|(\bv, \vartheta)\|_{L_p((0, T), H^2_q(\Omega))}
\leq M\}.
\end{split}\end{equation}
Here, $T$ and $M$ are positive constants determined later.  Since $T$ will be 
chosen positive small number eventually, we may assume that $0 < T \leq 1$. 
Moreover, by Sobolev's inequality and H\"older's inequality, we have 
\begin{align*}
\int^T_0\|\nabla\bv(\cdot, t)\|_{L_\infty(\Omega)}\,dt
\leq C\int^T_0\|\bv(\cdot, t)\|_{H^2_q(\Omega)}\,dt
\leq T^{1/{p'}}
\Bigl(\int^T_0\|\bv(\cdot, t)\|_{H^2_q(\Omega)}^p\,dt\Bigr)^{1/p}
\leq MT^{1/p'}.
\end{align*}
Thus, choosing $T > 0$ so small that $MT^{1/p'}
\leq \delta$, we may assume that the condition \eqref{assump:1} holds
for any $(\zeta, \bv, \vartheta) \in \CH_{T,M}$. 
Let 
$$\CI = \|\nabla\rho_0\|_{H^1_q(\Omega)} 
+ \|(\bv, h_0)\|_{B^{2(1-1/p)}_{q,p}(\Omega)},
$$
and then by \eqref{initial:1} we have 
\begin{equation}\label{5.2}
\CI \leq L
\end{equation}
because $\rho_0(x) = \rho_{10}(x) + \rho_{20}(x)$. Let $\Psi$
 be the map defined in \eqref{trans:1}, which is $C^\infty$ diffeomorphism
from $\BR_+\times\BR_+$ onto $\BR\times \BR_+$.  Let $\Phi$ be its
inverse map.  Let $(\omega, \bw, \varphi) \in \CH_{T, M}$,
let $U = (\rho_0(x)+\omega, \bw, \varphi)$, and 
let $(\rho_1, \rho_2) = \Phi(\varphi, \rho_0+\omega)$. Since $(\omega, \bw, \varphi)|_{t=0} = 
(0, \bu_0, h_0)$, we have
\begin{equation}\label{5.3}
(\rho_{10}(x), \rho_{20}(x))=\Phi(\varphi, \rho_0(x) + \omega)|_{t=0}.
\end{equation}
Let  $R_i(U)$ be functions given in \eqref{lag:6}, \eqref{lag:7},
\eqref{lag:8}, and \eqref{lag:9}, where  $\eta$, 
$\bv={}^\top(v_1, \ldots, v_N)$, and $\vartheta$ are replaced  by 
$\rho_0 + \omega$, $\bw = {}^\top(w_1, \ldots, w_N)$, and 
$\varphi$. 
Let $(\zeta, \bv, \vartheta)$ be a solution of Eq. \eqref{linear:1} with 
$\zeta_0=0, \bv_0 = \bu_0$, $\vartheta_0=h_0$, 
$f_1 = f_1(U)$,
$\bff=\bff_2(U)$, $f_3=f_3(U)$ and $g=g(U)$,
where $\zeta$, $\bv$ and $\vartheta$ are replaced by
$\omega$, $\bw$ and $\varphi$, respectively. 

First, we estimate $f_1 = f_1(U)$,
$\bff=\bff_2(U)$, $f_3=f_3(U)$ and $g=g(U)$.
Notice that
\begin{equation}\label{5.4}\begin{split}
\sup_{t \in (0, T)} \|\omega(\cdot, t)\|_{H^1_q(\Omega)}
&\leq T^{1/{p'}}M \leq M, \\
\sup_{t \in (0, T)} \|\varphi(\cdot, t) - h_0\|_{B^{2(1-1/p)}_{q,p}(\Omega)}
+\sup_{t \in (0, T)} \|\bw(\cdot, t) -\bu_0\|_{B^{2(1-1/p)}_{q,p}(\Omega)}
&\leq C(M+L).
\end{split}\end{equation}
In fact, since $\omega(\cdot, 0) = 0$, we have
$$\|\omega(\cdot, t)\|_{H^1_q(\Omega)}
\leq \int^t_0\|\pd_t\omega)(\cdot, s)\|_{H^1_q(\Omega)}\,ds
\leq T^{1/{p'}}\|\pd_t\omega\|_{L_p((0, T), H^1_q(\Omega))}
\leq T^{1/{p'}}M \leq M,
$$
where we have used the fact that $T\leq 1$ in the last step. 
To prove the bound for the second term in \eqref{5.4}, we use the extension map
$e_T$ defined by 
\begin{equation}\label{ext:1} e_T[f](\cdot, t)
= \begin{cases}
0 \quad &t < 0, \\ f(\cdot, t) \quad &0 < t < T, \\
f(\cdot, 2T-t)\quad & T < t < 2T, \\
0 \quad& t > 2T.
\end{cases}
\end{equation}
Obviously, $e_T[f](\cdot, t) = f(\cdot, t)$ for $t \in (0, T)$.  If
$f|_{t=0}$,  then we have
\begin{equation}\label{ext:2} \pd_te_T[f](\cdot, t)
= \begin{cases}
0 \quad &t < 0, \\ (\pd_tf)(\cdot, t) \quad &0 < t < T, \\
-(\pd_tf)(\cdot, 2T-t)\quad & T < t < 2T, \\
0 \quad& t > 2T.
\end{cases}
\end{equation}
Let $X$ and $Y$ be two Banach spaces such that
$X$ is a dense subset of $Y$ and $X\subset Y$ is continuous,
and then we know (cf. \cite[p.10]{Tanabe}) that 
\begin{equation} \label{real-int:1}
H^1_p((0, \infty), Y) \cap L_p((0, \infty), X) 
\subset C([0, \infty), (X, Y)_{1/p,p})
\end{equation}
and
\begin{equation} \label{real-int:2}
\sup_{t \in (0, \infty)}\|u(t)\|_{(X, Y)_{1/p,p}}
\leq (\|u\|_{L_p((0, \infty),X)}^p
+ \|u\|_{H^1_p((0, \infty), Y)}^p)^{1/p}
\end{equation}
for each $p \in (1, \infty)$,  Applying this fact and using 
\eqref{ext:1} and \eqref{ext:2}, we have
\begin{align*}
&\sup_{t \in (0, T)}\|\varphi(\cdot, t)-h_0\|_{B^{2(1-1/p)}_{q,p}(\Omega)}
\leq \sup_{t \in (0, \infty)}\|e_T[\varphi-h_0]\|_{B^{2(1-1/p)}_{q,p}(\Omega)}\\&\quad 
= (\|e_T[\varphi-h_0]\|_{L_p((0, \infty), H^2_q(\Omega))}^p
+ \|e_T[\varphi-h_0]\|_{H^1_p((0, \infty), L_q(\Omega))}^p)^{1/p}\\
&\quad \leq C(\|\varphi-h_0\|_{L_p((0, \infty), H^2_q(\Omega))}
+ \|\pd_t\varphi\|_{L_p((0, T), L_q(\Omega))})
\leq C(M + T^{1/p}L) \leq C(M+L).
\end{align*}
 Here and in the following, $C$ denotes a generic constant independent 
 of $M$, $L$, and $T$.  $C$ depends at most on $a_1$ and $a_2$,  
for which \eqref{initial:0} holds.
Analogously, we have the
third inequality in \eqref{5.4}.

Since $2/p + 3/q < 1$, we have $1 + 3/q < 2(1-1/p)$, and so by Sobolev's 
imbedding theorem and \eqref{5.4} we have
\begin{equation}\label{5.4*}
\|(\varphi, \bw)\|_{L_\infty((0, T), H^1_\infty(\Omega))}
\leq CM.
\end{equation}

Since $\rho_0(x) = \rho_{10}(x) + \rho_{20}(x)$, by \eqref{initial:0} we have 
\begin{equation}\label{3.3*}
2a_1 \leq \rho_0(x) \leq 2a_2\quad\text{for $x \in \Omega$}.
\end{equation}
If we choose $T>0$ so small that $T^{1/{p'}}M \leq a_1$, by \eqref{3.3*} and 
\eqref{5.4}, we have
\begin{equation}\label{5.5}
a_1 \leq \rho_0(x) + \omega \leq 2a_2+a_1
\end{equation}
for all $(x,t) \in \Omega\times(0,T)$. 
Since $\Phi$ is a $C^\infty$ diffeomorphism from 
$\BR\times\BR_+$ onto $\BR_+\times\BR_+$, for any
compact set $A \subset \BR\times\BR_+$
$\Phi(A)$ is a compact set in $\BR_+\times \BR_+$,
and so by \eqref{5.5} and \eqref{5.4*}, there exist
positive constants $a_4$ and $a_5$ depending on
$a_1$, $a_2$, and $M$  for which
\begin{equation}\label{5.7}
a_4 \leq \rho_1(x,t), \rho_2(x,t) \leq a_5
\quad\text{for $(x, t) \in \Omega\times(0, T)$}.
\end{equation}

We now prove that 
\begin{equation}\label{diff-est:1}
\|(\rho_1, \rho_2) - (\rho_{10}, \rho_{20})\|_{L_\infty((0, T), H^1_q(\Omega))}
\leq C(L+M)T^{\theta/{p'}}
\end{equation}
for some $\theta \in (0, 1)$.  By \eqref{5.3} we have
\begin{equation}\label{5.9}\begin{split}
&\sup_{t \in (0, T)}\|(\rho_1(\cdot, t), \rho_2(\cdot, t))-(\rho_{10}(\cdot),
\rho_{20}(\cdot))\|_{L_q(\Omega)}
\leq \int^T_0\|\pd_t\Phi(\varphi(\cdot, t), 
\rho_0(\cdot)+\omega(\cdot, t))\|_{L_q(\Omega)}\,dt\\
&\quad \leq \int^T_0
\|\Phi'(\varphi(\cdot, t), \rho_0(\cdot) + \omega(\cdot, t))
\|_{L_\infty(\Omega)}
\|(\pd_t\varphi(\cdot, t), \pd_t\omega(\cdot, t))\|_{L_q(\Omega)}\,dt. 
\end{split}\end{equation}
By \eqref{5.4*} and \eqref{5.5}, we have
\begin{equation}\label{5.8}
\sup_{t \in (0, T)}\|\Phi'(\varphi(\cdot, t), 
\rho_0(\cdot) + \omega(\cdot, t))\|_{L_\infty(\Omega)}
\leq a_6
\end{equation}
for some positive constant $a_6$ depending on 
$a_1$, $a_2$, $M$  but independent of $T$.  Thus, by \eqref{5.9} 
we have
\begin{equation}\label{5.10}\begin{split}
\sup_{t \in (0, T)}\|(\rho_1(\cdot, t), \rho_2(\cdot, t))
-(\rho_{10}(\cdot), \rho_{20}(\cdot))\|_{L_q(\Omega)} 
&\leq a_6\int^T_0\|(\pd_t\varphi(\cdot, t), 
\pd_t\omega(\cdot, t))\|_{L_q(\Omega)}\,dt \\
& \leq a_6T^{1/{p'}}\|\pd_t(\varphi, \omega)\|_{L_p((0, T), L_q(\Omega))} 
\leq a_6MT^{1/{p'}}.
\end{split}\end{equation}
Moreover, by \eqref{initial:1} and  \eqref{5.2} we have
\begin{align*}
&\|\nabla(\rho_1(\cdot, t), \rho_2(\cdot, t))-(\rho_{10}(\cdot),
\rho_{20}(\cdot))\|_{L_q(\Omega)} \\
&\quad
\leq \|\Phi'(\varphi(\cdot, t), \rho_0(\cdot)+\omega(\cdot, t)
\|_{L_\infty(\Omega)}
\|(\nabla\varphi(\cdot, t), \nabla\rho_0(\cdot)
+\nabla\omega(\cdot, t))\|_{L_q(\Omega)}
+ \|\nabla(\rho_{10}, \rho_{20})\|_{L_q(\Omega)} \\
&\quad\leq a_6(\|\nabla\varphi(\cdot, t)\|_{L_q(\Omega)}
+ \|\nabla\omega(\cdot, t)\|_{L_q(\Omega)})
+a_6\|\nabla\rho_0\|_{L_q(\Omega)}
+ \|\nabla(\rho_{10}, \rho_{20})\|_{L_q(\Omega)}.
\end{align*}
Thus, by \eqref{5.4} 
\begin{equation}\label{5.11}
\sup_{t \in (0, T)}\|\nabla\{(\rho_1(\cdot, t), \rho_2(\cdot, t))
-(\rho_{10}(\cdot), \rho_{20}(\cdot))\}\|_{L_q(\Omega)}
\leq C(L+M).
\end{equation}
Since $W^{3/q+\epsilon}_q(\Omega) \subset L_\infty(\Omega)$ 
with some small $\epsilon$ for which $3/q  + \epsilon < 1$ and 
this inclusion is continuous as follows from Sobolev's imbedding theorem,
by real interpolation theorem
\begin{equation}\label{5.12}\begin{split}
&\sup_{t \in (0, T)}\|(\rho_1(\cdot, t), \rho_2(\cdot, t))
-(\rho_{10}(\cdot), \rho_{20}(\cdot))\|_{L_\infty(\Omega)}\\
&\quad \leq (\sup_{0 \in (0, T)}
\|(\rho_1(\cdot, t), \rho_2(\cdot, t))
-(\rho_{10}(\cdot), \rho_{20}(\cdot))\|_{L_q(\Omega)})^{\theta}\\
&\qquad\times
(\sup_{0 \in (0, T)}
\|(\rho_1(\cdot, t), \rho_2(\cdot, t))
-(\rho_{10}(\cdot), \rho_{20}(\cdot))\|_{H^1_q(\Omega)})^{1-\theta}
%\\
%&
\leq C(M+L)T^{\theta/{p'}}
\end{split}\end{equation}
with $\theta = 1-(3/q+\epsilon) \in (0, 1)$.  
By \eqref{5.12}, \eqref{5.7}, and
\eqref{initial:0}, we have
\eq{\label{5.13}
%\begin{split}
\Bigl\|\frac{1}{\Sigma_\rho}- \frac{1}{\Sigma_\rho^0}
\Bigr\|_{L_\infty((0, T), L_\infty(\Omega))}
%\leq C(M+L)T^{\theta/{p'}}, \\
+\Bigl\|\frac{\rho_1\rho_2}{\Sigma_\rho}
-\frac{\rho_{10}\rho_{20}}{\Sigma_\rho^0}
\Bigr\|_{L_\infty((0, T), L_\infty(\Omega))}
%& \leq C(M+L)T^{\theta/{p'}}, \\
+\Bigl\|\frac{\rho_1\rho_2}{\fp\rho}
-\frac{\rho_{10}\rho_{20}}{\fp_0\rho_0}
\Bigr\|_{L_\infty((0, T), L_\infty(\Omega))}\\
%& 
\leq C(M+L)T^{\theta/{p'}}.
%\end{split}
}
Moreover, by \eqref{5.11} we have
$$\sup_{t \in (0, T)}\|\nabla(\rho_1(\cdot,t), 
\rho_2(\cdot, t))\|_{L_q(\Omega)} \leq C(L+M),
$$
and so by \eqref{5.7} and \eqref{initial:0} we get
\begin{equation}\label{5.14}
\Bigl\|\nabla\Bigl(\frac{\rho_1\rho_2}{\fp\rho} - 
\frac{\rho_{10}\rho_{20}}{\fp_0\rho_0}\Bigr)
\Bigr\|_{L_\infty((0, T), L_q(\Omega))}
\leq C(M+L).
\end{equation}
Using \eqref{5.4*}, \eqref{5.4}, \eqref{5.13}, and \eqref{5.14},
we conclude
\allowdisplaybreaks{
\begin{align}
&\Bigl\|(\rho_0+\omega)\Bigl(\frac{1}{\Sigma_\rho}
-\frac{1}{\Sigma_\rho^0}\Bigr)\nabla(\rho_0+\omega)
\Bigr\|_{L_p((0, T), L_q(\Omega))} \nonumber\\
&\quad  \leq C\|\rho_0+\omega\|_{L_\infty((0, T), H^1_q(\Omega))}^2
T^{1/p}(M+L)T^{\theta/{p'}}
 \leq C(M+L)^3T^{(1/p+\theta/{p'})}; \nonumber\\
&\Bigl\|\frac{\rho_0+\omega}{\Sigma_\rho^0}\nabla\rho_0
\Bigr\|_{L_p((0, \infty), L_q(\Omega))}
 \leq C(M+L)LT^{1/p};\nonumber\\
&\Bigl\|\frac{\omega}{\Sigma^0_\rho}\nabla\omega
\Bigr\|_{L_p((0, T), L_q(\Omega))}
 \leq C\|\omega\|_{L_\infty((0, T), H^1_q(\Omega))}T^{1/p}
\leq CL^2T^{1/p};\nonumber\\
&\Bigl\|\Bigl(\frac{\rho_1\rho_2}{\Sigma_\rho}
- \frac{\rho_{10}\rho_{20}}{\Sigma^0_\rho}\Bigr)\nabla\varphi
\Bigr\|_{L_p((0, T), L_q(\Omega))}
 \leq C(M+L)LT^{(\theta/{p'} + 1/p)};\nonumber\\
&\Bigl\|\Bigl(\frac{\rho_1\rho_2}{\Sigma_\rho}
- \frac{\rho_{10}\rho_{20}}{\Sigma^0_\rho}\Bigr)\pd_t\omega
\Bigr\|_{L_p((0, T), L_q(\Omega))}
\leq C(M+L)T^{\theta/{p'}}\|\pd_t\omega\|_{L_p((0, T), L_q(\Omega))}
\leq CM(M+L)T^{\theta/{p'}}; \nonumber\\
&\Bigl\|\Bigl(\frac{\rho_1\rho_2}{\Sigma_\rho}
- \frac{\rho_{10}\rho_{20}}{\Sigma^0_\rho}\Bigr)\dv\bw
\Bigr\|_{L_p((0, T), L_q(\Omega))} \nonumber\\
&\quad
\leq C(M+L)T^{\theta/{p'}}T^{1/p}\|\bw\|_{L_\infty((0, T), H^1_q(\Omega))}
\leq C(M+L)^2T^{(\theta/{p'} + 1/p)};\nonumber \\
&\Bigl\|\dv\Bigl(\Bigl(\frac{\rho_1\rho_2}{\fp\rho}
-\frac{\rho_{10}\rho_{20}}{\fp_0\rho_0}\Bigr)\nabla\varphi\Bigr)
\Bigr\|_{L_p((0, T), L_q(\Omega))} \nonumber\\
&\quad \leq C(M+L)T^{\theta/{p'}}\|\varphi\|_{L_p((0, T), H^2_q(\Omega))} 
\nonumber\\
&\phantom{aaaaaaaaaaaaaaaaaaa} + 
\Bigl\|\nabla\Bigl(\frac{\rho_1\rho_2}{\fp\rho} - 
\frac{\rho_{10}\rho_{20}}{\fp_0\rho_0}\Bigr)
\Bigr\|_{L_\infty((0, T), L_q(\Omega))}
\|\nabla\varphi\|_{L_\infty((0, T), L_\infty(\Omega))}T^{1/p}
\nonumber\\
&\quad \leq C(M(M+L)T^{\theta/{p'}}
+ (M+L)^2T^{1/p}).
\label{5.15}\end{align}

Next, we estimate nonlinear terms from the Lagrange transformation. 
In \eqref{neweq:3}, we set $U = (\omega, \bw, \varphi)$. 
Recall that $3 < q < \infty$. 
By Sobolev's inequality and  \eqref{5.4}, we have
$$\|\omega\dv\bw\|_{H^1_q(\Omega)} \leq C\|\omega\|_{H^1_q(\Omega)}
\|\bw\|_{H^2_q(\Omega)} \leq CT^{1/p'}M\|\bw\|_{H^2_q(\Omega)},
$$
and so, we have
$$\|\omega\dv\bw\|_{L_p((0, T), H^1_q(\Omega))}
\leq CT^{1/{p'}}M\|\bw\|_{L_p((0, T), H^2_q(\Omega))}
\leq CT^{1/{p'}}M^2.
$$
Replacing $\bv$ by $\bw$ in \eqref{lag:6}, 
by Sobolev's inequality and \eqref{5.4}, we have
\begin{align*}
\|R_1\|_{H^1_q(\Omega)} 
& \leq C(\|\rho_0\|_{H^1_q(\Omega)} 
+ \|\omega\|_{H^1_q(\Omega)})
\int^t_0\|\bw(\cdot, s)\|_{H^2_q(\Omega)}\,ds
\|\bw(\cdot, t)\|_{H^2_q(\Omega)}\\
& \leq C(L+M)T^{1/{p'}}\|\bw\|_{L_p((0, T), H^2_q(\Omega))}
\|\bw(\cdot, t)\|_{H^2_q(\Omega)},
\end{align*}
and so, we have
$$\|R_1\|_{L_p((0, T), H^1_q(\Omega))}
\leq C(L+M)M^2T^{1/{p'}}.
$$
Thus, we obtain
\begin{equation}\label{lagest:1}
\|f_1(U)\|_{L_p((0, T), H^1_q(\Omega))} \leq 
C(M^2 + (L+M)M^2)T^{1/{p'}}.
\end{equation}
Next, we consider $\bff_2(U)$. By  \eqref{5.4*}, we have
\begin{align*}
\|\int^t_0\nabla\bw(\cdot, s)\,ds \nabla^2 \bw(\cdot, t)\|_{L_q(\Omega)}
&\leq T\|\nabla\bw\|_{L_\infty(0, T), L_\infty(\Omega)}\|\nabla^2\bw(\cdot, 
t)\|_{L_q(\Omega)}\\
&\leq CMT\|\nabla^2\bw(\cdot, t)\|_{L_q(\Omega)},
\end{align*}
and therefore
$$\|\int^t_0\nabla\bw(\cdot, s)\,ds \nabla^2 \bw(\cdot, t)\|_{L_p((0, T),
L_q(\Omega))} \leq CTML.
$$
By H\"older's inequality and \eqref{5.4*}, we also get
\begin{align*}
\|\int^t_0\nabla^2\bw(\cdot, s)\,ds \nabla \bw(\cdot, t)\|_{L_q(\Omega)}
&\leq T^{1/{p'}}\Bigl(\int^T_0\|\nabla^2\bw(\cdot, t)\|_{L_q(\Omega)}^p
\Bigr)^{1/p}\|\nabla\bw(\cdot, t)\|_{L_\infty(\Omega)} \\
& \leq CMT^{1/{p'}}\|\bw\|_{L_p((0, T), H^2_q(\Omega))}\leq CTML.
\end{align*}
In this way, setting $\bk_\bw = \int^t_0\nabla\bw\,ds$, we have
$$
\|(A_{2\Delta}(\bk_\bw)\nabla^2\bw, A_{1\Delta}(\bk_\bw)\nabla\bw,
A_{2\dv}(\bk_\bw)\nabla^2\bw, A_{1\dv}(\bk_\bw)\nabla\bw)
\|_{L_p((0, T), L_q(\Omega))} \leq CTLM.
$$
By \eqref{5.4},\eqref{5.5}, \eqref{5.7}, and Sobolev's inequality we obtain
\begin{align*}
\|\frac{\rho_0+\omega}{\Sigma_\rho}\bV^0(\bk_\bw)
\nabla(\rho_0+\omega)\|_{L_q(\Omega)}
&\leq C\int^T_0\|\nabla\bw(\cdot, s)\|_{H^1_q(\Omega)}\,ds
(\|\nabla\rho_0\|_{L_q(\Omega)} + \|\nabla\omega(\cdot, t)\|_{L_q(\Omega)}\\
&\leq CT^{1/{p'}}\|\bw\|_{L_p((0, T), H^2_q(\Omega))}
(L + \|\nabla\omega(\cdot, t)\|_{L_q(\Omega)})\leq CT(L+M)M.
\end{align*}
Analogously, \eqref{5.5}, \eqref{5.7} and Sobolev's inequality give
\begin{align*}
\|\frac{(m_1-m_2)\rho_1\rho_2}{\Sigma_\rho}\bV^0(\bk_\bw)
\nabla\varphi\|_{L_q(\Omega)}
&\leq CT^{1/{p'}}\|\bw\|_{L_p((0, T), H^2_q(\Omega))}
\|\nabla\varphi(\cdot, t)\|_{L_q(\Omega)} \leq CTM^2.
\end{align*}
Putting  the estimates above and the
estimates obtained in \eqref{5.15} together
gives 
\begin{equation}\label{lagest:2}\begin{split}
\|\bff_2(U)\|_{L_p((0, T), L_q(\Omega))} 
&\leq C\Big\{(LM +M^2+L^2)T  + (M+L)^3T^{(\theta/{p'} + 1/p)} \\
&\quad+(M+L)T^{1/p} + L^2T^{1/p}+(M+L)LT^{(\theta/{p'} + 1/p)}\Big\}.
\end{split}\end{equation}
Next, we consider $R_3$ defined in \eqref{lag:8}
replacing $\vartheta$ and $\bv$ by $\varphi$ and $\bw$. 
By \eqref{5.14},\eqref{5.4*}, Sobolev's inequality and by H\"older's inequality 
\eqh{
\|\nabla\Bigl(\frac{\rho_1\rho_2}{\fp\rho}\Bigr) \lr{2\bV^0(\bk_\bw)+(\bV^0(\bk_\bw))^2}
\nabla\varphi\|_{L_q(\Omega)}
\leq C(M+L)\int^T_0\|\bw(\cdot, s)\|_{H^2_q(\Omega)}\,ds
\|\nabla\varphi(\cdot, t)\|_{L_q(\Omega)}\\
\leq C(M+L)M^2T.
}
Other terms in $R_3$ can be estimated in the similar manner to
the estimate of $R_2$, hence we obtain
$$\|R_3\|_{L_p((0, T), L_q(\Omega))}
\leq C(M+L)LT,$$
which, combined with the estimates obtained in 
\eqref{5.15}, leads to 
\begin{equation}\label{lagest:3}\begin{split}
\|f_3(U)\|_{L_p((0, T), L_q(\Omega))}
&\leq C((LM +M^2+L^2)T + M(M+L)T^{\theta/{p'}} \\
&\quad+ (M+L)^2T^{(\theta/{p'} + 1/p)} 
+ M(M+L)T^{\theta/{p'}} + (M+L)^2T^{1/p}).
\end{split}\end{equation}
Finally, we estimate $R_4$ defined in \eqref{lag:9}
replacing $\bv$ and $\vartheta$ by $\bw$ and 
$\varphi$. For this purpose, we have to 
extend $R_4$ to the whole time interval $\BR$. Let $e_T$ be
the extension operator defined in \eqref{ext:1}. Let
$\tilde h_0$ be a function in $B^{2(1-1/p)}_{q,p}(\BR^N)$
such that $\tilde h_0 = h_0$ in $\Omega$ and 
$$\|\tilde
h_0\|_{B^{2(1-1/p)}_{q,p}(\BR^N)} 
\leq C\|h_0\|_{B^{2(1-1/p)}_{q,p}(\Omega)}.$$
Let 
$$T(t)h_0 = e^{(\Delta-2)t}\tilde h_0
= \CF^{-1}[e^{-(|\xi|^2+2)t}\CF[\tilde h_0](\xi)]
$$
where $\CF$ and $\CF^{-1}$ denote the Fourier transform
on $\BR^N$ and its inverse transform. We know that
\begin{equation}\label{heat:1}
\|e^tT(\cdot)h\|_{L_p((0, \infty), H^2_q(\BR^N))}
+ \|e^t\pd_tT(\cdot)h\|_{L_p((0, \infty), L_q(\BR^N))}
\leq C\|h\|_{B^{2(1-1/p)}_{q,p}(\Omega)}.
\end{equation}
Let $\psi(t) \in C^\infty(\BR)$ be one for $t > -1$ and 
zero for $t < -2$. Since $\omega|_{t=0} - T(t)h|_{t=0}
= h-h = 0$ in $\Omega$, we set  
$$\tilde e_T[\omega] = e_T[\omega - T(\cdot)h] + 
\psi(t)T(|t|)h.
$$
Then, by \eqref{ext:1}, \eqref{ext:2} and 
\eqref{heat:1}, we have
\begin{equation}\label{ext:3}
\|e^{-\gamma t}\tilde e_T[\omega]\|_{L_p(\BR, H^2_q(\Omega))}
+ \|e^{-\gamma t}\pd_t \tilde e_T[\omega]
\|_{L_p(\BR, L_q(\Omega))} 
\leq C(e^{2\gamma}L+M)
\end{equation}
for any $\gamma \geq 0$, where $C$ is a constant independent of $\gamma$, $T$, $L$, and 
$M$. To treat $R_4$, setting 
$$\CR_\bw = -\Big\{\bn\lr{y+\int^t_0\bw(y, s)\,ds}\bV^0(\bk_\bw)
+ \int^1_0(\nabla\bn)\lr{y + \tau\int^t_0\bw(y, s)\,ds}\,d\tau
\int^t_0\bw(y, s)\,ds\Big\},$$
we write it as $R_4 = \CR_\bw\nabla\varphi$.  Here, we may assume that 
$\bn$ is defined in $\BR^N$ and $\|\bn\|_{H^2_\infty(\BR^N)} \leq C$.  
Notice that $\CR_\bw|_{t=0} = 0$.  
We then define $\tilde R_4$ by 
$$\tilde R_4 = e_T[\CR_\bw]\nabla(\tilde e_T[\varphi]).$$
$\tilde R_4$ is an extension of $R_4$ to the whole time interval
$\BR$.  Obviously, $\tilde R_4=R_4$ in $(0, T)$.

To estimate $\tilde R_4$, we use the following 
lemma due to Shibata and Shimizu \cite{SS1}.
\begin{lem}\label{lem:5.1}
Let $1 < p < \infty$, $3 < q < \infty$ and $0 < T \leq 1$.  Assume that 
$\Omega$ is a uniformly $C^2$ domain. Let 
\begin{align*}
f \in H^1_\infty(\BR, L_q(\Omega)) \cap L_\infty(\BR, H^1_q(\Omega)), 
\quad
g \in L_p(\BR, H^1_q(\Omega)) \cap H^{1/2}_p(\BR, L_q(\Omega)).
\end{align*} 
If we assume that $f \in L_p(\BR, H^1_q(\Omega))$ and 
that $f$ vanishes for $t \in [0, 2T]$ in addition, then we have
\begin{align*}
&\|fg\|_{L_p(\BR, H^1_q(\Omega))} + \|fg\|_{H^{1/2}_p(\BR, L_q(\Omega))}\\
&\quad\leq C(\|f\|_{L_\infty(\BR, H^1_q(\Omega))}
+T^{(q-3)/(pq)}\|\pd_tf\|_{L_\infty(\BR, L_q(\Omega))}^{(1-3/(2q))}
\|\pd_tf\|_{L_p((\BR, H^1_q(\Omega))}^{3/(2q)})
(\|g\|_{_p(\BR, H^1_q(\Omega))} + \|g\|_{H^{1/2}_p(\BR, L_q(\Omega))}).
\end{align*}
\end{lem}
\begin{remark} \thetag1~ The boundary of $\Omega$ 
was assumed to be bounded in Shibata-Shimizu \cite{SS1}. 
But, Lemma \ref{lem:5.1} can be proved with the help of Sobolev's 
inequality and complex interpolation theorem, and so 
employing the same argument as that in the proof of
\cite[Lemma 2.7]{SS1}, 
we can prove Lemma \ref{lem:5.1}. \\
\thetag2~ By Sobolev's inequality, $\|fg\|_{H^1_q(\Omega)}
 \leq C\|f\|_{H^1_q(\Omega)}\|g\|_{L_q(\Omega)}$, and so 
the essential part of Lemma \ref{lem:5.1} is the estimate of
$\|fg\|_{H^{1/2}_p(\BR, L_q(\Omega))}$. 
\end{remark}
Since $\Omega$ is a uniformly $C^3$ domain, we may assume that 
$\bn$ is defined on the whole $\BR^N$ and $\|\bn\|_{H^2_\infty(\BR^N)}
< \infty$.  We then have
\begin{align*}
\|e_T[\CR_\bw](\cdot, t)\|_{H^1_q(\Omega)}
\leq C\Bigl\{\int^T_0\|\bw(\cdot, s)\|_{H^2_q(\Omega)}\,ds
+ \Bigl(\int^T_0\|\bw(\cdot, s)\|_{H^1_q(\Omega)}\,ds\Bigr)^2
\Bigr\} 
\leq C(T^{1/{p'}}M + T^{2/{p'}}M^2),
\end{align*}
and so
\begin{equation}
\label{halfest:1}
\|e_T[\CR_\bw]\|_{L_\infty(\BR, H^1_q(\Omega))} 
\leq C(T^{1/{p'}}M + T^{2/{p'}}M^2). 
\end{equation}
Choosing $T > $ so small that $T^{1/{p'}}M \leq 1$,
by \eqref{ext:2} we have
$$\|\pd_t e_T[\CR_\bw](\cdot, t)\|_{H^1_q(\Omega)}
\leq C\begin{cases} 0 \quad&\text{for $t < 0$}, \\
\|\bw(\cdot, t)\|_{H^2_q(\Omega)} 
\quad &\text{for $0 < t < T$}, \\
\|\bw(\cdot, 2T-t)\|_{H^2_q(\Omega)} 
\quad &\text{for $T < t < 2T$}, \\
0 \quad &\text{for $t > 2T$}, \end{cases}
$$
and therefore
\begin{equation}\label{halfest:2}
\|\pd_te_T[\CR_\bw]\|_{L_p(\BR, H^1_q(\Omega))}
\leq C\|\bw\|_{L_p((0, T), H^2_q(\Omega))}
\leq CM.
\end{equation}
To estimate $\nabla(\tilde e_T[\varphi])$, we use the following lemma.
\begin{lem}\label{lem:5.2} Let $1 < p, q < \infty$. Assume that 
$\Omega$ is a uniform $C^2$ domain.  Then
$$H^1_p(\BR, L_q(\Omega)) \cap L_p(\BR, H^2_q(\Omega))
\subset H^{1/2}_p(\BR, H^1_q(\Omega)), $$
and 
$$\|\nabla u\|_{H^{1/2}_p(\BR, L_q(\Omega))}
\leq C(\|u\|_{L_p(\BR, H^2_q(\Omega))} 
+ \|\pd_tu\|_{L_p(\BR, L_q(\Omega))}).
$$
\end{lem}
\begin{remark}
As was mentioned in Shibata and Shimizu \cite{SS2}, 
in the case that $\Omega= \BR^N$, Lemma \ref{lem:5.2} can
be proved by Weis's operator valued Fourier multiplier
theorem. In the uniformly $C^2$ domain case,
localizing the estimate, using the uniformity
of the domain and the partition of unity, we can prove
Lemma \ref{lem:5.2}.  The detailed proof was given in
Shibata \cite{S17}.  In the case that
$p=q$ and $\Omega$ is bounded, Lemma \ref{lem:5.2}
was proved by M.~Meyries and R.~Schnaubelt
\cite{MS12}. 
\end{remark}
Applying Lemma \ref{lem:5.2} and using \eqref{heat:1},
 we have
\eq{
&\|e^{-\gamma t}\nabla\tilde e_T[\varphi]\|_{H^{1/2}_p(\BR, L_q(\Omega))}
+
\|e^{-\gamma t}\nabla\tilde e_T[\varphi]\|_{L_p(\BR, H^1_q(\Omega))}\\
&\quad\leq C(\|e^{-\gamma t}\tilde e_T[\varphi]\|_{H^1_p(\BR, L_q(\Omega))}
+\|e^{-\gamma t}\tilde e_T[\varphi]\|_{L_p(\BR, H^2_q(\Omega))} \\
&\quad \leq C(\|\varphi\|_{L_p((0, T), H^2_q(\Omega))}
+ \|\varphi\|_{H^1_p((0, T), L_q(\Omega))}
+ e^{2\gamma}L + M) 
\leq C(e^{2\gamma }L + M) \label{halfest:3}
}
for any $\gamma > 0$. Since 
$e_T[\CR_\bw] = 0$ for $t \not\in (0, 2T)$,
 applying Lemma \ref{lem:5.1} to $\tilde R_4$
and using two estimates \eqref{halfest:1},
\eqref{halfest:2} and \eqref{halfest:3},  we have
\begin{equation}\label{lagest:4}
\|e^{-\gamma t}\tilde R_4\|_{L_p(\BR, H^1_q(\Omega))}
+ \|e^{-\gamma t}\tilde R_4\|_{H^{1/2}_p(\BR, L_q(\Omega))}
\leq C(T^{1/{p'}}M + T^{(q-3)/(pq)}M)(e^{2\gamma}L + M)
\end{equation}
for any $\gamma > 0$.

Applying Theorem \ref{thm:linear:1} to Eq. \eqref{neweq:3},
using \eqref{lagest:1}, \eqref{lagest:2},
\eqref{lagest:3}, and \eqref{lagest:4}, noting that
$0 < T \leq 1$, 
and fixing $\gamma>0$ a large positive number,  we see that 
there exists three positive constants $C$ and $C_{M,L, \gamma}$ and $\tau$
for which 
\begin{equation}\label{est:5.1}
[\zeta, \bv, \vartheta]_T \leq Ce^{2\gamma T}(L+T^\tau
C_{M,L,\gamma}).
\end{equation}
Here, $C_{M,L,\gamma}$ is a constant depending on 
$L$, $M$, and $\gamma$. Letting  $M = 2Ce^{2\gamma}L$ and choosing
$T>0$ so small that $T^\tau C_{M,L,\gamma} \leq L$, we have 
\begin{equation}\label{fix:1}
[\zeta, \bv, \vartheta]_T \leq M.
\end{equation}
Let $\bS$ be a map acting on $U = (\omega, \bw, \varphi)
\in \CH_{T, M}$ defined by $\bS U =V$, where  $V = (\zeta, \bv, \vartheta)$
 is a unique solution of Eq. \eqref{linear:1}, 
and then by \eqref{fix:1}
we see that $\bS$ maps $\CH_{T,M}$ into itself. Let 
$U_1$, $U_2 \in \CH_{T, M}$, and then applying the 
same argument as that in the proof of \eqref{est:5.1}
to $V_1-V_2$ with $V_i = \bS U_i$, 
we see that there exists a constant $K$ depending on
$M$ and $L$  for which
\begin{equation}\label{est:5.2}
[\bS U_1-\bS U_2]_T \leq KT^\tau[U_1-U_2]_T.
\end{equation}
Here,  $(U_1-U_2)|_{t=0} = 0$, and so constructing the extension 
of the term corresponding to $R_4$ in the previous argument
we can use $e_T[\varphi_1-\varphi_2]$ instead of $\tilde e_T[\varphi_1
-\varphi_2]$.  Namely, we de not need to use the operator $T(\cdot)$,
and so $\gamma$ does not appear in the estimate, because
 $e_T[\varphi_1-\varphi_2]$ vanishes for $t\not\in (0, 2T)$. 

From \eqref{est:5.2} we see that $\bS$ is a contraction map from
$\CH_{T,M}$ into itself, and so by the Banach fixed point theorem
there exists a unique $V = (\zeta, \bv, \vartheta) \in
\CH_{T,M}$ with $M = 2CL$ such that $V = \bS V$.  This $V$
is a unique solution of Eq. \eqref{neweq:3}, which completes the proof
of Theorem \ref{thm:main1}. 

Employing the same argument as that in the proof of Theorem \ref{thm:main1}
we can prove the following theorem, which is so called almost global
existence theorem and used to prove the global well-posedness.
\begin{thm}\label{cor:1}
Let $2 < p < \infty$, $3 < q < \infty$ and $T > 0$. Assume that 
$2/p + 3/q < 1$ and that $\Omega$ is a uniform $C^3$ domain in
$\BR^N$ ($N \geq 2$).  Let $\rho_{10}(x)$, $\rho_{20}(x)$, and 
$\bu_0(x)$ be initial data for Eq. \eqref{mf:1}. Assume that there exist
positive numbers $a_1$ and $a_2$ for which
\begin{equation}\label{cor-initial:0}
a_1 \leq \rho_{10}(x), \enskip \rho_{20}(x) \leq a_2
\quad\text{for any $x \in \overline{\Omega}$}.
\end{equation}
Let $(h_0(x), \rho_0(x)) = \Psi(\rho_{10}(x), \rho_{20}(x))$. 
 Then, there exists a small constant $\epsilon_0 >0$ 
depending on $a_1$, $a_2$ and $T$ such that if 
$\rho_{10}$, $\rho_{20}$, $\bu_0$ and $h_0$ satisfy the condition:
\begin{equation}\label{cor-initial:1}
\|\nabla(\rho_{10}, \rho_{20})\|_{L_q(\Omega)}
+ \|\bu_0\|_{B^{2(1-1/p)}_{q,p}(\Omega)} 
+ \|h_0\|_{B^{2(1-1/p)}_{q,p}(\Omega)}
\leq \epsilon_0
\end{equation}
and the compatibility condition:
\begin{equation}\label{cor-initial:3}
\bu_0|_\Gamma=0, \quad (\nabla h_0)\cdot\bn|_\Gamma = 0,
\end{equation}
then problem \eqref{neweq:2} admits a unique solution 
$(\eta, \bv, \vartheta)$ with
\begin{gather*}
\eta - \rho_0 \in H^1_p((0, T), H^1_q(\Omega)),
\quad \bv \in H^1_p((0, T), L_q(\Omega)^3) \cap L_p((0, T), H^2_q(\Omega)^3),\\
\vartheta \in H^1_p((0, T), L_q(\Omega)) \cap L_p((0, T), H^2_q(\Omega))
\end{gather*}
possessing the estimates:
\begin{gather*}
\|\eta-\rho_0\|_{H^1_p((0, T), H^1_q(\Omega))}
+ \|\pd_t(\bv, \vartheta)\|_{L_p((0, T), L_q(\Omega))}
+ \|(\bv, \vartheta)\|_{L_p((0, T), H^2_q(\Omega))}
\leq C\epsilon_0, \\
a_1 \leq \rho(x,t) \leq 2a_2+a_1
\quad\text{for $(x, t) \in \Omega\times(0, T)$}, \quad 
 \int^T_0\|\nabla\bv(\cdot, s)\|_{L_\infty(\Omega)}
\leq \delta. 
\end{gather*}
Here, $C$ is some constant independent of $\epsilon_0$. 
\end{thm}
%%%%%%%%%%%%%%%%%%%%%%
\section{Global well-posedness -- proof of Theorem \ref{thm:main2}}
\label{sec:6} 

In this section, $\Omega$ is a bounded domain whose boundary
$\Gamma$ is a compact hypersurface of $C^3$ class. 
Let $\rho_{1*}$ and $\rho_{2*}$ be any positive numbers and set
$(h_*, \rho_*) = \Psi(\rho_{1*}, \rho_{2*}) \in \BR\times\BR_+$. 
Let $T>0$ and let $(\eta, \bv, \vartheta)$ be a solution of 
Eq. \eqref{neweq:2} such that 
\begin{gather}
\eta \in H^1_p((0, T), H^1_q(\Omega)), \quad
\bv \in H^1_p((0, T), L_q(\Omega)^3) \cap L_p((0, T), H^2_q(\Omega)^3), 
\nonumber \\
\vartheta  \in H^1_p((0, T), L_q(\Omega)) \cap L_p((0, T), H^2_q(\Omega)),
\quad 
\int^T_0\|\nabla\bv(\cdot, s)\|_{L_\infty(\Omega)}\,ds \leq \delta
\nonumber \\
\rho_*/4 \leq \eta(x, t) \leq
4\rho_*, \quad |\vartheta(x, t)|\leq 4|h_*|
\quad\text{for $(x, t) \in \Omega\times(0, T)$}.
\label{eq:6.1}
\end{gather}
To prove the global well-posedness, we prolong $(\eta, \bv, \vartheta)$ 
to any time interval beyond $T$. 
Let $\zeta = \eta-\rho_*$ and $h= \vartheta-h_*$, and let 
\begin{align*}
\CI &= \|\rho_0-\rho_*\|_{H^1_q(\Omega)}
+ \|(\bu_0, h_0-h_*)\|_{B^{2(1-1/p)}_{q,p}}, \\
<e^{\gamma t}V>_T & 
= \|e^{\gamma t}\nabla\zeta\|_{L_p((0, T), L_q(\Omega))}
+ \|e^{\gamma t}\pd_t\zeta\|_{L_p((0, T), H^1_q(\Omega))}
+ \|e^{\gamma t}\bv\|_{L_p((0, T), H^2_q(\Omega)}\\
&
\quad+ \|e^{\gamma t}\nabla h\|_{L_p((0, T), H^1_q(\Omega))}
+\|e^{\gamma t}\pd_t(\bv, h)\|_{L_p((0, T), L_q(\Omega))}.
\end{align*}
Here, $\gamma$ is a positive constant appearing in Theorem \ref{thm:decay1}
below. 
The key step is to prove the estimate:
\begin{equation}\label{eq:6.2}
<e^{\gamma t}V>_T \leq C(\CI + <e^{\gamma t}V>_T^2)
\end{equation}
for some constant $C > 0$.

To prove \eqref{eq:6.2},  we linearize
Eq. \eqref{neweq:2} at $(\rho_1,\rho_2) = (\rho_{1*},\rho_{2*})$, $\eta=\rho_*$,
$\bv=0$ and $\vartheta=h_*$. Namely, $\eta = \rho_*+\zeta$, 
$\bv$ and $\vartheta = h_*+h$ satisfy the following equations:
\begin{equation}\label{neweq:5}\left\{\begin{aligned}
\pd_t\zeta + a_{0*}\dv \bv = \tilde f_1(U)
\qquad&\text{in $\Omega\times(0, T)$}, \\
a_{0*}\pd_t\bv - \mu\Delta\bv - \nu\nabla\dv\bv
+a_{1*}\nabla \zeta + a_{2*}\nabla h = \tilde\bff_2(U)\qquad
&\text{in $\Omega\times(0, T)$}, \\
a_{3*}\pd_t h + a_{2*} \dv\bv - a_{4*}\Delta h = \tilde f_3(U)\qquad
&\text{in $\Omega\times(0, T)$}, \\
\bv=0, \quad (\nabla h)\cdot\bn=g(U)\qquad
&\text{on $\Gamma\times(0, T)$}, \\
(\zeta, \bv, h)|_{t=0}=(\rho_0-\rho_*, \bu_0, h_0-h_*)
\qquad&\text{in $\Omega$}.
\end{aligned}\right.\end{equation}
Here, we have set
\begin{align*}
a_{0*} &= \rho_*, \quad a_{1*} = \frac{a_0}{\Sigma_{\rho_*}}, \quad
a_{2*} = \frac{(m_1-m_2)\rho_{1*}\rho_{2*}}{\Sigma_{\rho_*}}, \quad
a_{3*} = \frac{m_1m_2\rho_{1*}\rho_{2*}}{\Sigma_{\rho_*}}, \quad 
a_{4*} = \frac{\rho_{1*}\rho_{2*}}{\fp_*\rho_*}, \\
\Sigma_{\rho_*} &= m_1\rho_{1*} + m_2\rho_{2*},
\quad \fp_* = \frac{\rho_{1*}}{m_1} + \frac{\rho_{2*}}{m_2}, 
\quad U = (\eta, \bv, \vartheta)=(\rho_*+\zeta, \bv, h_*+h), \\
\tilde f_1(U) &= R_1(U) - \zeta\dv\bv, \\
\tilde\bff_2(U) & = R_2(U)- \zeta\pd_t\bv 
- \Bigl(\frac{\eta}{\Sigma_\rho}
 - \frac{\rho_*}{\Sigma_{\rho_*}}\Bigr)\nabla\zeta
-(m_1-m_2)\Bigl(\frac{\rho_1\rho_2}{\Sigma_\rho}
-\frac{\rho_{1*}\rho_{2*}}{\Sigma_{\rho_*}}\Bigr)\nabla h, \\
\tilde f_3(U)& = R_3(U) - m_1m_2\Bigl(\frac{\rho_1\rho_2}{\Sigma_\rho}
-\frac{\rho_{1*}\rho_{2*}}{\Sigma_{\rho_*}}\Bigr)\nabla h
-(m_1-m_2)\Bigl(\frac{\rho_1\rho_2}{\Sigma_\rho}
-\frac{\rho_{1*}\rho_{2*}}{\Sigma_{\rho_*}}\Bigr)\dv\bv \\
&\quad + \dv\Bigl(\Bigl(\frac{\rho_1\rho_2}{\fp\rho}-
\frac{\rho_{1*}\rho_{2*}}{\fp_*\rho_*}\Bigr)\nabla h\Bigr), \\
g(U) &= R_4(U).
\end{align*}
Notice that $a_{0*}$, $a_{1*}$, $a_{3*}$ and $a_{4*}$
are positive constants, while $a_{2*}$ is a real number. 
We consider the system of linear equations:
\begin{equation}\label{newlinear:1}\left\{\begin{aligned}
\pd_t\zeta + a_{0*}\dv \bv = g_1
\qquad&\text{in $\Omega\times(0, T)$}, \\
a_{0*}\pd_t\bv - \mu\Delta\bv - \nu\nabla\dv\bv
+a_{1*}\nabla \zeta + a_{2*}\nabla \vartheta = \bg_2\qquad
&\text{in $\Omega\times(0, T)$}, \\
a_{3*}\pd_t \vartheta + a_{2*} \dv\bv - a_{4*}\Delta\vartheta = g_3\qquad
&\text{in $\Omega\times(0, T)$}, \\
\bv=0, \quad (\nabla \vartheta)\cdot\bn=g_4\qquad
&\text{on $\Gamma\times(0, T)$}, \\
(\zeta, \bv, \vartheta)|_{t=0}=(\zeta_0, \bv_0, \vartheta_0)
\qquad&\text{in $\Omega$}.
\end{aligned}\right.\end{equation}
For Eq. \eqref{newlinear:1}, we have the following decay 
theorem.
\begin{thm}\label{thm:decay1}
Let $1 < p, q < \infty$, $2/p+1/q \not=1$ and $2/p + 1/q \not=2$.
Assume that $\Omega$ is a bounded domain whose boundary
$\Gamma$ is a compact hypersurface of $C^3$ class. 
Let 
\begin{align*}
\rho_0 &\in H^1_q(\Omega), \quad
\bv_0 \in B^{2(1-1/p)}_{q,p}(\Omega)^3, \quad 
\vartheta_0 \in B^{2(1-1/p)}_{q,p}(\Omega), \\
g_1 &\in L_p((0, T), H^1_q(\Omega)), \quad
\bg_2 \in L_p((0, T), L_q(\Omega)^3)
\cap H^1_p((0, T), L_q(\Omega)^3), \\
g_3 & \in L_p((0, T), L_q(\Omega)), \quad
E[e^{\gamma_1 t}g_4] \in H^{1/2}_p(\BR, L_q(\Omega))
\cap L_p(\BR, H^1_q(\Omega))
\end{align*}
for some $\gamma_1 > 0$. Here, $E[e^{\gamma_1 t}g_4]$ denotes some
extension of $e^{\gamma_1 t}g_4$ to the whole time interval $\BR$.
  Assume that $\bv_0$, $\vartheta_0$
and $g_4$ 
satisfy the compatibility conditions:
$$\bv_0=0  \enskip \text{on $\Gamma$ for $2/p + 1/q < 2$}, \quad
(\nabla\vartheta_0)\cdot\bn = g_4|_{t=0}
\enskip \text{on $\Gamma$ for $2/p + 1/q < 1$}.$$
Then, problem \eqref{newlinear:1} admits unique solutions
$\eta$, $\bv$, and $\vartheta$ with 
\begin{gather*}
\eta \in H^1_p((0, T), H^1_q(\Omega)), \quad
\bv \in L_p((0, T), H^2_q(\Omega)^3)
\cap H^1_p((0, T), L_q(\Omega)^3), \\
\vartheta \in L_p((0, T), H^2_q(\Omega))
\cap H^1_p((0, T), L_q(\Omega))
\end{gather*}
possessing the estimate:
\begin{align*}
&\|e^{\gamma t}\nabla\eta\|_{L_p((0, T), L_q(\Omega))}
+ \|e^{\gamma t}\pd_t\eta\|_{L_p((0, T), H^1_q(\Omega))}
+ \|e^{\gamma t}\bv\|_{L_p((0, T), H^2_q(\Omega)} \\
&\quad + \|e^{\gamma t}\nabla \vartheta\|_{L_p((0, T), H^1_q(\Omega))}
+\|e^{\gamma t}\pd_t(\bv, \vartheta)\|_{L_p((0, T), L_q(\Omega))} \\
&\leq C(\|\zeta_0\|_{H^1_q(\Omega)} 
+ \|(\bv_0, \vartheta_0)\|_{B^{2(1-1/p)}_{q,p}(\Omega)}
+ \|e^{\gamma t}(g_1, \bg_2, g_3)\|_{L_p((0, T), L_q(\Omega))} \\
&\quad + \|E[e^{\gamma_{_1} t}g_4]\|_{H^{1/2}_p(\BR, L_q(\Omega))}
+ \|E[e^{\gamma_{_1} t}g_4]\|_{L_p(\BR, H^1_q(\Omega))})\end{align*}
for some constants $\gamma \in (0, \gamma_1]$ and $C > 0$. 
\end{thm}
Postponing the proof of Theorem \ref{thm:decay1}, we prove \eqref{eq:6.2}. 
Let $(\rho_1(x), \rho_2(x)) = \Phi(\vartheta, \eta)$. 
Following the ideas from \cite{SSchade},
we first prove that 
\begin{equation}\label{eq:6.3}\begin{split}
\|\eta-\rho_0\|_{L_\infty((0, T), H^1_q(\Omega))} 
&\leq C<e^{\gamma t}V>_T, \\
\|\vartheta-h_0\|_{L_\infty((0, T), B^{2(1-1/p)}_{q,p}(\Omega))}& 
\leq C(\CI + <e^{\gamma t}V>_T).
\end{split}\end{equation}
where $(h_0, \rho_0) = (\vartheta, \eta)|_{t=0}$ (cf. \eqref{neweq:2}
in Introduction).
In fact, by H\"older's inequality we have 
\begin{align*}
\|\eta(\cdot, t)-\eta(\cdot, 0)\|_{H^1_q(\Omega)}
&\leq \int^T_0\|\pd_t\eta(\cdot, t)\|_{H^1_q(\Omega)}\,dt \\
&\leq \Bigl(\int^T_0e^{-p'\gamma t}\,dt\Bigr)^{1/{p'}}
\Bigl(\int^T_0(e^{\gamma t}\|\pd_t\eta(\cdot, t)\|_{H^1_q(\Omega)})^p\,dt
\Bigr)^{1/p} 
%\\
%&
\leq C<e^{\gamma t}V>_T.
\end{align*}
Recalling that $\vartheta-h_* = h$ and $\vartheta_0-h_*
=h_0-h_*$,  we have 
\begin{align*}
\|\vartheta(\cdot, t) - \vartheta_0\|_{L_q(\Omega)}
\leq \int^T_0\|\pd_sh(\cdot,s)\|_{L_q(\Omega)}\,ds
+ \|h_0-h_*\|_{L_q(\Omega)}
\leq C(<e^{\gamma t}V>_T + \CI). 
\end{align*}
Let $H(x,t) = h(x, t) - |\Omega|^{-1}\int_\Omega h(x, t)\,dx$.
Since $\int_\Omega H(x, t)\,dx = 0$, by Poincar\'e's inequality we have
$$\|H(\cdot, t)\|_{H^2_q(\Omega)} \leq C\|\nabla H(\cdot, t)\|_{H^1_q(\Omega)}
= C\|\nabla h(\cdot, t)\|_{H^1_q(\Omega)}.
$$
Moreover, noting that $2(1-1/p) > 1$, we have 
\begin{align*}
\|H|_{t=0}\|_{B^{2(1-1/p)}_{q,p}(\Omega)}
&
\leq \|H|_{t=0}\|_{L_q(\Omega)} + \|\nabla H|_{t=0}\|_{B^{1-2/p}_{q,p}(\Omega)}
\leq C(\|h_0-h_*\|_{L_q(\Omega)} + \|\nabla h_0\|_{B^{1-2/p}_{q,p}(\Omega)}) \\
&= C\|h_0-h_*\|_{B^{2(1-1/p)}_{q,p}(\Omega)}.
\end{align*}
On the other hand, employing the same argument as that in the proof of 
\eqref{5.4}, by real interpolation theory, we have
\begin{align*}
\sup_{t \in (0, T)} \|H(\cdot, t)\|_{B^{2(1-1/p)}_{q,p}(\Omega)}
&\leq \sup_{t \in (0, T)} \|\tilde e_T[H](\cdot, t)
\|_{B^{2(1-1/p)}_{q,p}(\Omega)} \\
&\leq C(\|H\|_{L_p((0, T), H^2_q(\Omega))} + \|\pd_tH\|_{L_p((0, T), L_q(\Omega))}+\|H|_{t=0}\|_{B^{2(1-1/p)}_{q,p}(\Omega)}).
\end{align*}
Therefore, since
$\|\pd_tH\|_{L_q(\Omega)} \leq C\|\pd_t h\|_{L_q(\Omega)},$
we obtain
$$\sup_{t \in (0, T)} \|H(\cdot, t)\|_{B^{2(1-1/p)}_{q,p}(\Omega)}
\leq C(\|\nabla h\|_{L_p((0, T), H^1_q(\Omega))}
+ \|\pd_th\|_{L_p((0, T), L_q(\Omega))}
+ \|h_0-h_*\|_{B^{2(1-1/p)}_{q,p}(\Omega)}).
$$
Since
$$\sup_{t\in(0, T)}\|h(\cdot, t)\|_{L_q(\Omega)}
\leq \|h_0-h_*\|_{L_q(\Omega)}
+ \int^T_0\|\pd_th(\cdot, t)\|_{L_q(\Omega)}\,dt
\leq C(\|h_0-h_*\|_{L_q(\Omega)} + C<e^{\gamma t}V>_T), 
$$
we have 
\begin{align*}
\sup_{t\in (0, T)}\|h\|_{B^{2(1-1/p)}_{q,p}(\Omega)}
&\leq \sup_{t \in (0, T)}\|H\|_{B^{2(1-1/p)}_{q,p}(\Omega)}
+ \sup_{t\in(0, T)}\|h(\cdot, t)\|_{L_q(\Omega)}
\nonumber \\ 
&\leq C(\|h_0-h_*\|_{B^{2(1-1/p)}_{q,p}(\Omega)}
+ <e^{\gamma t}V>_T)
\leq C(\CI + <e^{\gamma t}V>_T),
\end{align*}
which shows the second inequality in \eqref{eq:6.3}.
Next we show that 
\begin{equation}\label{eq:6.4}
\|(\rho_1, \rho_2) - (\rho_{1*}, \rho_{2*})\|_{L_\infty((0, T),
H^1_q(\Omega))} \leq C(\CI + <e^{\gamma t}V>_T).
\end{equation}
In fact, by the Taylor formula, we have
\begin{equation*}
(\rho_1, \rho_2)- (\rho_{10}, \rho_{20})
=\Phi(\vartheta, \eta) - \Phi(h_0, \rho_0) 
\leq \int^1_0\Phi'((h_0, \rho_0)
+\theta(\vartheta-\vartheta_0, \eta-\eta_0))\,d\theta
(\vartheta-h_0, \eta-\rho_0),
\end{equation*}
where $(h_0, \rho_0) = (\vartheta, \eta)|_{t=0}$. 
Set 
$$D = \{(\zeta, \eta) \in \BR^2\mid |\zeta| \leq |h_*|/4,
\quad \rho_*/4 \leq \eta \leq 4\rho_*\},$$
and then by \eqref{eq:6.1}, $(\vartheta,\eta) \in D$ for 
any $(x, t) \in \Omega\times(0,T)$.  Let $C_0$ be a positive 
constant for which 
$$\sup_{(\vartheta, \eta)\in D}
|\Phi'(\vartheta, \eta)|\leq C_0,\quad
\sup_{(\vartheta, \eta)\in D}
|\Phi''(\vartheta, \eta)(\vartheta, \eta)| \leq C_0.
$$ 
We then have 
$$\|(\rho_1, \rho_2)- (\rho_{10}, \rho_{20})\|_{L_\infty((0, T), 
H^1_q(\Omega))} \leq 3C_0
\|(\vartheta-h_0, \eta-\rho_0)\|_{L_\infty((0, T),
H^1_q(\Omega))},
$$
which, combined with \eqref{eq:6.3}, leads to  \eqref{eq:6.4},
because $\|(\rho_{10}, \rho_{20})-(\rho_{1*}, \rho_{2*})\|_{H^1_q(\Omega)}
\leq \CI$. 

By \eqref{eq:6.1} we may assume that there exist
two positive constants $a_1$ and $a_2$ depending on 
$\rho_*$ and $h_*$ for which 
\begin{equation}\label{eq:6.5}
a_1 \leq \rho_1(x, t), \rho_2(x, t) \leq a_2
\quad\text{for any $(x, t) \in \Omega\times(0, T)$}.
\end{equation}
By \eqref{eq:6.4} and \eqref{eq:6.5}, we have
the following estimates:
\begin{align}
&\Bigl\|e^{\gamma t}\Bigl(\frac{\eta}{\Sigma_\rho} 
- \frac{\rho_*}{\Sigma_{\rho_*}}
\Bigr)\nabla\zeta\Bigr\|_{L_\infty((0, T), L_q(\Omega))}
+\Bigl\|e^{\gamma t}\Bigl(\frac{\rho_1\rho_2}{\Sigma_\rho}
-\frac{\rho_{1*}\rho_{2*}}{\Sigma_{\rho_*}}\Bigr)\nabla h
\Bigr\|_{L_\infty((0, T), L_q(\Omega))}\nonumber\\
&+\Bigl\|e^{\gamma t}\dv\Bigl(\Bigl(\frac{\rho_1\rho_2}{\fp\rho}
-\frac{\rho_{1*}\rho_{2*}}{\fp_*\rho_*}
\Bigr)\nabla h\Bigr)\Bigr\|_{L_\infty((0, T), L_q(\Omega))}
\leq C(\CI + <e^{\gamma t}V>_T)<e^{\gamma t}V>_T.
\label{non:6.1}
\end{align}
By Sobolev's inequality and \eqref{eq:6.3}, we have
$$\|\zeta\|_{H^1_q(\Omega)} 
\leq C\|\eta-\rho_0\|_{H^1_q(\Omega)}
+ \|\rho_0-\rho_*\|_{H^1_q(\Omega)}
\leq C(\CI + <e^{\gamma t}V>_T),
$$
and so  
\begin{align*}
\|e^{\gamma t}\zeta \dv\bv\|_{L_p((0, T), 
H^1_q(\Omega))}
&\leq C\|\zeta\|_{L_\infty((0, T), H^1_q(\Omega))}
\|e^{\gamma t}\nabla\bv\|_{L_p((0, T), H^1_q(\Omega))}\\
&\leq C(\CI + <e^{\gamma t}V>_T)<e^{\gamma t}V>_T.
\end{align*}
By \eqref{heat:1} and real interpolation theory, we have
\begin{equation}\label{eq:6.6}
%\begin{split}
\|\bv\|_{L_\infty((0, T), B^{2(1-1/p)}_{q,p}(\Omega))}
+\|\bv\|_{L_\infty((0, T), H^1_\infty(\Omega))}
\leq C(\CI + <e^{\gamma t}V>_T).
%\end{split}
\end{equation}
In fact, 
\begin{align*}
&\sup_{t \in (0, T)}\|\bv(\cdot, t)\|_{B^{2(1-1/p)}_{q,p}(\Omega)}
\leq \sup_{t \in (0, T)}\|\tilde e_T[\bv]
(\cdot, t)\|_{B^{2(1-1/p)}_{q,p}(\Omega)} \\
&\leq C(\|\bv\|_{L_p((0, T), H^2_q(\Omega))}
+ \|\pd_t\bv\|_{L_p((0, T), L_q(\Omega))}
+ 
\|T(\cdot)\tilde\bv_0\|_{L_p((0, \infty), H^2_q(\Omega))}
+ \|\pd_tT(\cdot)\tilde \bv_0\|_{L_p((0, \infty), L_q(\Omega))}),
\end{align*}
where $\tilde \bv_0 \in B^{2(1-1/p)}_{q,p}(\BR^3)$ equals
to $\bv_0$ in $\Omega$ and 
$\|\tilde \bv_0\|_{B^{2(1-1/p)}_{q,p}(\BR^3)}
\leq C\|\bv_0\|_{B^{2(1-1/p)}_{q,p}(\Omega)}$.
Thus, by \eqref{heat:1}, we have the estimate of the first term in \eqref{eq:6.6}. 
Since $2/p + 3/q < 1$, we have $\|\bv\|_{H^1_\infty(\Omega)}
\leq C\|\bv\|_{B^{2(1-1/p)}_{q,p}(\Omega)}$, 
which completes the proof of \eqref{eq:6.6}.

%Let $U = (\eta, \bv, \vartheta)$, 
Now and we shall estimate
$R_i(U)$. 
By Sobolev's inequality and H\"older's inequality,
we have
\begin{align*}
&\Bigl\|\int^t_0\nabla\bv(\cdot, s)\,ds
\nabla^2 f\Bigr\|_{L_q(\Omega)}
%\\
%&\quad
\leq \Bigl(\int^T_0e^{-\gamma p's}\,ds\Bigr)^{1/{p'}}
\Bigl(\int^T_0(e^{\gamma s}\|\nabla\bv(\cdot, s)\|_{L_\infty(\Omega)}
)^p\,ds\Bigr)^{1/p}\|\nabla^2 f(\cdot, t)\|_{L_q(\Omega)}\\
&\quad 
\leq C<e^{\gamma t}V>_T\|f(\cdot, t)\|_{H^2_q(\Omega)},
\end{align*}
and therefore
\begin{equation}\label{lag:6.1}
\Bigl\|e^{\gamma t}\int^t_0\nabla\bv(\cdot, s)\,ds
\nabla^2f\Bigr\|_{L_p((0, T), L_q(\Omega))}
\leq C<e^{\gamma t}V>_T
\|e^{\gamma t}f\|_{L_p((0, T), H^2_q(\Omega))}.
\end{equation}
A similar estimate of $\Bigl\|\int^t_0\nabla^2\bv(\cdot, s)\,ds \nabla f\Bigr\|_{L_q(\Omega)}$ yields 
\begin{equation}\label{lag:6.2}
\Bigl\|e^{\gamma t}\int^t_0\nabla^2\bv(\cdot, s)\,ds
\nabla f\Bigr\|_{L_p((0, T), L_q(\Omega))}
\leq C<e^{\gamma t}V>_T
\|e^{\gamma t}f\|_{L_p((0, T), H^2_q(\Omega))}.
\end{equation}
By \eqref{eq:6.1} and \eqref{lag:6.1}, we have
\begin{align*}
\|e^{\gamma t}R_1(U)\|_{L_p((0, T), L_q(\Omega))}
&\leq C\|\int^t_0\nabla\bv(\cdot, s)\,ds \nabla\bv\|_{L_p((0, T),
L_q(\Omega))} \\
&\leq C<e^{\gamma t}V>_T\|\nabla \bv(\cdot, t)\|_{L_p((0, T),
L_q(\Omega))} 
\leq 
C<e^{\gamma t}V>_T^2.
\end{align*}
Noting that $\nabla\eta = \nabla\zeta$, we have
\begin{align*}
\nabla R_1(U) = -\sum_{i,j=1}^3\Big(\nabla\zeta
V^0_{ij}(\bk_\bv)\frac{\pd v_i}{\pd x_j} +
\eta (\nabla_\bk V^{0}_{ij})(\bk_\bv) \int^t_0\nabla^2\bv
(\cdot, s)\,ds\,\frac{\pd v_i}{\pd x_j}
+ \eta V^0_{ij}(\bk_\bv)\nabla\frac{\pd v_i}{\pd x_j}\Big).
\end{align*}
Noting that $|\bk_\bv| \leq \delta$, we have 
\begin{align*}
&\|e^{\gamma t}\nabla R_1(U)\|_{L_p((0, T), L_q(\Omega))}\\
&\quad
\leq C(\|\nabla\zeta\|_{L_\infty((0, T), L_q(\Omega))}
\|e^{\gamma t}\nabla\bv\|_{L_p((0, T), L_q(\Omega))}
+ <e^{\gamma t}V>_T\|e^{\gamma t}\nabla\bv\|_{L_p((0, T), H^1_q(\Omega))}). 
\end{align*}
Since 
$$\|\nabla \zeta\|_{L_q(\Omega)} 
\leq \|\nabla(\eta -\rho_0)\|_{L_q(\Omega)}
+ \|\nabla\rho_0\|_{L_q(\Omega)},
$$
by \eqref{eq:6.3} we have
$$\|\nabla\zeta\|_{L_\infty((0, T), L_q(\Omega))}
\leq C(\CI + <e^{\gamma t}V>_T),
$$
and so, \eqref{lag:6.1} and \eqref{lag:6.2} imply
$$
\|e^{\gamma t}\nabla R_1(U)\|_{L_p((0, T), L_q(\Omega))}
\leq  C(\CI + <e^{\gamma t}V>_T) <e^{\gamma t}V>_T.
$$
Summing up, we have obtained 
\begin{equation}\label{non:6.2}
\|e^{\gamma t}R_1(U)\|_{L_p((0, T), H^1_q(\Omega))}
\leq C(\CI + <e^{\gamma t}V>_T) <e^{\gamma t}V>_T.
\end{equation}
We next consider $R_2(U)$ given in \eqref{lag:7}.  By \eqref{lag:6.1}
and \eqref{lag:6.2}, we have
$$\|e^{\gamma t}(A_{2\Delta}(\bk_\bv)\nabla^2\bv,
A_{1D}(\bk_\bv)\nabla\bv,
A_{2\dv}(\bk_\bv)\nabla^2\bv, 
A_{1\dv}(\bk_\bv)\nabla\bv)
\|_{L_p((0, T), L_q(\Omega))}
\leq C<e^{\gamma t}V>_T^2.
$$
By \eqref{eq:6.1} and \eqref{lag:6.1} we have
$$\|e^{\gamma t}(\frac{\eta}{\Sigma_\rho}
\bV^0(\bk_\bv)\nabla\zeta,
\frac{\rho_1\rho_2}{\Sigma_\rho}
\bV^0(\bk_\bv)\nabla h)\|_{L_p((0, T), L_q(\Omega))}
\leq C<e^{\gamma t}V>_T^2.
$$
Summing up, we have obtained 
\begin{equation}\label{non:6.3}
\|e^{\gamma t}R_2(U)\|_{L_p((0, T), H^1_q(\Omega))}
\leq C(\CI + <e^{\gamma t}V>_T) <e^{\gamma t}V>_T.
\end{equation}
We next consider $R_3(U)$ given in \eqref{lag:8}.
By \eqref{eq:6.1} and \eqref{eq:6.4}, we have
\begin{gather*}
a_1 \leq \rho_1(x, t), \rho_2(x, t) \leq 
a_2 \quad\text{for $(x, t) \in \Omega\times(0,T)$}, \\
\|\nabla\rho_i\|_{L_\infty((0, T), L_q(\Omega))}
\leq C(\CI + <e^{\gamma t}V>_T) \quad(i = 1,2), 
\end{gather*}
where $a_1$ and $a_2$ are some positive constants
depending on $\rho_{1*}$ and $\rho_{2*}$,
and therefore
$$\Bigl\|\nabla\Bigl(\frac{\rho_1\rho_2}{\fp\rho}\Bigr)
\Bigr\|_{L_\infty((0, T), L_q(\Omega))} \leq 
 C(\CI + <e^{\gamma t}V>_T).
$$
Thus, by \eqref{eq:6.1}, \eqref{lag:6.1}, and \eqref{lag:6.2},
we have 
\begin{equation}\label{non:6.4}
\|e^{\gamma t}R_3(U)\|_{L_p((0, T), H^1_q(\Omega))}
\leq C(\CI + <e^{\gamma t}V>_T) <e^{\gamma t}V>_T.
\end{equation}
Finally, we estimate $R_4(U)$ given  in \eqref{lag:9}. 
Similarly to Sect. \ref{sec:5}, we set 
$$\CR_\bv = -\{\bn(y+\int^t_0\bv(y, s)\,ds)\bV^0(\bk_\bv)
+ \int^1_0(\nabla\bn)(y + \tau\int^t_0\bv(y, s)\,ds)\,d\tau
\int^t_0\bv(y, s)\,ds\}.$$
Let $H(x, t) = h(x, t) - |\Omega|^{-1}\int_\Omega h(x, t)\,dx$.
Obviously, $\nabla H = \nabla h$.  Moreover, by Poincar\'e's
inequality, we have
\begin{equation}\label{poincare:1}\begin{split}
\|e^{\gamma t}H\|_{L_p((0, T), H^2_q(\Omega))} 
&+ \|e^{\gamma t}\pd_tH\|_{L_p((0, T), L_q(\Omega))}\\
&\leq C(\|e^{\gamma t}\nabla h\|_{L_p((0, T), H^1_q(\Omega))}
+ \|e^{\gamma t}\pd_th\|_{L_p((0, T), L_q(\Omega))}).
\end{split}\end{equation}
In particular,  we can write $R_4(U)$ as
$R_4(U) = \CR_\bv\nabla H$. We define the extension of
$e^{\gamma t}R_4(U)$ by 
$$E[e^{\gamma t} R_4(U)] = e_T[\CR_\bv](\nabla
\tilde e_T[e^{\gamma t}H]).
$$
To estimate $E[e^{\gamma t} R_4(U)] $, we use the following lemma.
\begin{lem}\label{lem:6.1} Let $1 < p < \infty$ and $3 < q < \infty$.
Then, the following two assertions hold.
\begin{itemize}
\item[\thetag1]~
If $f \in H^1_\infty(\BR, L_\infty(\Omega))$ and 
$g \in H^{1/2}_p(\BR, L_q(\Omega))$, then
$$\|fg\|_{H^{1/2}_p(\BR, L_q(\Omega))}
\leq C\|f\|_{H^1_\infty(\BR, L_\infty(\Omega))}
\|g\|_{H^{1/2}_p(\BR, L_q(\Omega))}.
$$
\item[\thetag2]~
If $f \in L_\infty(\BR, H^1_q(\Omega))$ and 
$g \in L_p(\BR, H^1_q(\Omega))$, then
$$\|fg\|_{L_p(\BR, H^1_q(\Omega))}
\leq C\|f\|_{L_\infty(\BR, H^1_q(\Omega))}
\|g\|_{L_p(\BR, H^1_q(\Omega))}.
$$
\end{itemize}
\end{lem}
\begin{proof}
To prove the first assertion, we use the fact that 
\begin{equation}\label{complex:1}
H^{1/2}_p(\BR, L_q(\Omega))=(L_p(\BR, L_q(\Omega)),
H^1_p(\BR, L_q(\Omega)))_{[1/2]},
\end{equation}
where $(\cdot, \cdot)_{[\theta]}$ denotes a complex interpolation
functor for $\theta\in(0,1)$.  Since
\begin{align*}
\|\pd_t(fg)\|_{L_p(\BR, L_q(\Omega))}
&\leq \|f\|_{H^1_\infty(\BR, L_\infty(\Omega))}
\|g\|_{H^1_p(\BR, L_q(\Omega))}, \\
\|fg\|_{L_p(\BR, L_q(\Omega))}
&\leq \|f\|_{L_\infty(\BR, L_\infty(\Omega))}
\|g\|_{L_p(\BR, L_q(\Omega))}, 
\end{align*}
by \eqref{complex:1} we have the first assertion. 
The second assertion follows immediately from the 
Banach algebra property of $H^1_q(\Omega)$ for $3 < q< \infty$.
\end{proof}
Recalling that $\bn$ is defined on $\BR^3$ and 
$\|\bn\|_{H^2_\infty(\BR^3)} < \infty$, 
by \eqref{eq:6.6} we have 
\begin{align*}
\|\pd_te_T[\CR_\bv]\|_{L_\infty(\BR, L_\infty(\Omega))}
\leq C\|\bv\|_{L_\infty((0, T), H^1_\infty(\Omega))}
\leq C(\CI + <e^{\gamma t}V>_T).
\end{align*}
By Sobolev's inequality and H\"older's inequality, we have
\begin{equation*}
\|e_T[\CR_\bv]\|_{L_\infty(\BR, L_\infty(\Omega))}
%&
\leq C\int^T_0\|\bv(\cdot, s)\|_{H^2_q(\Omega)}\,ds 
\leq C\Bigl(\int^T_0(e^{\gamma t}
\|\bv(\cdot, s)\|_{H^2_q(\Omega)})^p\,ds\Bigr)^{1/p}
%\\
%&
\leq C<e^{\gamma t}V>_T.
\end{equation*}
Noting that $|\bk_\bv| \leq\delta$, we also have
$$\|e_T[\CR_\bv]\|_{L_\infty(\BR, H^1_q(\Omega))}
\leq C<e^{\gamma t}V>_T.
$$
Thus, applying Lemma \ref{lem:6.1} and Lemma \ref{lem:5.2}, 
we obtain
\begin{align*}
&\|E[e^{\gamma t} R_4(U)]\|_{L_p(\BR, H^1_q(\Omega))}
+ \|E[e^{\gamma t} R_4(U)]\|_{H^{1/2}_p(\BR, L_q(\Omega))}\\
&\quad \leq C(\CI + <e^{\gamma t}V>_T)
(\|\tilde e_T[e^{\gamma t}H]\|_{L_p(\BR, H^2_q(\Omega))}
+ \|\pd_t\tilde e_T[e^{\gamma t}H]\|_{L_p(\BR, L_q(\Omega))}).
\end{align*}
Since $e^{\gamma t}H|_{t=0} = H|_{t=0}$, we have
$$\tilde e_T[e^{\gamma t}H] = e_T[e^{\gamma t}H-
T(t)\tilde H_0] +\psi(t)T(|t|)\tilde H_0,$$
where $\tilde H_0$ is a function in $B^{2(1-1/p)}_{q,p}(\BR^N)$ 
such that 
$$\tilde H_0 = H|_{t=0} \quad\text{ in $\Omega$},
\quad \|\tilde H_0\|_{B^{2(1-1/p)}_{q,p}(\BR^N)}
\leq C\|H|_{t=0}\|_{B^{2(1-1/p)}_{q,p}(\Omega)}.
$$ 
Thus, using \eqref{heat:1} and \eqref{poincare:1}, we get
\begin{align*}
&\|\tilde e_T[e^{\gamma t}H]\|_{L_p(\BR, H^2_q(\Omega))}
+ \|\pd_t\tilde e_T[e^{\gamma t}H]\|_{L_p(\BR, L_q(\Omega))}\\
&\quad
\leq C(\|e^{\gamma t}\nabla h\|_{L_p((0, T), H^1_q(\Omega))}
+ \|e^{\gamma t}\pd_th\|_{L_p((0, T), L_q(\Omega))}
+ \|H|_{t=0}\|_{B^{2(1-1/p)}_{q,p}(\Omega)}).
\end{align*}
Finally, by Poincar\'e's inequality, we have
\begin{align*}
\|H|_{t=0}\|_{B^{2(1-1/p)}_{q,p}(\Omega)}
&= \|H|_{t=0}\|_{L_q(\Omega)} + \|\nabla (H|_{t=0})
\|_{B^{1-2/p}_{q,p}(\Omega)}
\leq C\|\nabla h_0\|_{B^{1-2/p}_{q,p}(\Omega)}.
\end{align*}
Summing up, we have obtained
\begin{equation}\label{non:6.5}
\|E[e^{\gamma t}R_4(U)]\|_{H^{1/2}_p(\BR, L_q(\Omega))}
+ \|E[e^{\gamma t}R_4(U)]\|_{L_p(\BR, H^1_q(\Omega))}
\leq C(\CI + <e^{\gamma t}V>_T)^2.
\end{equation}
Applying Theorem \ref{thm:decay1} to Eq. \eqref{neweq:5}
and using the estimates \eqref{non:6.1},
\eqref{non:6.2}, \eqref{non:6.3}, \eqref{non:6.4},
and \eqref{non:6.5}, we have
$$<e^{\gamma t}V>_T \leq C(\CI + 
(\CI + <e^{\gamma t}V>_T)^2).
$$
We assume that $\CI \leq \epsilon < 1$, and so 
$(\CI+<e^{\gamma t}V>_T)^2 \leq 2(\CI + <e^{\gamma t}V>_T^2)$,
which completes the proof of \eqref{eq:6.2}. 

We now prolong a local solution to $(0, \infty)$. Let $T > 0$
and $\eta$, $\bv$ and $\vartheta$ be solutions of Eq. 
\eqref{neweq:2} satisfying \eqref{eq:6.1}.  Then, by \eqref{eq:6.2}
we have
\begin{equation}\label{prolong:1}
<e^{\gamma s}V>_t \leq C(\CI + <e^{\gamma s}V>_t^2)
\end{equation}
for any $t \in (0, T)$, where $C$ is independent of 
$t \in (0, T)$ and $T > 0$. 
Let $r_\pm(\epsilon)$ be two roots of the quadratic equation:
$C^{-1}x = \epsilon + x^2$, that is $r_\pm(\epsilon)
= (2C)^{-1}\pm\sqrt{(2C)^{-2} - \epsilon}$. We find a small
positive number $\epsilon_1 > 0$ such that 
$$0 < r_-(\epsilon) \leq 2C\epsilon < 2C^{-1} < r_+(\epsilon)$$
for $0 < \epsilon < \epsilon_1$. 
Since $<e^{\gamma s}V>_t$ satisfies the inequality \eqref{prolong:1},
we have $<e^{\gamma s}V>_t \leq r_-(\epsilon)$ or 
$<e^{\gamma s}V>_t \geq r_+(\epsilon)$. 
 Since 
$$<e^{\gamma s}V>_t \to 0 \quad\text{as $t\to 0$},
$$
for small $t \in (0, T)$, we have $<e^{\gamma s}V>_t \leq r_-(\epsilon)$.
But, $<e^{\gamma s}V>_t$ is continuous with respect to
$t \in (0, T)$, and so 
$<e^{\gamma s}V>_t \leq r_-(\epsilon)$ for any $t \in (0, T)$.
Thus, we have 
\begin{equation}\label{prolong:2}
<e^{\gamma s}V>_T \leq 2C\epsilon.
\end{equation}
By \eqref{eq:6.3}, \eqref{eq:6.4}, and 
\eqref{eq:6.6}, we see that there exists a
constant $M > 0$ for which 
\begin{gather}
\|\eta -\rho_0\|_{L_\infty((0, T), H^1_q(\Omega))}
\leq M\epsilon, 
\quad
\|(\bv, \vartheta-h_0)\|_{L_\infty((0, T), B^{2(1-1/p)}_{q,p}(\Omega))}
\leq M\epsilon, \nonumber \\
\|(\rho_1, \rho_2) - (\rho_{1*}, \rho_{2*})
\|_{L_\infty((0, T), H^1_q(\Omega))}
 \leq M\epsilon. \label{prolong:3}
\end{gather}
Let $\eta'$, $\bv'$ and $\vartheta'$ be solutions of the 
following equations:
\begin{equation}\label{neweq:6.1}\left\{
\begin{aligned}
\pd_t\eta' + \eta\dv\bv' = R_1'(U)
\quad&\text{in $\Omega\times(T, T+T_1)$}, \\
%%%%%%%%%%%%%%%%%
\eta\pd_t\bv' - \mu\Delta\bv' - \nu\nabla\dv\bv'
+\frac{\eta}{\Sigma_{\rho'}}\nabla\eta' + 
\frac{(m_1-m_2)\rho_1'\rho_2'}{\Sigma_{\rho'}}\nabla \vartheta' 
= R_2'(U)
\quad&\text{in $\Omega\times(T, T+T_1)$}, \\
%%%%%%%%%%%%%%%%%%%%%%
\frac{m_1m_2\rho_1'\rho_2'}{\Sigma_{\rho'}}\pd_t \vartheta' + 
\frac{(m_1-m_2)\rho_1'\rho_2'}{\Sigma_{\rho'}}\dv\bv'
-\dv\Bigl(\frac{\rho_1'\rho_2'}{\fp'\rho'}\nabla \vartheta'\Bigr)
=R_3'(U)
\quad&\text{in $\Omega\times(T, T+T_1)$}, \\
\bv'=0, \quad(\nabla \vartheta)\cdot\bn  = R_4'(U)
\quad&\text{on $\Gamma \times(T, T+T_1)$}, \\
(\eta', \bv', \vartheta')|_{t=T}  =(\eta(\cdot, T), 
\bv(\cdot, T), \vartheta(\cdot, T))
\quad&\text{in $\Omega$}.
\end{aligned}\right. \end{equation}
Here, $\Sigma_{\rho'} = m_1\rho_1' + m_2\rho_2'$, 
$\fp' = \rho_1'/m_1 + \rho_2'/m_2$, and $R_i(U)$ are 
defined by replacing $\int^t_0\nabla\bv(\cdot, s)\,ds$,
$\eta$, $\rho_1$, $\rho_2$, $\rho$
$\bv$, and $\vartheta$  
 by $\int^T_0\nabla\bv(\cdot, s)\,ds 
+ \int^t_T \nabla\bv'(\cdot, s)\,ds$,
$\eta'$, $\rho_1'$, $\rho_2'$, $\rho'$
$\bv'$, and $\vartheta'$.
Employing the same argument as that in the proof of
Theorem \ref{thm:main1}, we can show that
there exists a $T_1$ depending on $\epsilon > 0$ such that
problem \eqref{neweq:6.1} admits unique solutions
$\eta'$, $\bv'$ and $\vartheta'$
with
\begin{gather}
\eta'\in H^1_p((T, T+T_1), H^1_q(\Omega)), \quad
\bv' \in H^1_p((T, T+T_1), L_q(\Omega)^3) 
\cap L_p((T, T+T_1), H^2_q(\Omega)^3), 
\nonumber \\
\vartheta'  \in H^1_p((T, T+T_1), L_q(\Omega)) 
\cap L_p((T, T+T_1), H^2_q(\Omega)),
\quad 
\int^{T+T_1}_T\|\nabla\bv'(\cdot, s)\|_{L_\infty(\Omega)}\,ds
\leq \delta,
\nonumber \\
\rho_*/4 \leq \eta'(x, t) \leq
4\rho_*, \quad |\vartheta'(x, t)|\leq 4|h_*|
\quad\text{for $(x, t) \in \Omega\times(T, T+T_1)$}.
\label{eq:6.7}
\end{gather}
Because, choosing $\epsilon > 0$ small enough, 
in view of \eqref{prolong:3} we may 
assume that 
$$\rho_{i*}/2 \leq \rho_i(x,T)  \leq 2\rho_{i*}
\quad\text{in $x \in \Omega$ for $i=1,2$}.
$$
Thus, setting 
$$f'' = \begin{cases} f&\quad\text{for $t \in (0, T)$}, \\
f' &\quad\text{for $t \in (T, T+T_1)$},
\end{cases}
$$
for $f \in \{\eta, \bv, \vartheta\}$, 
$\eta''$, $\bv''$, and $\vartheta''$ are solutions of
Eq. \eqref{neweq:2} satisfying \eqref{eq:6.1}, where $T$
is replaced by $T+T_1$.  The repeated use of this argument
implies the existence of solutions $\eta$, $\bv$, $\vartheta$
of Eq. \eqref{neweq:2} with $T = \infty$, which satisfies 
the estimate: $<e^{\gamma t}V>_\infty \leq C\epsilon$.
This completes the proof of Theorem \ref{thm:main2}. 

\section{Maximal $L_p$-$L_q$ regularity -- proof of 
Theorem \ref{thm:linear:1}}\label{sec:7} 

In this section, we consider the linear problem \eqref{linear:1}
in a uniformly  $C^2$ domain in the $N$-dimensional Euclidean
space $\BR^N$ ($N \geq 2$).  To prove Theorem \ref{thm:linear:1},
we use the $\CR$-bounded solution operators for the generalized
resolvent problem corresponding to Eq. \eqref{linear:1}. 
We first make a definition. 
\begin{dfn}\label{dfn:7.1} Let $X$ and $Y$ be two Banach spaces, and 
$\|\cdot\|_X$ and $\|\cdot\|_Y$ their norms.  
A family of operators $\CT \subset \CL(X, Y)$ is called 
$\CR$-bounded on $\CL(X, Y)$ if there exist constants $C > 0$
and $p \in [1, \infty)$ such that 
for any $n \in \BN$, $\{T_j\}_{j=1}^n \subset \CT$ and $\{f_j\}_{j=1}^n
\subset X$, the inequality 
$$\int^1_0\|\sum_{j=1}^n r_j(u)T_jf_j\|_Y^p\,du
\leq C\int^1_0\|\sum_{j=1}^nr_j(u)f_j\|_X^p\,du.
$$
Here, the Rademacher functions $r_j: [0, 1] \to \{-1,1\}$, $j \in \BN$, are
given by 
$r_j(t) = {\rm sign}(\sin (2^j\pi t))$.  The smallest such
$C$ is called $\CR$-bound of $\CT$ on 
$\CL(X, Y)$ which is written by $\CR_{\CL(X, Y)}\CT$.
\end{dfn}
The generalized resolvent problem corresponding to
Eq. \eqref{linear:1} is the following system:
\begin{equation}\label{7.1}\left\{\begin{aligned}
\lambda\zeta + \rho_0(x)\dv\bv & = f_1&\quad&\text{in $\Omega$},\\
\rho_0(x)\lambda\bv - \mu\Delta\bv - \nu\nabla\dv\bv 
+ \gamma_1(x)\nabla\zeta + \gamma_2(x)\nabla\vartheta
& = \bff_2&\quad&\text{in $\Omega$}, \\
\gamma_3(x)\lambda\vartheta + \gamma_2(x)\dv\bv - 
\dv(\gamma_4(x)\nabla\vartheta)&=f_3
&\quad&\text{in $\Omega$}, \\
\bv = 0, \quad (\nabla\vartheta)\cdot\bn
&=f_4 &\quad&\text{on $\Gamma$}.
\end{aligned}\right.\end{equation}
We assume that the coefficients $\rho_0(x)$,
$\gamma_i(x), \quad i=1,\ldots,4$ are uniformly continous
on $\overline{\Omega}$ and satisfy the conditions
\eqref{cond:4*}.  The main part of this section is
to prove the following theorem concerning the existence of $\CR$-bounded 
solution operators for Eq. \eqref{7.1}.
\begin{thm}\label{thm:7.1} Let $1 < q < \infty$ and $0 < \epsilon < \pi/2$.
Assume that $\Omega$ is a uniform $C^2$ domain. Let
\begin{align*}
X_q(\Omega) & = \{(f_1, \bff_2, f_3, f_4) \mid 
f_1, f_4 \in H^1_q(\Omega), \enskip \bff_2 \in L_q(\Omega)^N, 
\enskip f_3 \in L_q(\Omega)\}, \\
\CX_q(\Omega) & = \{(F_1, F_2, F_3, F_4, F_5) \mid F_1, F_5 \in H^1_q(\Omega),
\enskip F_3, F_4 \in L_q(\Omega), \enskip F_2 \in L_q(\Omega)^N\}.
\end{align*}
Then, there exist a positive constant $\lambda_0$ and 
operator families $\CA(\lambda) \in {\rm Hol}\,(\Sigma_{\epsilon, \lambda_0},
\CL(\CX_q(\Omega), H^1_q(\Omega)))$,  
$\CB_1(\lambda) \in {\rm Hol}\,(\Sigma_{\epsilon, \lambda_0},
\CL(\CX_q(\Omega), H^2_q(\Omega)^N))$, 
and  $\CB_2(\lambda) \in {\rm Hol}\,(\Sigma_{\epsilon, \lambda_0},
\CL(\CX_q(\Omega), H^2_q(\Omega)))$
such that
for any $(f_1, \bff_2, f_3, f_4) \in X_q(\Omega)$
and $\lambda \in \Sigma_{\epsilon, \lambda_0}$, 
$\zeta = \CA(\lambda)F_\lambda$, 
$\bv= \CB_1(\lambda)F_\lambda$, 
and $\vartheta= \CB_2(\lambda)F_\lambda$ 
are unique solutions of Eq. \eqref{7.1}, where 
$F_\lambda =  (f_1, \bff_2, f_3, \lambda^{1/2}f_4, f_4)$, and 
\begin{align*}
\CR_{\CL(\CX_q(\Omega), H^1_q(\Omega))}(\{(\tau\pd_\tau)^\ell
\CA(\lambda)\mid \lambda \in \Sigma_{\epsilon, \lambda_0}\}) &\leq r_b, \\
\CR_{\CL(\CX_q(\Omega), H^{2-j}_q(\Omega)^N)}(\{(\tau\pd_\tau)^\ell
(\lambda^{j/2}\CB_1(\lambda)) 
\mid \lambda \in \Sigma_{\epsilon, \lambda_0}\}) &
\leq r_b, \\
\CR_{\CL(\CX_q(\Omega), H^{2-j}_q(\Omega))}(\{(\tau\pd_\tau)^\ell
(\lambda^{j/2}\CB_2(\lambda)) 
\mid \lambda \in \Sigma_{\epsilon, \lambda_0}\}) &
\leq r_b
\end{align*}
for $\ell=0,1$, $j=0,1,2$, and for some constant $r_b$. 
\end{thm}
\begin{remark}
$F_1$, $F_2$, $F_3$, $F_4$ and $F_5$ are variables 
corresponding to $f_1$, $f_2$, $f_3$, $\lambda^{1/2}f_4$ and 
$f_4$. The norm of $\CX_q(\Omega)$ is defined by
$$\|(F_1, F_2, F_3, F_4, F_5)\|_{\CX_q(\Omega)}
= \|(F_1, F_5)\|_{H^1_q(\Omega)} + \|(F_2, F_3, F_4)\|_{L_q(\Omega)}.
$$
\end{remark}
Since we consider the case that $\lambda \in \Sigma_{\epsilon, \lambda_0}$ with $\lambda_0 > 0$, setting $\zeta = \lambda^{-1}(f_1-\rho_0(x)
\dv\bv)$, and inserting this formula into the second equation in
\eqref{7.1}, we rewrite it as
\begin{equation}\label{7.2} 
\rho_0(x) \lambda\bv - \mu\Delta\bv - \nu\nabla\dv\bv
-\gamma_1(x)\lambda^{-1}\nabla(\rho_0(x)\dv\bv)
+\gamma_2(x) \nabla\vartheta = \bff_2 - \lambda^{-1}\gamma_1(x)\nabla f_1.
\end{equation}
Since $\gamma_2(x)\nabla\vartheta$ and $\gamma_2(x)\dv\bv$ are lower
order terms, our main concern is to prove the existence of
$\CR$-bounded solution operators for the following two equations:
\begin{gather}
\rho_0(x)\lambda\bv - \mu\Delta\bv - \nu\nabla\dv\bv
-\gamma_1(x)\lambda^{-1}\nabla(\rho_0(x)\dv\bv)
= \bg \quad\text{in $\Omega$}, \quad
\bv|_\Gamma = 0; \label{7.3} \\
\gamma_3(x)\lambda\vartheta - \dv(\gamma_4(x)\nabla\vartheta)
= h_1 \quad\text{in $\Omega$}, \quad 
(\nabla\vartheta)\cdot\bn|_\Gamma = h_2. \label{7.4}
\end{gather}
Then, we shall prove the following theorem. 
\begin{thm}\label{thm:7.2} Let $1 < q < \infty$ and 
$0 < \epsilon < \pi/2$.  Assume that $\Omega$ is a uniform 
$C^2$ domain in $\BR^N$.  Then, there exists a positive constant
$\lambda_0$ such that the following assertions hold:
\begin{itemize}
\item[\thetag1]~There exists   
an operator family $\CC(\lambda) \in {\rm Hol}\,
(\Sigma_{\epsilon, \lambda_0}, \CL(L_q(\Omega)^N, H^2_q(\Omega)^N))
$
such that for any 
$\lambda \in \Sigma_{\epsilon, \lambda_0}$
and $\bg \in L_q(\Omega)^N$, 
$\bv = \CC(\lambda)\bg$ is a unique 
solution of Eq. \eqref{7.3}, and
$$\CR_{\CL(L_q(\Omega)^N, H^{2-j}_q(\Omega)^N)}
(\{(\tau\pd_\tau)^\ell\CC(\lambda) \mid \lambda \in 
\Sigma_{\epsilon, \lambda_0}\}) \leq r_b$$
for $\ell=0,1$ and $j=0,1,2$.
\item[\thetag2]~Let $Y_q(\Omega)$ and $\CY_q(\Omega)$ be the spaces
defined in  Notation of Sect. 1 with $G = \Omega$. 
Then, there exists 
an operator family $\CD(\lambda) \in {\rm Hol}\,
(\Sigma_{\epsilon, \lambda_0}, \CL(\CY_q(\Omega), H^2_q(\Omega))$
such that for any $\lambda \in \Sigma_{\epsilon, \lambda_0}$
and $(h_1, h_2)\in Y_q(\Omega)$, $\vartheta
= \CD(\lambda)(h_1, \lambda^{1/2}h_2, h_2)$ is a unique 
solution of Eq. \eqref{7.4}, and
$$\CR_{\CL(\CY_q(\Omega), H^{2-j}_q(\Omega))}
(\{(\tau\pd_\tau)^\ell\CD(\lambda) \mid \lambda \in 
\Sigma_{\epsilon, \lambda_0}\}) \leq r_b$$
for $\ell=0,1$ and $j=0,1,2$.
\end{itemize}
\end{thm}
\subsection{The model problems in $\BR^N$ and 
$\BR^N_+$}\label{subsec:7.1}
First, we consider the model problem in $\BR^N$.  In what follows,
let $\rho_{0*}$, $\gamma_{1*}$, $\gamma_{3*}$ and $\gamma_{4*}$ be 
positive constants.  Assume that there exist two positive constants
$b_1$ and $b_2$ for which
\begin{equation}\label{assump:0}
b_1 \leq \rho_{0*}, \gamma_{*1}, \gamma_{3*}, \gamma_{4*} \leq b_2.
\end{equation} 
Let us consider the following problems:
\begin{gather}
\rho_{0*}\lambda\bv - \mu\Delta\bv - \nu\nabla\dv\bv
-\gamma_{1*}\rho_{0*}\lambda^{-1}\nabla\dv\bv
= \bg \quad\text{in $\BR^N$}; \label{7.5} \\
\gamma_{3*}\lambda\vartheta - \gamma_{4*}\Delta\vartheta
= h \quad\text{in $\BR^N$}. \label{7.6}
\end{gather}
\begin{thm}\label{thm:7.3} Let $1 < q < \infty$ and 
$0 < \epsilon < \pi/2$.  Then, we have the following assertions:
\begin{itemize}
\item[\thetag1]~There exist a large constant $\lambda_0 > 0$ and 
an operator family $\CC_1(\lambda)$ with
$$\CC_1(\lambda) \in {\rm Hol}\,
(\Sigma_{\epsilon, \lambda_0}, 
\CL(L_q(\BR^N)^N, H^2_q(\BR^N)^N))$$
such that for any $\bg \in L_q(\BR^N)^N$ and 
$\lambda \in \Sigma_{\epsilon, \lambda_0}$, 
$\bv = \CC_1(\lambda)\bg$ is a unique 
solution of Eq. \eqref{7.5}, and
$$\CR_{\CL(L_q(\BR^N)^N, H^{2-j}_q(\BR^N)^N)}
(\{(\tau\pd_\tau)^\ell\CC_1(\lambda) \mid \lambda \in 
\Sigma_{\epsilon, \lambda_0}\}) \leq r_{b1}$$
for $\ell=0,1$ and $j=0,1,2$. Here, $\lambda_0$ and 
$r_{1b}$ depend solely on $N$, $q$, $\mu$, $\nu$, $b_1$ and $b_2$.
\item[\thetag2]~
Let $\lambda_0 \geq 1$. 
Then, there exists 
an operator family $\CD_1(\lambda) \in {\rm Hol}\,
(\Sigma_{\epsilon, \lambda_0}, \CL(L_q(\BR^N), H^2_q(\BR^N))$
such that for any $h \in Y_q(\BR^N)$
and $\lambda \in \Sigma_{\epsilon, \lambda_0}$,  $\vartheta
= \CD_1(\lambda)h$ is a unique 
solution of Eq. \eqref{7.6}, and
$$\CR_{\CL(L_q(\BR^N), H^{2-j}_q(\BR^N))}
(\{(\tau\pd_\tau)^\ell\CD_1(\lambda) \mid \lambda \in 
\Sigma_{\epsilon, \lambda_0}\}) \leq r_{b2}$$
for $\ell=0,1$ and $j=0,1,2$. 
Here,  $r_{b2}$ depends solely on $N$, $q$, $\lambda_0$, 
$b_1$ and $b_2$.
\end{itemize} 
\end{thm}
\begin{proof}
The assertion \thetag1 was proved in Enomoto-Shibata
\cite[Theorem 3.2]{ES1}, and so we may omit the proof.
To prove \thetag2,  
using the Fourier tranform $\CF$ and its inversion 
formula $\CF^{-1}$, we  define $\vartheta$ by
$$\vartheta = \CF^{-1}\Bigl[\frac{\CF[h](\xi)}
{\gamma_{3*}\lambda + \gamma_{4*}|\xi|^2}\Bigr](x).$$
Thus, by Lemma 3.1 and Theorem 3.3 in \cite{ES1},
we can show the assertion \thetag2.  Thus, we also
may omit the detailed proof. 
\end{proof}
Next we consider the half space problem.  Let 
$$
\BR^N_+ = \{x = (x_1, \ldots, x_N) \in \BR^N \mid x_N > 0\},
\quad 
\BR^N_0 = \{x = (x_1, \ldots, x_N) \in \BR^N \mid x_N = 0\},
$$
and $\bn_0 = {}^\top(0, \ldots, 0, -1)$. We consider the
following problems in $\BR^N_+$: 
\begin{gather}
\rho_{0*}\lambda\bv - \mu\Delta\bv - \nu\nabla\dv\bv
-\gamma_{1*}\rho_{0*}\lambda^{-1}\nabla\dv\bv
= \bg \quad\text{in $\BR^N_+$},
\quad \bv|_{\BR^N_0} = 0;  \label{7.7} \\
\gamma_{3*}\lambda\vartheta - \gamma_{4*}\Delta\vartheta
= h_1 \quad\text{in $\BR^N_+$},
\quad (\nabla \vartheta)\cdot\bn_0 = h_2
\quad\text{on $\BR^N_0$}. \label{7.8}
\end{gather}
\begin{thm}\label{thm:7.4} Let $1 < q < \infty$, 
$0 < \epsilon < \pi/2$, and $\lambda_0 \geq 1$. 
\begin{itemize}
\item[\thetag1]~There exist a large constant $\lambda_0 > 0$ and 
an operator family $\CC_2(\lambda)$
with 
$$\CC_2(\lambda) \in {\rm Hol}\,
(\Sigma_{\epsilon, \lambda_0}, \CL(L_q(\BR^N_+)^N, H^2_q(\BR^N_+)^N)$$
such that for any $\bg \in L_q(\BR^N_+)^N$ and 
$\lambda \in \Sigma_{\epsilon, \lambda_0}$, 
$\bv = \CC_2(\lambda)\bg$ is a unique 
solution of Eq. \eqref{7.7}, and
$$\CR_{\CL(L_q(\BR^N_+)^N, H^{2-j}_q(\BR^N_+)^N)}
(\{(\tau\pd_\tau)^\ell\CC_2(\lambda) \mid \lambda \in 
\Sigma_{\epsilon, \lambda_0}\}) \leq r_b$$
for $\ell=0,1$ and $j=0,1,2$.
\item[\thetag2]
Let $\lambda_0 \geq 1.$ Then, there exists 
an operator family $\CD_2(\lambda) \in {\rm Hol}\,
(\Sigma_{\epsilon, \lambda_0}, \CL(\CY_q(\BR^N_+), H^2_q(\BR^N_+))$
such that for any $(h_1, h_2)\in Y_q(\BR^N_+)$
and $\lambda \in \Sigma_{\epsilon, \lambda_0}$,  $\vartheta
= \CD_2(\lambda)(h_1, \lambda^{1/2}h_2, h_2)$ is a unique 
solution of Eq. \eqref{7.8}, and
$$\CR_{\CL(\CY_q(\BR^N_+), H^{2-j}_q(\BR^N_+))}
(\{(\tau\pd_\tau)^\ell\CD_2(\lambda) \mid \lambda \in 
\Sigma_{\epsilon, \lambda_0}\}) \leq r_b$$
for $\ell=0,1$ and $j=0,1,2$.
\end{itemize}
Here, $Y_q(\BR^N_+)$ and $\CY_q(\BR^N_+)$ are spaces defined in
Sect. 1 with $G = \BR^N_+$, and
 $r_b$ is a constant depending solely on 
$N$, $q$, $\lambda_0$, $b_1$ and $b_2$. 
\end{thm}
\begin{proof}
The first assertion has been proved in \cite[Theorem 4.1]{ES1}.
To prove the second one we divide a solution of \eqref{7.8} 
into two parts: $\vartheta = \vartheta_1 + \vartheta_2$,
where $\vartheta_1$ and $\vartheta_2$ are solutions of the 
following problems:
\begin{alignat}4
\gamma_{3*}\lambda\vartheta_1 - \gamma_{4*}\Delta\vartheta_1
&= h_1 &\quad&\text{in $\BR^N_+$},
&\quad (\nabla \vartheta_1)\cdot\bn_0 &= 0
&\quad&\text{on $\BR^N_0$}; \label{7.9}\\
\gamma_{3*}\lambda\vartheta_2 - \gamma_{4*}\Delta\vartheta_2
&= 0 &\quad&\text{in $\BR^N_+$},
&\quad (\nabla \vartheta_2)\cdot\bn_0 &= h_2
&\quad&\text{on $\BR^N_0$}. 
\label{7.10}
\end{alignat}
Given function $F_1$ defined on $\BR^N_+$, 
let $F^e_1$ be the even extension of $F_1$ to 
$x_N < 0$, that is $F^e_1(x) = F_1(x)$ for
$x_N > 0$ and $F^e_1(x) = F_1(x', -x_N)$ for 
$x_N < 0$, where $x' = (x_1, \ldots, x_{N-1})$.
We then define an $\CR$ bounded
solution operator $\CD_{21}(\lambda)$ 
acting on $F_1 \in L_q(\BR^N_+)$ by
$$\CD_{21}(\lambda)[F_1] = \CF^{-1}\Bigl[\frac{\CF[F^e_1](\xi)}
{\gamma_{3*}\lambda + \gamma_{4*}|\xi|^2}\Bigr].
$$
Obviously, $\vartheta_1 = \CD_1(\lambda)[h_1]$ is a
unique solution of  Eq. \eqref{7.9}. 

To construct an $\CR$ bounded solution operator for 
Eq. \eqref{7.10}, we introduce the partial Fourier
transform $\CF'$ and its inversion formula $\CF^{-1}_{\xi'}$,
which are defined by
\begin{align*}
&\hat f(\xi', x_N) = \CF'[f](\xi', x_N)
= \int_{\BR^{N-1}}e^{-ix'\cdot\xi'}f(x', x_N)\,dx',
\\ 
&\CF^{-1}_{\xi'}[g(\xi', x_N)](x')
= \frac{1}{(2\pi)^{N-1}}
\int_{\BR^{N-1}}e^{ix'\cdot\xi'}g(\xi', x_N)\,d\xi',
\end{align*}
where $\xi' = (\xi_1, \ldots,\xi_{N-1}) \in \BR^{N-1}$
and $x'\cdot\xi' = \sum_{j=1}^{N-1}x_j\xi_j$. Applying the 
partial Fourier transform to \eqref{7.10},
we have
$$(\gamma_{3*}\lambda + \gamma_{4*}|\xi'|^2)\hat \vartheta
 - \gamma_{*4}\pd_N^2 \hat\vartheta = 0
\quad\text{for $x_N > 0$}, \quad 
\pd_N\vartheta|_{x_N=0} = -\hat h_2(\xi', 0),
$$
where $|\xi'|^2 = \sum_{j=1}^{N-1}\xi_j^2$ and 
$\pd_N = \pd/\pd _N$.  Thus, $\vartheta_2$ is given by
\begin{align*}
\vartheta_2 &= \CF^{-1}\Bigl[\frac{e^{-\sqrt{\gamma_{3*}
\gamma_{4*}^{-1}\lambda + |\xi'|^2}\,x_N}}
{\sqrt{\gamma_{3*}\gamma_{4*}^{-1}\lambda + |\xi'|^2}}
\hat h_2(\xi', 0)\Bigr](x') \\
&= \int^\infty_0\CF^{-1}_{\xi'}
\Bigl[\frac{e^{-\sqrt{\gamma_{3*}
\gamma_{4*}^{-1}\lambda + |\xi'|^2}\,(x_N+y_N)}}
{\sqrt{\gamma_{3*}\gamma_{4*}^{-1}\lambda + |\xi'|^2}}
\sqrt{\gamma_{3*}
\gamma_{4*}^{-1}\lambda + |\xi'|^2}
\hat h_2(\xi', y_N)\Bigr](x') \\
&- \int^\infty_0\CF^{-1}_{\xi'}
\Bigl[\frac{e^{-\sqrt{\gamma_{3*}
\gamma_{4*}^{-1}\lambda + |\xi'|^2}\,(x_N+y_N)}}
{\sqrt{\gamma_{3*}\gamma_{4*}^{-1}\lambda + |\xi'|^2}}
\pd_N\hat h_2(\xi', y_N)\Bigr](x').
\end{align*}
Writing 
$$\sqrt{\gamma_{3*}\gamma_{4*}^{-1}\lambda + |\xi'|^2}
= \frac{\gamma_{3*}\gamma_{4*}^{-1}\lambda}
{\sqrt{\gamma_{3*}\gamma_{4*}^{-1}\lambda + |\xi'|^2}}
-\sum_{j=1}^{N-1}
\frac{i\xi_ji\xi_j}
{\sqrt{\gamma_{3*}\gamma_{4*}^{-1}\lambda + |\xi'|^2}}
$$
we have
\begin{align*}
&\int^\infty_0\CF^{-1}_{\xi'}
\Bigl[\frac{e^{-\sqrt{\gamma_{3*}
\gamma_{4*}^{-1}\lambda + |\xi'|^2}\,(x_N+y_N)}}
{\sqrt{\gamma_{3*}\gamma_{4*}^{-1}\lambda + |\xi'|^2}}
\sqrt{\gamma_{3*}
\gamma_{4*}^{-1}\lambda + |\xi'|^2}
\hat h_2(\xi', y_N)\Bigr](x')\\
&=
\int^\infty_0\CF^{-1}_{\xi'}
\Bigl[\lambda^{1/2}
e^{-\sqrt{\gamma_{3*}\gamma_{4*}^{-1}\lambda + |\xi'|^2}\,(x_N+y_N)}
\frac{\gamma_{3*}\gamma_{4*}^{-1}\lambda^{1/2}}
{\sqrt{\gamma_{3*}\gamma_{4*}^{-1}\lambda + |\xi'|^2}}
\hat h_2(\xi', y_N)\Bigr](x') \\
& -\sum_{j=1}^{N-1}
\int^\infty_0\CF^{-1}_{\xi'}
\Bigl[|\xi'|e^{-\sqrt{\gamma_{3*}
\gamma_{4*}^{-1}\lambda + |\xi'|^2}\,(x_N+y_N)}
\frac{i\xi_j}
{|\xi'|\sqrt{\gamma_{3*}\gamma_{4*}^{-1}\lambda + |\xi'|^2}}
\CF'[\pd_j h_2](\xi', y_N)\Bigr](x').
\end{align*}
We then define an operator $\CD_{22}(\lambda)$ acting on
$(F_2, F_3) \in L_q(\BR^N_+\times H^1_q(\BR^N_+)$ by 
\begin{align*}
\CD_{22}(\lambda)(F_2, F_3)
& = \int^\infty_0\CF^{-1}_{\xi'}
\Bigl[\lambda^{1/2}
e^{-\sqrt{\gamma_{3*}\gamma_{4*}^{-1}\lambda + |\xi'|^2}\,(x_N+y_N)}
\frac{\gamma_{3*}\gamma_{4*}^{-1}}
{\sqrt{\gamma_{3*}\gamma_{4*}^{-1}\lambda + |\xi'|^2}}
\CF'[F_2](\xi', y_N)\Bigr](x') \\
& -\sum_{j=1}^{N-1}
\int^\infty_0\CF^{-1}_{\xi'}
\Bigl[|\xi'|e^{-\sqrt{\gamma_{3*}
\gamma_{4*}^{-1}\lambda + |\xi'|^2}\,(x_N+y_N)}
\frac{i\xi_j}
{|\xi'|\sqrt{\gamma_{3*}\gamma_{4*}^{-1}\lambda + |\xi'|^2}}
\CF'[\pd_j F_3](\xi', y_N)\Bigr](x')\\
&- \int^\infty_0\CF^{-1}_{\xi'}
\Bigl[\frac{e^{-\sqrt{\gamma_{3*}
\gamma_{4*}^{-1}\lambda + |\xi'|^2}\,(x_N+y_N)}}
{\sqrt{\gamma_{3*}\gamma_{4*}^{-1}\lambda + |\xi'|^2}}
\CF'[\pd_NF_3](\xi', y_N)\Bigr](x').
\end{align*}
Obviously, $\vartheta_2 = \CD_2(\lambda)(\lambda^{1/2}h_2, h_2)$.
Moreover, the $\CR$ boundedness of the operator
$\CD_{22}(\lambda)$ follows from Lemma 4.2 
in \cite{ES1}.  This completes the proof of the assertion
\thetag2.
\end{proof}
\subsection{Problem in a bent half space}\label{sec:7.2}
Let $\Phi: \BR^N\to\BR^N$ be a bijection of $C^2$ class and let 
$\Phi^{-1}$ be its inverse map.  Writing $\nabla\Phi=
\CA + B(x)$ and $\nabla\Phi^{-1} = \CA_{-} + B_{-}(x)$,
we assume that $\CA$ and $\CA_{-}$ are orthogonal matrices
with constant coefficients and $B(x)$ and $B_-(x)$ are matrices
of functions in $C^1(\BR^N)$ with $N < r < \infty$ such that
\begin{equation}\label{7.11}
\|(B, B_-)\|_{L_\infty(\BR^N)} \leq M_1, \quad
\|\nabla(B, B_-)\|_{L_\infty(\BR^N)} \leq M_2.
\end{equation}
We will choose $M_1$ small enough eventually, and so we may assume that
$0 < M_1 \leq 1 \leq M_2$ in the following. Set 
$\Omega_+ = \Phi(\BR^N_+)$ and $\Gamma_+ = \Phi(\BR^N_0)$.
Let $\bn_+$ be the unit outer normal to $\Gamma_+$.  
Since $\Gamma_+$ is represented by $\Phi_{-1, N}(y) = 0$, where
$\Phi^{-1} = {}^\top(\Phi_{-1,1}, \ldots, \Phi_{-1,N})$, 
$\bn_+$ is given by 
\begin{equation}\label{7.12}
\bn_+ = -\frac{\nabla\Phi_{-1,N}}{|\nabla\Phi_{-1,N}|}
= -\frac{{}^\top(\CA_{N1}+B_{N1}, \ldots, \CA_{NN} + B_{NN})}
{\sqrt{\sum_{j=1}^N(A_{Nj} + B_{Nj})^2}}.
\end{equation}
Choosing $M_1 > 0$ small enough, by \eqref{7.11} we have
\begin{equation}\label{7.12*}
\bn_+ = -{}^\top(\CA_{N1}, \ldots, \CA_{NN}) + \tilde\bn_+
\end{equation}
where $\tilde\bn_+$ has the estimates:
\begin{equation}\label{7.13}
\|\tilde\bn_+\|_{L_\infty(\BR^N)} \leq C_NM_1,
\quad \|\nabla\tilde\bn_+\|_{L_\infty(\BR^N)}
\leq C_{M_2}.
\end{equation}
We consider the following two equations:
\begin{gather}
\rho_{0*}\lambda\bv - \mu\Delta\bv - \nu\nabla\dv\bv
-\gamma_{1*}\rho_{0*}\lambda^{-1}\nabla\dv\bv
= \bg \quad\text{in $\Omega_+$},
\quad \bv|_{\Gamma_+} = 0;  \label{7.14} \\
\gamma_{3*}\lambda\vartheta - \gamma_{4*}\Delta\vartheta
= h_1 \quad\text{in $\Omega_+$},
\quad (\nabla \vartheta)\cdot\bn_0 = h_2
\quad\text{on $\Gamma_+$}. \label{7.15}
\end{gather}
\begin{thm}\label{thm:7.5} Let $1 < q < \infty$ and 
$0 < \epsilon < \pi/2$.  Then, we have the following assertions:
\begin{itemize}
\item[\thetag1]~There exist a large constant $\lambda_0 > 0$ and 
an operator family $\CC_3(\lambda)$
with 
$$\CC_3(\lambda) \in {\rm Hol}\,
(\Sigma_{\epsilon, \lambda_0}, \CL(L_q(\Omega_+)^N, H^2_q(\Omega_+)^N)$$
such that for any $\bg \in L_q(\Omega_+)^N$ and 
$\lambda \in \Sigma_{\epsilon, \lambda_0}$, 
$\bv = \CC_3(\lambda)\bg$ is a unique 
solution of Eq. \eqref{7.14}, and
$$\CR_{\CL(L_q(\Omega_+)^N, H^{2-j}_q(\Omega_+)^N)}
(\{(\tau\pd_\tau)^\ell\CC_3(\lambda) \mid \lambda \in 
\Sigma_{\epsilon, \lambda_0}\}) \leq r_{b1}$$
for $\ell=0,1$ and $j=0,1,2$. Here, $r_{b1}$ is a constant
depending solely on $N$, $q$, $\mu$, $\nu$, $b_1$ and $b_2$. 
\item[\thetag2]~Let 
$Y_q(\Omega_+)$ and $\CY_q(\Omega_+)$ be spaces defined by
replacing $\Omega$ by $\Omega_+$ in Theorem \ref{thm:7.2}. 
Then, there exist a positive constant $\lambda_0$ and 
an operator family $\CD_3(\lambda) \in {\rm Hol}\,
(\Sigma_{\epsilon, \lambda_0}, \CL(\CY_q(\Omega_+), H^2_q(\Omega_+))$
such that for any $(h_1, h_2)\in Y_q(\Omega_+)$
and $\lambda \in \Sigma_{\epsilon, \lambda_0}$,  $\vartheta
= \CD_3(\lambda)(h_1, \lambda^{1/2}h_2, h_2)$ is a unique 
solution of Eq. \eqref{7.15}, and
$$\CR_{\CL(L_q(\Omega_+), H^{2-j}_q(\Omega_+))}
(\{(\tau\pd_\tau)^\ell\CD_2(\lambda) \mid \lambda \in 
\Sigma_{\epsilon, \lambda_0}\}) \leq r_{b2}$$
for $\ell=0,1$ and $j=0,1,2$. Here, $r_{b2}$ is a constant
depending solely on $N$, $q$, $b_1$ and $b_2$. 
\end{itemize}
\end{thm}
\begin{proof}
The first assertion was proved in Enomoto-Shibata
\cite[Theorem 5.1]{ES1}, and so we may omit the proof.
Thus, we prove the assertion \thetag2 below.  For this purpose,
we shall transform \eqref{7.15} into the equations in $\BR^N_+$
by the change of variables: $x = \Phi^{-1}(y)$ with
 $x \in \BR^N_+$ and $y \in \Omega_+$. We have
\begin{equation}\label{7.16}
\frac{\pd}{\pd y_j} = \sum_{k=1}^N (\CA_{kj} + B_{kj}(x))
\frac{\pd}{\pd x_k},
\end{equation}
where $\CA_{kj}$ is the $(k,j)^{\rm th}$ component of $\CA_{-}$ and 
$B_{kj}(x)$ is the $(k,j)^{\rm th}$ component of 
$B_{-}(\Phi(x))$.  Let $\varphi(x) = \vartheta(\Phi(x))$ in 
\eqref{7.15}, and then by \eqref{7.12*} and  \eqref{7.16}
we have
\begin{equation}\label{7.17}
\gamma_{3*}\lambda\varphi - \gamma_{4*}
[\Delta\varphi + A_1\nabla^2\varphi + A_2\nabla\varphi]
= H_1 \quad\text{in $\BR^N_+$}, \quad 
(\nabla\varphi)\cdot\bn_0 + (\nabla\varphi)\cdot\bn_1
= H_2
\quad\text{on $\BR^N_0$}.
\end{equation}
Here, we have set 
\begin{align*}
A_1\nabla^2\varphi& = \sum_{j,k,\ell=1}^N(\CA_{kj}B_{\ell j}(x)
+ \CA_{\ell j}B_{kj}(x) + B_{kj}(x)B_{\ell j}(x))
\frac{\pd^2\varphi}{\pd x_k\pd _\ell}, \\
A_2\nabla\varphi & = \sum_{j,k,\ell=1}^N(\CA_{kj} + B_{kj}(x))
\Bigl(\frac{\pd}{\pd x_k}B_{\ell j}(x)\Bigr)\frac{\pd \varphi}{\pd x_\ell},
\\
(\nabla\varphi)\cdot\bn_1 & = 
\sum_{j,k=1}^N(\CA_{Nj}B_{kj}(x) + 
\tilde n_j(x)\CA_{kj} + \tilde n_j(x)B_{kj}(x))
\frac{\pd \varphi}{\pd x_k}.
\end{align*}
Notice that
\begin{equation}\label{7.19*}\begin{split}
\|A_1\nabla^2\varphi\|_{L_q(\BR^N_+)} &
\leq CM_1\|\nabla^2\varphi\|_{L_q(\BR^N_+)}, \\
\|A_2\nabla\varphi\|_{L_q(\BR^N_+)} &
\leq C_{M_2}\|\nabla\varphi\|_{L_q(\BR^N_+)}, \\
\|(\nabla\varphi)\cdot\bn_1\|_{L_q(\BR^N_+)}
& \leq CM_1\|\nabla\varphi\|_{L_q(\BR^N_+)}, \\
\|(\nabla\varphi)\cdot\bn_1\|_{H^1_q(\BR^N_+)}
& \leq C(M_1\|\nabla^2\varphi\|_{L_q(\BR^N_+)}
+ C_{M_2}\|\nabla\varphi\|_{L_q(\BR^N_+)}).
\end{split}\end{equation}
Let $\CC_2(\lambda)$ be an $\CR$-bounded solution operator given in Theorem \ref{thm:7.4} and set $\psi = \CC_2(\lambda)F_\lambda(H_1, H_2)$.  Here
and in the following, $F_\lambda$ is an operator acting on 
$(H_1, H_2) \in Y_q(\BR^N_+)$ defined by
$F_\lambda(H_1, H_2) = (H_1, \lambda^{1/2}H_2, H_2) \in \CY_q(\BR^N_+)$.
%Here and in the following, $Y_q(\BR^N_+)$ and $\CY(\BR^N_+)$
%are the spaces given in Theorem \ref{thm:7.4}. 
We then have 
\begin{equation}\label{7.18}\begin{split}
\gamma_{3*}\lambda\psi - \gamma_{4*}(\Delta\psi + A_1\nabla^2\psi
+ A_2\nabla\psi)  = H_1 +R_1(\lambda)(H_1, H_2)\quad&\text{in
$\BR^N_+$}, \\
(\nabla\psi)\cdot\bn_0 + (\nabla\psi)\cdot\bn_1
= H_2 + R_2(\lambda)(H_1, H_2)\quad&\text{on $\BR^N_0$}, 
\end{split}\end{equation}
where
\begin{equation}\label{7.20}\begin{split}
R_1(\lambda)(H_1, H_2) & = \gamma_{4*}(A_1\nabla^2\CC_2(\lambda)
F_\lambda(H_1, H_2) + A_2\nabla\CC_2(\lambda)F_\lambda(H_1, H_2)),
\\
R_2(\lambda)(H_1, H_2) & = (\nabla\CC_2(\lambda)F_\lambda(H_1, H_2))\cdot\bn_1.
\end{split}\end{equation}
For $F = (F_1, F_2, F_3) \in \CY_q(\BR^N_+)$, let 
\begin{equation*}
\CR_1(\lambda)F = \gamma_{4*}(A_1\nabla^2\CC_2(\lambda)F)
+ A_2\nabla\CC_2(\lambda)F, \quad
\CR_2(\lambda)F = [\nabla\CC_2(\lambda)F]\cdot\bn_1,
\end{equation*}
and let $\CR(\lambda)F = (\CR_1(\lambda)F, \CR_2(\lambda)F)
\in Y_q(\BR^N_+)$ and $R(\lambda)(H_1, H_2)
= (R_1(\lambda)(H_1, H_2), R_2(\lambda)(H_1, H_2))$. 
We then have 
\begin{equation}\label{7.21}
\CR(\lambda)F_\lambda(H_1, H_2) = R(\lambda)(H_1, H_2).
\end{equation}
We now use the following two lemmas to calculate the 
$\CR$-norm. 
\begin{lem}\label{lem:7.3.4}
\thetag1~Let $X$ and $Y$ be Banach spaces,
and let $\CT$ and $\CS$ be
$\CR$-bounded families in $\CL(X,Y)$.  
Then, $\CT+ \CS = \{T + S \mid 
T \in \CT, \enskip S \in \CS\}$ is also an $\CR$-bounded 
family in 
$\CL(X,Y)$ and 
$$\CR_{\CL(X,Y)}(\CT + \CS) \leq 
\CR_{\CL(X,Y)}(\CT) + \CR_{\CL(X,Y)}(\CS).$$
\thetag2~Let $X$, $Y$ and $Z$ be Banach spaces,  and let 
$\CT$ and $\CS$ be $\CR$-bounded families in $\CL(X, Y)$ and 
$\CL(Y, Z)$, respectively.  Then, $\CS\CT = \{ST \mid 
T \in \CT, \enskip S \in \CS\}$ is also an $\CR$-bounded 
family in $\CL(X, Z)$ and 
$$\CR_{\CL(X, Z)}(\CS\CT) \leq \CR_{\CL(X,Y)}(\CT)\CR_{\CL(Y, Z)}(\CS).$$
\end{lem}
\begin{lem}\label{lem:7.3.5}
Let $1 < p, \, q < \infty$ and let $D$ be a domain in $\BR^N$. \\
\thetag1~Let $m(\lambda)$ be a bounded function defined on a subset 
$\Lambda$ in a complex plane $\BC$ and let $M_m(\lambda)$ be a
multiplication operator with $m(\lambda)$ defined by 
$M_m(\lambda)f = m(\lambda)f$ for any $f \in L_q(D)$.  Then, 
$$\CR_{\CL(L_q(D))}(\{M_m(\lambda) \mid \lambda \in \Lambda\}) 
\leq C_{N,q,D}\|m\|_{L_\infty(\Lambda)}.$$
\thetag2~Let $n(\tau)$ be a $C^1$ function defined on 
$\BR\setminus\{0\}$ that satisfies the conditions:  
$|n(\tau)| \leq\gamma$ and $|\tau n'(\tau)| \leq \gamma$ 
with some constant $\gamma > 0$ for any $\tau \in \BR\setminus\{0\}$.
Let $T_n$ be an operator valued Fourier multiplier defined by 
$T_nf = \CF^{-1}[n\CF[f]]$ for any $f \in \CS(\BR, X)$ 
with $\CF[f] \in \CD(\BR, X)$. 
Then, $T_n$ is extended
to a bounded linear operator from $L_p(\BR, L_q(D))$ into itself.
Moreover, denoting this extension also by $T_n$, we have 
$$\|T_n\|_{\CL(L_p(\BR, L_q(D)))} \leq C_{p,q,D}\gamma.$$
\end{lem}
\begin{remark}\label{rem:1}
For  proofs of Lemma \ref{lem:7.3.4} and Lemma \ref{lem:7.3.5}, we refer 
to \cite[p.28, 3.4.Proposition and p.27, 3.2.Remarks \thetag4]{DHP}
(cf. also Bourgain \cite{Bourgain}), respectively.
\end{remark}
By  Lemma \ref{lem:7.3.4},
Lemma \ref{lem:7.3.5}, \eqref{7.19*}, and Theorem \ref{thm:7.4}
\thetag2 we have
\begin{equation}\label{7.22}
\CR_{\CL(\CY_q(\BR^N_+))}(\{(\tau\pd_\tau)^\ell
F_\lambda \CR(\lambda) \mid \lambda \in \Sigma_{\epsilon, \lambda_1}
\}) \leq r_b(CM_1 +  C_{M_2}\tilde\lambda_0^{-1/2})
\end{equation}
for any $\tilde\lambda_0\geq \lambda_0$.  In fact, by \eqref{7.19*}, we have
\begin{align*}
&\int^1_0\|\sum_{j-1}^nr_j(u)A_2\nabla\CC_2(\lambda_j)
F_j\|_{L_q(\BR^N_+)}^q\,du \leq C_{M_2}^q\int^1_0\|\sum_{j=1}^n
r_j(u)\nabla\CC_2(\lambda_j)F_j\|_{L_q(\BR^N_+)}^q\,du.
\end{align*}
By Lemma \ref{lem:7.3.5}, we have
\begin{align*}
\int^1_0\|\sum_{j=1}^n
r_j(u)\nabla\CC_2(\lambda_j)F_j\|_{L_q(\BR^N_+)}^q\,du
&=\int^1_0\|\sum_{j=1}^nr_j(u)\lambda_j^{-1/2}
\lambda_j^{1/2}\nabla\CC_2(\lambda_j)F_j\|_{L_q(\BR^N_+)}^q\,du\\
& \leq \tilde\lambda_0^{-q/2}
\int^1_0\|\sum_{j=1}^nr_j(u)
\lambda_j^{1/2}\nabla\CC_2(\lambda_j)F_j\|_{L_q(\BR^N_+)}^q\,du
\end{align*}
for any $\lambda_j \in \Sigma_{\epsilon, \tilde\lambda_0}$ and 
$\tilde\lambda_0 \geq \lambda_0$. Thus, 
by Theorem \ref{thm:7.2} \thetag2, we have
$$
\int^1_0\|\sum_{j-1}^nr_j(u)A_2\nabla\CC_2(\lambda_j)
F_j\|_{L_q(\BR^N_+)}^q\,du 
\leq \tilde\lambda_0^{-q/2}r_b^q\int^1_0\|\sum_{j=1}^nr_j(u)F_j
\|_{L_q(\BR^N_+)}^q\,du.
$$
Analogously, we can estimate 
$\CR_{\CL(\CY_q(\BR^N_+), L_q(\BR^N_+))}$ norm of 
$B_1$ and $B_2$ and $\CR_{\CL(\CY_q(\BR^N_+), H^1_q(\BR^N_+))}$
norm of $B_3$, where 
$B_1= A_1\nabla^2\CC_2(\lambda)F$, 
$B_2= \lambda^{1/2}[\nabla\CC_2(\lambda)F]\cdot\bn$, 
$B_3=[\nabla\CC_2(\lambda)F]\cdot\bn$, and so we have \eqref{7.22}. 

Choosing $M_1$ so small that $r_bCM_1 \leq 1/4$ and 
choosing $\tilde\lambda_0$ so large that 
$r_bC_{M_2}\tilde\lambda_0^{-1/2} \leq 1/4$ in \eqref{7.22},
we have 
\begin{equation}\label{7.23}
\CR_{\CL(\CY_q(\BR^N_+))}(\{(\tau\pd_\tau)^\ell
F_\lambda \CR(\lambda) \mid \lambda \in \Sigma_{\epsilon, \tilde\lambda_0}
\}) \leq 1/2.
\end{equation}
Since $\CR$-boundedness implies the usual boundedness, we have
$$\|F_\lambda \CR(\lambda)F_\lambda(H_1, H_2)\|_{\CY_q(\BR^N_+)}
\leq (1/2)\|F_\lambda(H_1, H_2)\|_{L_q(\CY_q(\BR^N_+)}.
$$
Here and in the following, the norm of  $\CY_q(\BR^N_+)$
is given by
$$\|(F_1, F_2, F_3)\|_{\CY_q(\BR^N_+)}
= \|(F_1, F_2)\|_{L_q(\BR^N_+)} + \|F_3\|_{H^1_q(\BR^N_+)}.
$$
Thus, $\|F_\lambda(H_1, H_2)\|_{\CY_q(\BR^N_+)}$ gives the 
equivalent norm of $Y_q(\BR^N_+)$. By  \eqref{7.21} and \eqref{7.23} 
we see that $(\bI + R(\lambda))^{-1} = \sum_{j=0}^\infty
(-R(\lambda))^j$ exists as an operator from 
$Y_q(\BR^N_+)$ into itself and its operator norm does not
exceed $2$. Thus, in view of \eqref{7.18}, 
$\varphi= \CC_2(\lambda)F_\lambda(\bI + R(\lambda))^{-1}(H_1, H_2)$
is a solution of Eq. \eqref{7.17}.

On the other hand, by \eqref{7.23} and lemma \ref{lem:7.3.4},
we see that 
$(\bI + F_\lambda\CR(\lambda))^{-1} 
= \sum_{j=1}^\infty (F_\lambda \CR(\lambda))^j$
exists and 
\begin{equation}\label{7.24}
\CR_{\CL(\CY_q(\BR^N_+))}(\{(\tau\pd_\tau)^\ell
(\bI + F_\lambda \CR(\lambda))^{-1} \mid \lambda \in 
\Sigma_{\epsilon, \tilde\lambda_0}\}) \leq 4.
\end{equation}
Set $\tilde\CC_3(\lambda) = \CC_2(\lambda)(\bI + F_\lambda\CR(\lambda))^{-1}
$.  Since $\CR(\lambda)F_\lambda = R(\lambda)$ as follows from 
\eqref{7.21}, we have
\begin{align*}
(\bI + F_\lambda\CR(\lambda))^{-1}F_\lambda 
&= \sum_{j=1}^\infty(-1)^j(F_\lambda \CR(\lambda))^jF_\lambda
= \sum_{j=0}^\infty(-1)^jF_\lambda(\CR(\lambda)F_\lambda)^j \\
&= F_\lambda\sum_{j=0}^\infty (-R(\lambda))^j 
= F_\lambda(\bI+R(\lambda))^{-1},
\end{align*}
which leads to  $\varphi = \tilde \CC_3(\lambda)F_\lambda(H_1, H_2)$.
Thus, $\tilde\CC_3(\lambda)$ is an $\CR$-bounded solution operator
for Eq. \eqref{7.17}.  Set 
$$\CC_3(\lambda)F = [\tilde\CC_3(\lambda)(F\circ \Phi^{-1})]\circ\Phi,
$$
and then $\CC_3(\lambda)$ is an $\CR$-bounded solution operator
of Eq. \eqref{7.15}.  This completes the proof of
the assertion \thetag2. 
\end{proof}
%%%%%%%%%
\subsection{Proof of Theorem \ref{thm:7.2}}\label{subsec:7.3}

To prove Theorem \ref{thm:7.2}, we need to use several properties of
uniform $C^2$ domain, which are stated in the following proposition.
\begin{prop}\label{prop:6.1} \mdseries Let $\Omega$ be a uniform 
$C^2$-domain in $\mathbb{R}^N$ with boundary $\Gamma$. 
Then,  for any positive constant $M_1$, there exist constants
 $M_2 > 0$, $d \in (0, 1)$,  at most countably 
many functions $\Phi_j \in C^2(\mathbb{R}^N )$,  and points $x^1_j \in
\Omega$ and $x^2_j \in \Gamma$ $( j \in \BN)$   
such that the following assertions hold:
\begin{enumerate}
\item[\thetag1]~For every $j \in \BN$, the map $\BR^N \ni 
x \rightarrow \Phi_j(x) \in \BR^N$ is bijective.
\item[\thetag2]~$\Omega = (\bigcup_{j=1}^\infty B_d(x^1_j))
\cup (\bigcup_{j=1}^\infty (\Phi_j(\BR_+^N) \cap B_d(x^2_j)))$, 
$B_d(x^1_j) \subset \Omega$, 
$\Phi_j(\BR_+^N) \cap B_d(x^2_j) = \Omega \cap B_d(x^2_j)$, 
and $\Phi_j(\BR_0^N) \cap B_d(x^2_j) = \Gamma \cap B_d(x^2_j)$.
\item[\thetag3]~There exist $C^\infty$ functions $\zeta^i_j (i =  1, 2, j \in \BN)$ 
such that ${\rm supp}\, \zeta^i_j$, ${\rm supp}\,
\tilde{\zeta}^i_j \subset B_d(x^i_j)$, 
$\|\zeta^i_j\|_{H^2_\infty(\mathbb{R}^N)} 
\leq c_0$, $\| \tilde{\zeta}^i_j\|_{H^2_\infty(\mathbb{R}^N)} \leq c_0$, 
$\tilde{\zeta}^i_j = 1$  on ${\rm supp}\, \zeta^i_j$, 
$\sum_{i=1,2}\sum_{j=1}^\infty \zeta^i_j = 1$ on $\overline{\Omega}$, 
$\sum_{j=1}^\infty \zeta^1_j = 1$ on $\Gamma$. 
Here, $c_0$ is a constant which depends on $M_2, N, q, q'$ and $r$, 
but is independent of $j \in \BN$.
\item[\thetag4]~ 
$\nabla \Phi_j = \mathcal{R}_j + R_i, \nabla( \Phi_j)^{-} 
= \mathcal{R}_j^{-} + R_j^{-} $, 
where $\mathcal{R}_j$ and $\mathcal{R}_j^{-}$ are 
$ N \times N$ constant orthogonal matrices, and $R_j$ and $R_j^{-}$ 
are $N\times N$ matrices of $H_\infty^1$ functions defined on $\BR^N$ 
which satisfies the conditions: $\|R_j\|_{L_\infty(\BR^N)} \leq M_1$, 
$\|R_j^{-}\|_{L_\infty(\BR^N)} \leq M_1$, 
$\|\nabla R_j\|_{L_\infty(\BR^N)} \leq M_2$ and 
$\|\nabla R_j^-\|_{L_\infty(\mathbb{R}^N)} \leq M_2$ for any 
$j \BN$.
\item[\thetag5]~There exist a natural number $L>2$ such that
 any $L+1$ distinct sets of $\{B_d(x^i_j) \mid  i = 1, 2, 
\enskip j \in \BN\}$ 
have an empty intersection.
\end{enumerate}
\end{prop}
In what follows, we write 
$\Omega_\ell = \Phi_\ell(\BR^N_+)$ 
and $\Gamma_\ell = \Phi_\ell(\BR^N_0)$
for $\ell \in \BN$. Moreover, we write $B_d(x^i_j)$ 
simply by $B^i_j$.  
Since $\rho_0(x)$ and $\gamma_k(x)$ ($k=1, 3, 4$) are uniformly continuous
functions on $\overline{\Omega}$, choosing $d$ smaller
if necessary, we may assume that 
\begin{equation}\label{unif:1}
|\rho_0(x) - \rho_0(x^i_j)| \leq M_1, \quad |\gamma_k(x)-\gamma_k(x^i_j)|
\leq M_1
\quad \text{for $x \in B^i_J\cap\overline{\Omega}$,
\enskip $k=1, 3, 4$}.
\end{equation}
By the finite intersection property stated in
Proposition \ref{prop:6.1} \thetag5, we have 
\begin{equation}\label{7.25}
\Bigl(\sum_{i=1, 2}
\sum_{j=1}^\infty \|f\|^q_{L_q(B^i_j\cap\Omega)}\Bigr)^{1/q} 
\leq C_q\|f\|_{L_q(\Omega)}
\end{equation}
for any $f \in L_q(\Omega)$ and $1 \leq q < \infty$. In particular, 
 by \eqref{7.25} we have
\begin{lem}\label{lem:cal} Let $i=1, 2$ and $1 < q < \infty$.
 Let $\{f_j\}_{j=0}^\infty$ be a sequence of functions
in $L_q(\Omega)$ such that $\sum_{j=0}^\infty \|f_j\|_{L_q(\Omega)}^q
< \infty$, and ${\rm supp}\, f_j \subset B^i_j$ 
$(j \in \BN)$.
Then, $\sum_{j=0}^\infty f_j \in L_q(\Omega)$ and 
$\|\sum_{j=1}^\infty f_j\|_{L_q(\Omega)} \leq (\sum_{j=1}^\infty
\|f_j\|_{L_q(\Omega)}^q)^{1/q}$.
\end{lem}

We first prove the assertion  \thetag1 in Theorem \ref{thm:7.2}.  
We construct a parametrix.  Let 
$\bv^1_j \in H^2_q(\BR^N)^N$ be solutions of the equations:
\begin{equation}\label{7.26} 
\rho_0(x^1_j)\lambda \bv^1_j - \mu\Delta\bv^1_j - \nu
\nabla\dv\bv^1_j - \gamma_1(x^1_j)\rho_0(x^1_j)\lambda^{-1}
\nabla\dv\bv^1_j = \zeta^1_j\bg
\quad\text{in $\BR^N$}; 
\end{equation}
and $\bv^2_j \in H^2_q(\Omega_j)^N$ solutions of the equations:
\begin{equation}\label{7.27} 
\rho_0(x^2_j)\lambda \bv^2_j - \mu\Delta\bv^2_j - \nu
\nabla\dv\bv^2_j - \gamma_1(x^2_j)\rho_0(x^2_j)\lambda^{-1}
\nabla\dv\bv^2_j = \zeta^2_j\bg
\quad\text{in $\Omega_j$}, \quad \bv^2_j|_{\Gamma_j} = 0.
\end{equation}
By Theorem \ref{thm:7.3} \thetag1 and Theorem \ref{thm:7.5}
\thetag1, there are $\CR$-bounded solution operators $\CC^i_j(\lambda)$
with 
$$\CC^i_j(\lambda) \in {\rm Hol}\,(\Sigma_{\epsilon, \lambda_0},
\CL(L_q(\Omega^i_j)^N, H^2_q(\Omega^i_j)^N))$$
such that
for any $\bg \in L_q(\Omega)$ and $\lambda \in \Sigma_{\epsilon, \lambda_0}$, 
$\bv^1_j = \CC^1_j(\lambda)\zeta^1_j\bg$ are solutions of Eq. 
\eqref{7.26} and $\bv^2_j = \CC^2_j(\lambda)\zeta^2_j\bg$ solutions of Eq. 
\eqref{7.27},  where we have set $\Omega^1_j = \BR^N$ and 
$\Omega^2_j = \Omega_j$.  Moreover, we have
\begin{equation}\label{7.29}
\CR_{\CL(L_q(\Omega^i_j)^N, H^{2-k}_q(\Omega^i_j)^N)}
(\{(\tau\pd_\tau)^\ell(\lambda^{k/2}\CC^i_j(\lambda))\mid
\lambda \in \Sigma_{\epsilon, \lambda_0}\})
\leq r_b
\end{equation}
for $\ell=0,1$ and $k=0,1,2$. Notice that $\lambda_0$ and 
$r_b$ are independent of $i=1,2$ and $j \in \BN$. 
Let 
\begin{align*}
\CU_1(\lambda)\bg  = \sum_{i=1,2}\sum_{j=1}^\infty
\tilde\zeta^i_j\CC^i_j(\lambda)\zeta^i_j\bg
\end{align*}
for $\bg \in L_q(\Omega)^N$.  By Lemma \ref{lem:cal}, 
we have
\begin{equation}
\CR_{\CL(L_q(\Omega)^N, H^{2-k}_q(\Omega)^N)}
(\{(\tau\pd_\tau)^\ell(\lambda^{k/2} \CU(\lambda)) \mid 
\lambda \in \Sigma_{\epsilon, \lambda_0}\}) \leq C_{N,q}r_b.
\label{7.33}
\end{equation}
In fact, by \eqref{7.29} and \eqref{7.25} we have 
\begin{align*}
&\sum_{i=1,2}\sum_{j=1}^\infty\int^1_0\int_\Omega
 |\sum_{k=1}^nr_k(u)\tilde\zeta^i_j\CC^i_j(\lambda_k)\zeta^i_j\bg_k|^q\,dx\,du
%\\
%&
\leq c_0^q\sum_{i=1,2}\sum_{j=1}^\infty\int^1_0\int_{\Omega^i_j}
 |\sum_{k=1}^nr_k(u)\CC^i_j(\lambda_k)\zeta^i_j\bg_k|^q\,dx\,du \\
&\leq (c_0r_b)^q\sum_{i=1,2}\sum_{j=1}^\infty\int^1_0\int_{\Omega^i_j}
 |\sum_{k=1}^nr_k(u)\zeta^i_j\bg_k|^q\,dx\,du 
% \\
%&
\leq (c_0^2r_b)^q\sum_{i=0,1}\sum_{j=1}^\infty
\int^1_0\int_{\Omega\cap B^i_j}
 |\sum_{k=1}^nr_k(u)\bg_k|^q\,dx\,du \\
&= (c_0^2r_b)^q\int^1_0 \Bigl(\sum_{i=0,1}\sum_{j-1}^\infty
\int_{\Omega\cap B^i_j}
 |\sum_{k=1}^nr_k(u)\bg_k|^q\,dx\Bigr)\,du 
% \\
%&
\leq (C_qc_0^2r_b)^q \int^1_0\|\sum_{k=1}^nr_k(u)\bg_k\|_{L_q(\Omega)}^q\,du, 
\end{align*}
and so by Lemma \ref{lem:cal} we have 
$$\|\sum_{k=1}^nr_k(u)\CU_1(\lambda_k)\bg_k\|_{L_q(\Omega\times(0,1))}
\leq C_qc_0^2r_b \|\sum_{k=1}^nr_k(u)\bg_k\|_{L_q(\Omega\times(0,1))}.
$$
In this way, we can show \eqref{7.33}.  
Next, since
\begin{align*}
\Delta(\tilde\zeta^i_j\bv^i_j) &= \tilde\zeta^i_j\Delta\bv^i_j
+ 2(\nabla\tilde\zeta^i_j)\nabla\bv^i_j+ (\Delta\tilde\zeta^i_j)\bv^i_j, \\
\nabla\dv(\tilde\zeta^i_j\bv^i_j) &= \tilde\zeta^i_j\nabla\dv\bv^i_j
+(\nabla\zeta^i_j)\dv\bv^i_j + \nabla((\nabla\zeta^i_j)\cdot\bv^i_j),
\\
\nabla(\rho_0(x)\dv(\tilde\zeta^i_j\bv^i_j)) & = 
\tilde\zeta^i_j\rho_0(x)\nabla\dv\bv^i_j 
+\nabla(\rho_0(x)(\nabla\zeta^i_j)\cdot\bv^i_j)
+ (\nabla\zeta^i_j)\rho_0(x)\dv\bv^i_j,
\end{align*}
with $\bv^i_j = \CC^i_j(\lambda)\zeta^i_j\bg$, 
%%%%%%%%%%%%%%%%%%%%
setting 
\begin{equation}\label{residue:1}\begin{split}
&\CV_1(\lambda)\bg = -\sum_{i=1,2}\sum_{j=1}^\infty
\{\mu(
 2(\nabla\tilde\zeta^i_j)\nabla\CC^i_j(\lambda)\zeta^i_j\bg
+ (\Delta\tilde\zeta^i_j)\CC^i_j(\lambda)\zeta^i_j\bg
+\nu(
(\nabla\zeta^i_j)\dv\CC^i_j(\lambda)\zeta^i_j\bg
 + \nabla((\nabla\zeta^i_j)\CC^i_j(\lambda)\zeta^i_j\bg)\\
&\quad +\lambda^{-1}\gamma_1(x)(\nabla(\rho_0(x)(\nabla\zeta^i_j)\cdot
\CC^i_j(\lambda)\zeta^i_j\bg)
+ (\nabla\zeta^i_j)\rho_0(x)\dv\CC^i_j(\lambda)\zeta^i_j\bg)\} \\
&\quad+\sum_{i=1,2}\sum_{j=1}^\infty\tilde\zeta^i_j(
(\rho_0(x)-\rho_0(x^i_j))\lambda\CC^i_j(\lambda)\zeta^i_j\bg
-(\gamma_1(x)\rho_0(x)-\gamma_1(x^i_j)\rho_0(x^i_j))
\lambda^{-1}\nabla\CC^i_j(\lambda)\zeta^i_j\bg)
\end{split}\end{equation}
and setting $\bv = \CU_1(\lambda)\bg$, we have
\begin{equation}\label{7.31}
\rho_0(x)\lambda\bv - \mu\Delta\bv - \nu\nabla\dv\bv
-\gamma_1(x)\lambda^{-1}\nabla(\rho_0(x)\dv\bv)
= \bg + \CV_1(\lambda)\bg \quad\text{in $\Omega$},\quad
\bv|_\Gamma = 0
\end{equation}
because $\Gamma \cap B^1_j = \Gamma_j$ and $\tilde\zeta^i_j
\zeta^i_j = \zeta^i_j$. 
Using \eqref{cond:4*}, 
\eqref{unif:1}, Lemma \ref{lem:cal}, \eqref{7.29},
and  \eqref{7.33}, we obtain 
\begin{equation}
\CR_{\CL(L_q(\Omega)^N)}(\{(\tau\pd_\tau)^\ell 
\CR(\lambda) \mid\lambda \in \Sigma_{\epsilon, \tilde\lambda_0}\})
\leq c_0^2C_q\{(1+2\alpha_4)M_1 + \alpha_4r_b\tilde\lambda_0^{-1/2}\}
\label{residu:3}
\end{equation}
for any $\tilde\lambda_0 \geq \lambda_0$,
where we have assumed that $\tilde\lambda_0\geq1$.  To prove \eqref{residu:3}, 
we have to estimate
$(\nabla\rho_0)(\nabla\zeta^i_j)\cdot\CC^i_j(\lambda)\zeta^i_j\bg$.  
For this purpose, we use the following lemma 
can be proved easily with the help of 
Sobolev's imbedding theorem.
\begin{lem}\label{lem:7.4} Let $1 < q  \leq r < \infty$ and 
$N < r< \infty$.  Then, the following two inequalities hold:
\begin{itemize}
\item[\thetag1]~There exists a constant 
$C$ depending only on $N$, $q$, and $r$ for which 
$$\|ab\|_{L_q(\Omega)} \leq C\|a\|_{L_r(\Omega)}\|b\|_{H^1_q(\Omega)}.
$$
\item[\thetag2]~For any $\sigma>0$, there exists a constant 
$C=C_{\sigma, \|a\|_{L_r(\Omega)}}$ for which
$$\|ab\|_{L_q(\Omega)} \leq \sigma\|b\|_{H^1_q(\Omega)}
+ C\|b\|_{L_q(\Omega)}.
$$
\end{itemize}
\end{lem}
For any $\lambda_k \in \Sigma_{\epsilon, \tilde\lambda_0}$
 and 
$\bg_k \in L_q(\Omega)^N$ ($k=1, \ldots, n$), 
by Lemma \ref{lem:7.4}, \eqref{7.29} and \eqref{cond:4*}, 
\begin{align*}
&\sum_{i=1,2}\sum_{j=1}^\infty\int^1_0\int_\Omega
|\sum_{k=1}^n r_k(u)(\nabla\rho_0)(\nabla\zeta^i_j)\cdot
\CC^i_j(\lambda_k)\zeta^i_j\bg_k|^q\,dx\,du \\
&\leq (c_0\alpha_4)^q
\sum_{i=1,2}\sum_{j=1}^\infty
\int^1_0\|\sum_{k=1}^nr_k(u)\CC^i_j(\lambda_k)\zeta^i_j\bg_k
\|_{H^1_q(\Omega^i_j)}^q\,du \\
&\leq (c_0\alpha_4)^q
\sum_{i=1,2}\sum_{j=1}^\infty
\int^1_0\|\sum_{k=1}^nr_k(u)\lambda_k^{-1/2}
\lambda_k^{1/2}\CC^i_j(\lambda_k)\zeta^i_j\bg_k
\|_{H^1_q(\Omega^i_j)}^q\,du \\
&\leq (c_0\alpha_4\tilde\lambda_0^{-1/2})^q
\sum_{i=1,2}\sum_{j=1}^\infty
\int^1_0\|\sum_{k=1}^nr_k(u)
\lambda_k^{1/2}\CC^i_j(\lambda_k)\zeta^i_j\bg_k
\|_{H^1_q(\Omega^i_j)}^q\,du \\
&\leq (c_0\alpha_4\tilde\lambda_0^{-1/2}r_b)^q
\sum_{i=1,2}\sum_{j=1}^\infty
\int^1_0\|\sum_{k=1}^nr_k(u) \zeta^i_j\bg_k
\|_{L_q(\Omega^i_j)}^q\,du \\
&\leq (c_0^2\alpha_4\tilde\lambda_0^{-1/2}r_b)^q
\sum_{i=1,2}\sum_{j=1}^\infty
\int^1_0\|\sum_{k=1}^nr_k(u)\bg_k
\|_{L_q(\Omega \cap B^i_j)}^q\,du\\
& =(c_0^2\alpha_4\tilde\lambda_0^{-1/2}r_b)^q \int^1_0
(\sum_{i=1,2}\sum_{j=1}^\infty
\|\sum_{k=1}^nr_k(u)\bg_k
\|_{L_q(\Omega \cap B^i_j)}^q)\,du \\
& = (C_qc_0^2\alpha_4\tilde\lambda_0^{-1/2}r_b)^q \int^1_0
\|\sum_{k=1}^nr_k(u)\bg_k\|^q\,du.
\end{align*}
 Other terms can be estimated
similarly, and so by Lemma \ref{lem:cal} and 
\eqref{7.25} we have \eqref{residu:3}.  
Choosing $M_1 > 0$ so small that $c_0^2C_q(1+2\alpha_4)M_1 \leq 1/4$
and choosing $\tilde\lambda_0\geq \max(\lambda_0, 1)$ so large that 
$c_0^2C_q\alpha_4r_b\tilde\lambda_0^{-1/2} \leq 1/4$, by \eqref{residu:3}
 we have 
\begin{equation}
\CR_{\CL(L_q(\Omega)^N)}(\{(\tau\pd_\tau)^\ell 
\CV_1(\lambda) \mid\lambda \in \Sigma_{\epsilon, \tilde\lambda_0}\})
\leq 1/2.
\label{residu:5}
\end{equation}
Thus,  $(\bI + \CV_1(\lambda))^{-1} = \sum_{j=1}^\infty
(-\CV_1(\lambda))^j$ exists  and satisfies the 
estimate:
$$\CR_{\CL(L_q(\Omega)^N)}(\{(\tau\pd_\tau)^\ell
(\bI + \CV_1(\lambda))^{-1}  \mid
\lambda \in \Sigma_{\epsilon, \tilde\lambda_0}\}) \leq 4.
$$
Let $\CC(\lambda) = \CU_1(\lambda)(\bI + \CV_1(\lambda))^{-1}$, 
and then in view of  \eqref{7.31} we see that $\bu = \CC(\lambda)\bg$
is a solution of Eq. \eqref{7.3}. The uniqueness of
solutions follows from the existence of solutions of the 
dual problem.  Moreover, by \eqref{7.33}
and \eqref{residu:5} we see that $\CC(\lambda)$ satisfies
the estimate:
$$\CR_{\CL(L_q(\Omega)^N, H^{2-k}_q(\Omega)^N)}
(\{(\tau\pd_\tau)^\ell(\lambda^{k/2} \CC(\lambda)) \mid 
\lambda \in \Sigma_{\epsilon, \lambda_0}\}) \leq 4C_{N,q}r_b.
$$
This completes the proof of assertion \thetag1 of Theorem \ref{thm:7.2}.

We next prove the assertion \thetag2. By Theorem \ref{thm:7.3} \thetag2 and 
Theorem \ref{thm:7.5} \thetag2, there are $\CR$-bounded solution operators
$\CD^i_j(\lambda)$ with 
\begin{equation}\label{6.7.3}\begin{split}
\CD^1_j(\lambda) \in {\rm Hol}\,(\Sigma_{\epsilon, \lambda_0}, 
\CL(L_q(\BR^N), H^2_q(\BR^N))), \quad
\CD^2_j(\lambda) \in {\rm Hol}\,(\Sigma_{\epsilon, \lambda_0}, 
\CL(\CY_q(\Omega_1), H^2_q(\BR^N)))
\end{split}\end{equation}
such that for any $(h_1, h_2) \in Y_q(\Omega)$ and $\lambda \in 
\Sigma_{\epsilon, \lambda_0}$, $\vartheta^1_j = \CD^1_j(\lambda)\zeta^1_jh_1$
are solutions of the equations:
\begin{equation}\label{6.7.1}
\gamma_3(x^1_j)\lambda\vartheta^1_j - \gamma_4(x^1_j)\Delta\vartheta^1_j 
= \zeta^1_jh_1
\quad\text{in $\BR^N$}, 
\end{equation}
and $\vartheta^2_j = \CD^2_j(\lambda)\zeta^2_j(h_1, \lambda^{1/2}h_2, h_2)$
are solutions of the equations:
\begin{equation}\label{6.7.2}
\gamma_3(x^2_j)\lambda\vartheta^2_j - \gamma_4(x^2_j)\Delta\vartheta^2_j 
= \zeta^2_jh_1
\quad\text{in $\Omega_j$}, \quad (\nabla\vartheta^2_j)\cdot\bn_j|_{\Gamma_j}=0
\end{equation}
where $\bn_j$ is the unit outer normal to $\Gamma_j$. Notice that 
$\bn_j = \bn$ on $\Gamma_j \cap B^2_j = \Gamma \cap B^2_j$. 
In particular, by \eqref{6.7.3} we have
\begin{equation}\label{6.7.4}
\sum_{k=0}^2 |\lambda|^{k/2}\|\vartheta^i_j\|_{H^{2-k}_q(\Omega^i_j}
\leq r_b\{\|\zeta^i_jh_1\|_{L_q(\Omega^i_j)}
+ \sigma^i(\|\lambda^{1/2}h_2\|_{L_q(\Omega^2_j)}
+ \|h_2\|_{H^1_q(\Omega^2_j)})\} \quad(i=1,2), 
\end{equation}
where $\sigma^1=0$ and $\sigma^2=1$. Let 
\begin{align*}
\bU_2(\lambda)(h_1, h_2)  = \sum_{i=1,2}\sum_{j=1}^\infty
\tilde\zeta^i_j\vartheta^i_j, \quad
\CU_2(\lambda)F = \sum_{j=1}^\infty\tilde\zeta^1_j\CD^1_j(\lambda)
\zeta^1_jF_1
+ \sum_{j=1}^\infty \tilde\zeta^2_j\CD^2_j(\lambda)\zeta^2_jF
\end{align*}
for $(h_1, h_2) \in Y_q(\Omega)$ and $F=(F_1, F_2, F_3) \in \CY_q(\Omega)$. 
By Lemma \ref{lem:cal} and \eqref{7.25}, we have 
\begin{equation}
\sum_{k=0}^2|\lambda|^{k/2}\|\bU_2(\lambda)(h_1, h_2)\|_{H^{2-k}_q(\Omega)}
\leq C_{N,q}r_b(\|h_1\|_{L_q(\Omega)}
+ |\lambda|^{1/2}\|h_2\|_{L_q(\Omega)} +
\|h_2\|_{H^1_q(\Omega)})
\label{6.7.4.1}
\end{equation}
for any $\lambda \in \Sigma_{\epsilon, \lambda_0}$
and $(h_1, h_2) \in Y_q(\Omega)$, and 
\begin{equation}\label{6.7.4.2}
\CR_{\CL(\CY_q(\Omega), H^{2-k}_q(\Omega))}
(\{(\tau\pd_\tau)^\ell(\lambda^{k/2}\CU_2(\lambda)) \mid
\lambda \in \Sigma_{\epsilon, \lambda_0}\}) \leq C_{N,q}r_b
\end{equation}
for $k = 0,1,2$ and $\ell=0,1$. For $F = (F_1, F_2, F_3) \in \CY_q(\Omega)$,
let 
\begin{align*}
\CV_{21}(\lambda)F & = -\sum_{i=1,2}\sum_{j=1}^\infty
\{\dv(\gamma_4(x)(\nabla\zeta^i_j)\CD^i_j(\lambda)\zeta^i_jF
+ \nabla(\gamma_4\zeta^i_j)\cdot\nabla(\CD^i_j(\lambda)\zeta^i_jF)\}\\
& + \sum_{i=1,2}\sum_{j=1}^\infty\zeta^i_j\{(\gamma_3(x)-\gamma_3(x^i_j))
\lambda \CD^i_j(\lambda)\zeta^i_jF - \dv((\gamma_4(x)-\rho(x^i_j))
\nabla\CD^i_j(\lambda)\zeta^i_jF)\},\\
\CV_{22}(\lambda)F & = -\sum_{j=1}^\infty(\nabla\zeta^2_j)\cdot\bn_j
\CD^2_j(\lambda)\zeta^i_jF,
\end{align*}
where we have set $\CD^1_j(\lambda)\zeta^i_jF = \CD^1_j(\lambda)\zeta^i_jF_1$.
We then have
\begin{equation}\label{6.7.5}\begin{split}
\gamma_3(x)\lambda\bU_2(\lambda)(h_1, h_2)
-\dv(\gamma_4(x)\nabla\bU_2(\lambda)(h_1, h_2))
= h_1 + \CV_{21}(\lambda)F_\lambda(h_1, h_2)
&\quad\text{in $\Omega$}, \\
(\nabla\bU_2(\lambda)(h_1, h_2))\cdot\bn
= h_2 + \CV_{22}(\lambda)F_\lambda(h_1, h_2)
&\quad 
\text{on $\Gamma$}
\end{split}\end{equation}
for any $(h_1, h_2) \in Y_q(\Omega)$, where we have
set $F_\lambda(h_1, h_2)
= (h_1, \lambda^{1/2}h_2, h_2) \in \CY_q(\Omega)$. 
Since 
$$\|(\nabla\gamma_4)\cdot\nabla\CD^i_j(\lambda)\zeta^i_jF\|_{L_q(\Omega)}
\leq \sigma\|\nabla\CD^i_j(\lambda)\zeta^i_jF\|_{H^1_q(\Omega)}
+ C_{\sigma, \alpha_4}\|\nabla\CD^i_j(\lambda)\zeta^i_jF\|_{L_q(\Omega)}
$$
as follows from Lemma \ref{lem:7.4} \thetag2, 
by \eqref{6.7.3}, Lemma \ref{lem:cal}, \eqref{unif:1},
\eqref{7.25} and Lemma \ref{lem:7.4}, we have
$$\CR_{\CL(\CY_q(\Omega))}
(\{(\tau\pd_\tau)^\ell(F_\lambda(\CV_{21}(\lambda), \CV_{22}(\lambda)))
\mid \lambda \in \Sigma_{\epsilon, \tilde\lambda_0}\})
\leq \{2M_1+\sigma + c_0^2C_qC_{\sigma, \alpha_4}\tilde\lambda_0^{-1/2}\}r_b
$$
for any $\tilde\lambda_0 \geq \max(\lambda_0, 1)$.
Choosing $M_1$ and $\sigma>0$ so small that $2M_1r_b < 1/8$, 
$\sigma r_b < 1/8$ and choosing $\tilde\lambda_0$ so large
that $c_0^2C_qC_{\sigma, \alpha_4}r_b\tilde\lambda_0^{-1/2}
\leq 1/4$, we have
\begin{equation}\label{6.7.6}
\CR_{\CL(\CY_q(\Omega))}(\{(\tau\pd_\tau)^\ell(
F_\lambda(\CV_{21}(\lambda), \CV_{22}(\lambda)))
\mid \lambda \in \Sigma_{\epsilon, \tilde\lambda_0}\})
\leq 1/2,
\end{equation}
and so $(\bI + F_\lambda(\CV_{21}(\lambda), \CV_{22}(\lambda)))^{-1}
= \sum_{j=0}^\infty(-F_\lambda(\CV_{21}(\lambda), \CV_{22}(\lambda)))^j$
exists and 
\begin{equation}\label{6.7.7}
 \CR_{\CL(\CY_q(\Omega))}(\{(\tau\pd_\tau)^\ell
(\bI + F_\lambda(\CV_{21}(\lambda), \CV_{22}(\lambda)))^{-1}
\mid \lambda \in \Sigma_{\epsilon, \tilde\lambda_0}\})
\leq 4.
\end{equation}
On the other hand, by \eqref{6.7.6} we have
$$\|F_\lambda(\CV_{21}(\lambda)F_\lambda(h_1, h_2), 
\CV_{22}(\lambda)F_\lambda(h_1, h_2))\|_{\CY_q(\Omega)}
\leq (1/2)\|F_\lambda(h_1, h_2)\|_{L_q(\Omega)}
$$
for any $\lambda \in \Sigma_{\epsilon, \tilde\lambda_0}$.  Since
$\|F_\lambda(h_1, h_2)\|_{\CY_q(\Omega)} =
\|h_1\|_{L_q(\Omega)} + |\lambda|^{1/2}\|h_2\|_{L_q(\Omega)}
+ \|h_2\|_{H^1_q(\Omega)}$ gives equivalent norms in
$Y_q(\Omega)$, we see that 
for each $\lambda \in \Sigma_{\epsilon, \tilde\lambda_0}$, $\bI 
+ (\CV_{21}(\lambda),\CV_{22}(\lambda))F_\lambda)^{-1} = \sum_{j=0}^\infty
(-(\CV_{21}(\lambda),\CV_{22}(\lambda))F_\lambda)^j$ exists as an operator
in $\CL(Y_q(\Omega))$ whose operator norm does not exceed $2$. 
Thus, in view of \eqref{6.7.5}, $\vartheta
= \bU_2(\lambda)(\bI + (\CV_{21}(\lambda), 
\CV_{22}(\lambda)))^{-1}(h_1, h_2)$ is a solution of 
Eq. \eqref{7.4}.  The uniqueness of solution follows from
the existence of solutions for the dual problem.  
Notice that $\CU_2(\lambda)F_\lambda(h_1, h_2) = \bU_2(\lambda)(h_1, h_2)$.
We then define an operator $\CD(\lambda)$ by 
$$\CD(\lambda)F = \CU_2(\lambda)(\bI + F_\lambda(\CV_{21}(\lambda),
\CV_{22}(\lambda)))^{-1}
$$
for $F = (F_1, F_2, F_3) \in \CY_q(\Omega)$.  Since
\begin{align*}
(\bI + F_\lambda(\CV_{21}(\lambda), \CV_{22}(\lambda)))^{-1}F_\lambda
&= \sum_{j=0}^\infty(-F_\lambda(\CV_{21}(\lambda), \CV_{22}(\lambda)))^j
F_\lambda
= F_\lambda\sum_{j=0}^\infty
(-(\CV_{21}(\lambda), \CV_{22}(\lambda))F_\lambda)^j \\
& = F_\lambda(\bI + (\CV_{21}(\lambda), \CV_{22}(\lambda))F_\lambda)^{-1},
\end{align*}
we have
\begin{align*}
\vartheta & = \bU_2(\lambda)
(\bI + (\CV_{21}(\lambda), \CV_{22}(\lambda))F_\lambda)^{-1}(h_1, h_2)
%\\
%& 
= \CU_2(\lambda)F_\lambda(\bI + (\CV_{21}(\lambda), \CV_{22}(\lambda))
F_\lambda)^{-1}(h_1, h_2) \\
& = \CU_2(\lambda)(\bI + F_\lambda(\CV_{21}(\lambda), 
\CV_{22}(\lambda)))^{-1}F_\lambda(h_1, h_2)
%\\
%& 
= \CD(\lambda)F_\lambda(h_1, h_2)
%\\
%& 
= \CD(\lambda)(h_1, \lambda^{1/2}h_2, h_2).
\end{align*}
By \eqref{6.7.4.2} and \eqref{6.7.7}, we have
$$\CR_{\CL(\CY_q(\Omega), H^{2-k}_q(\Omega))}
(\{(\tau\pd_\tau)^\ell(\lambda^{k/2}\CD(\lambda)) \mid
\lambda \in \Sigma_{\epsilon, \tilde\lambda_0}\})
\leq 4C_{N,q}r_b.
$$
This completes the proof of the assertion \thetag2 of Theorem \ref{thm:7.2}.
%%%%%%%%%%%%%%%%%%%%%
%%%%%%%
\subsection{Proof of Theorem \ref{thm:7.1}}\label{subsec:7.4}

Let $\CC(\lambda)$ and $\CD(\lambda)$ be the operators given in
Theorem \ref{thm:7.2}.  Let $\vartheta_0 
= \CD(\lambda)(0, \lambda^{1/2}h_2, h_2)$, and then the third equation of
Eq. \eqref{7.1} and the boundary condition for $\vartheta$ are reduced 
to the equations:
\begin{equation}\label{eq:7.1.1}
\gamma_3(x)\lambda\varphi + \gamma_2(x)\dv\bv - \dv(\gamma_4(x)\nabla\varphi)
= f_3 \quad\text{in $\Omega$}, \quad 
(\nabla\varphi)\cdot\bn|_\Gamma =0.
\end{equation}
Thus, in view of  \eqref{7.2} and \eqref{eq:7.1.1}, instead of 
\eqref{7.1} we consider the equations:
\begin{equation}\label{eq:7.1.2}\left\{\begin{aligned}
\rho_0(x)\lambda\bv - \mu\Delta\bv-\nu\nabla\dv\bv
-\gamma_1(x)\lambda^{-1}\nabla(\rho_0(x)\dv\bv) + \gamma_2(x)\nabla\varphi
&=\bff &\quad&\text{in $\Omega$}, \\
\gamma_3(x)\lambda\varphi + \gamma_2(x)\dv\bv - \dv(\gamma_4(x)
\nabla\varphi) & = g&\quad&\text{in $\Omega$}, \\
\bv=0, \quad (\nabla\varphi)\cdot\bn& = 0 &\quad&\text{on $\Gamma$}.
\end{aligned}\right.
\end{equation}
In the following, we write $\CD(\lambda)(g, 0, 0)$ simply by 
$\CD(\lambda)g$. 
Let 
$$\bv = \CC(\lambda)\bff, \quad \varphi = \CD(\lambda)g
$$
in \eqref{eq:7.1.2}, and then we have
\begin{equation}\label{eq:7.1.2*}\left\{\begin{aligned}
\rho_0(x)\lambda\bv - \mu\Delta\bv-\nu\nabla\dv\bv
-\gamma_1(x)\lambda^{-1}\nabla(\rho_0(x)\dv\bv) + \gamma_2(x)\nabla\varphi
&=\bff + \CE_1(\lambda)(\bff, g)
&\quad&\text{in $\Omega$}, \\
\gamma_3(x)\lambda\varphi + \gamma_2(x)\dv\bv - \dv(\gamma_4(x)
\nabla\varphi) & = g + \CE_2(\lambda)(\bff, g)
&\quad&\text{in $\Omega$}, \\
\bv=0, \quad (\nabla\varphi)\cdot\bn& = 0 &\quad&\text{on $\Gamma$},
\end{aligned}\right.
\end{equation}
where we have set 
$$\CE_1(\lambda)(\bff, g) = \gamma_3(x)\nabla\CD(\lambda)g,\quad
\CE_2(\lambda)(\bff, g) = \gamma_2(x)\dv\CC(\lambda)\bff.
$$
Let $\CE(\lambda)(\bff, g) 
= (\CE_1(\lambda)(\bff, g), \CE_2(\lambda)(\bff, g))$,
and then by Theorem \ref{thm:7.2} we have 
\begin{equation}\label{eq:7.1.4}
\CR_{\CL(L_q(\Omega)^{N+1})}(\{(\tau\pd_\tau)^\ell\CE(\lambda)
\mid \lambda \in \Sigma_{\epsilon, \lambda_1}\}) 
\leq r_b\lambda_1^{-1/2}
\end{equation}
for any $\lambda_1 \geq \lambda_0$.  Thus, choosing $\lambda_1 > 0$
so large that $r_b\lambda_1^{-1/2} \leq 1/2$, by
\eqref{eq:7.1.4} and Lemma \ref{lem:7.3.4} we see that
$(\bI + \CE(\lambda))^{-1} = \sum_{j=0}^\infty
(-\CE(\lambda))^j$ exists and satisfies the estimate:
\begin{equation}\label{eq:7.1.5}
\CR_{\CL(L_q(\Omega)^{NL+1})}(\{(\tau\pd_\tau)^\ell
(\bI + \CE(\lambda))^{-1} \mid
\lambda \in \Sigma_{\epsilon, \lambda_1}\}) \leq 4.
\end{equation}
Let $\tilde\CB_1(\lambda) = \CC(\lambda)(\bI + \CE(\lambda))^{-1}$
and $\tilde\CB_2(\lambda) = \CD(\lambda)(\bI + \CE(\lambda))^{-1}$, and then
by Theorem \ref{thm:7.2}, \eqref{eq:7.1.5}, and Lemma \ref{lem:7.3.4}, 
we see that for any $\lambda \in \Sigma_{\epsilon, \lambda_1}$
and $(\bff, g) \in L_q(\Omega)^{N+1}$, $\bv = \tilde\CB_1(\lambda)(\bff, g)$ 
and $\varphi = \tilde\CB_2(\lambda)(\bff, g)$ are solutions of 
\eqref{eq:7.1.2*} and 
\begin{equation}\label{eq:7.1.6}
\CR_{\CL(L_q(\Omega)^{N+1}, H^{2-k}_q(\Omega)^{N+1})}
(\{(\tau\pd_\tau)^\ell(\lambda^{k/2}(\tilde\CB_1(\lambda), 
\tilde\CB_2(\lambda)))\mid \lambda \in \Sigma_{\epsilon, \lambda_1}\})
\leq 4r_b.
\end{equation}
Finally, setting
\begin{align*}
\bv & = \tilde\CB_1(\lambda)(\bff_2-\lambda^{-1}\gamma_1(x)\nabla f_1, f_3), \\
\zeta & = \lambda^{-1}(f_1 - \rho_0(x)\dv\bv), \\
\vartheta & = \tilde\CB_2(\lambda)(\bff_2-\lambda^{-1}\gamma_1(x)\nabla f_1,
f_3) + \CD(\lambda)(0, \lambda^{1/2}f_4, f_4)
\end{align*}
we see that $\zeta$, $\bv$, and $\vartheta$ are solutions of Eq. \eqref{7.1}.
The uniqueness of solutions follows from the existence of solutions 
for the dual problem.  For $F = (F_1, F_2, F_3, F_4, F_5) \in \CX_q(\Omega)$, 
we set 
\begin{align*}
\CB_1(\lambda)F& = \tilde\CB_1(\lambda)(F_2, F_3)
-\lambda^{-1}\tilde\CB_1(\lambda)(\gamma_1(x)\nabla F_1, 0),  \\
\CA(\lambda)F & = \lambda^{-1}F_1 - \lambda^{-1}\rho_0(x)\dv
\CB_1(\lambda)F, 
\\
\CB_2(\lambda)F & = \tilde\CB_2(\lambda)(F_2, F_3) 
- \lambda^{-1}\tilde\CB_2(\lambda)(\gamma_1(x)\nabla F_1, 0)
+ \CD(\lambda)(0, F_4, F_5),
\end{align*}
and then, we have 
$\zeta = \CA(\lambda)\bF_\lambda$, $\bv= \CB_1(\lambda)\bF_\lambda$,
and $\vartheta = \CB_2(\lambda)\bF_\lambda$, where $\bF_\lambda
= (f_1, \bff_2, f_3, \lambda^{1/2}f_4, f_4)$. Moreover, 
by Lemma \ref{lem:7.3.5}, 
\eqref{cond:4*} and \eqref{eq:7.1.6} we have
\begin{align*}
\CR_{\CL(\CX_q(\Omega), H^{2-k}_q(\Omega)^N)}
(\{(\tau\pd_\tau)^\ell(\lambda^{k/2}\CB_1(\lambda)) \mid
\lambda \in \Sigma_{\epsilon, \lambda_1}\}) & \leq (4+ \lambda_1^{-1}
\alpha_4)r_b, \\
\CR_{\CL(\CX_q(\Omega), H^{2-k}_q(\Omega))}
(\{(\tau\pd_\tau)^\ell(\lambda^{k/2}\CB_2(\lambda)) \mid
\lambda \in \Sigma_{\epsilon, \lambda_1}\}) & \leq (4+ \lambda_1^{-1}
\alpha_4)r_b, \\
\CR_{\CL(\CX_q(\Omega), H^1_q(\Omega))}
(\{(\tau\pd_\tau)^\ell(\lambda \CA(\lambda)) \mid
\lambda \in \Sigma_{\epsilon, \lambda_1}\}) & \leq 1 
+(4+ \lambda_1^{-1}\alpha_4)r_b.
\end{align*}
This completes the proof of Theorem \ref{thm:7.1}.
%%%%%%%%%%%%%%
%%%%%%%
\subsection{Proof of Theorem \ref{thm:linear:1}} \label{subsec:7.5}

We first prove the generation of a $C_0$ analytic semigroup associated
with Eq. \eqref{linear:1}.  Let 
\begin{equation}\label{p.7.0}\begin{split}
&\CD_q(\Omega)  = \{(\zeta, \bv, \vartheta) \in D_q(\Omega)
\mid \bv|_\Gamma = 0, \quad (\nabla\vartheta)\cdot\bn|_{\Gamma} = 0\}, \\
&A(\zeta, \bv, \vartheta) 
= \left(\begin{matrix} 
-\rho_0(x)\dv\bv \\
\rho_0(x)^{-1}(\mu\Delta\bv+ \nu\nabla\dv\bv-\gamma_1(x)\nabla\zeta
-\gamma_2(x)\nabla\vartheta) \\
\gamma_3(x)^{-1}(-\gamma_2(x)\dv\bv + \dv(\gamma_4(x)\nabla\vartheta))
\end{matrix}\right), \\
&\CA_q(\zeta, \bv, \vartheta) = A(\zeta, \bv, \vartheta)
\quad\text{for $(\zeta, \bv, \vartheta) \in \CD_q(\Omega)$}.
\end{split}\end{equation} 
And then, Eq. \eqref{linear:1} with$f_1=\bff_2=f_3=g=0$ is 
formally written as 
\begin{equation}\label{p.7.1} \pd_t U - \CA_q U = 0
\quad\text{for $t > 0$}, \quad 
U|_{t=0} = U_0,
\end{equation}
where $U_0 = (\zeta_0, \bv_0, \vartheta_0) \in \CH_q(\Omega)$
and $U$ with
$$U \in C^0[0, \infty, \CH_q(\Omega))
\cap C^0((0, \infty), \CD_q(\Omega) \cap C^1((0, \infty), \CH_q(\Omega)).
$$ 
The resolvent equation 
corresponding to \eqref{p.7.1} is 
\begin{equation}\label{p.7.2}
\lambda V- \CA_q V = F \quad \text{in $\Omega$},
\end{equation}
where $F = (f_1, \bff_2, f_3) \in \CH_q(\Omega)$ 
and $V \in \CD_q(\Omega)$. By Theorem \ref{thm:7.1}, we 
see that the resolvent set $\rho(\CA_q)$ of $\CA_q$ contains 
$\Sigma_{\epsilon, \lambda_0}$ and 
for any $F \in \CH_q(\Omega)$ and $\lambda \in \Sigma_{\epsilon, \lambda_0}$, 
$V = (\lambda\bI - \CA_q)^{-1}F
\in \CD_q(\Omega)$ satisfies the 
estimate:
\begin{equation}\label{res.est:1}
|\lambda|\|V\|_{\CH_q(\Omega)} + \|V\|_{\CD_q(\Omega)}
\leq r_b\|F\|_{\CH_q(\Omega)}
\end{equation}
where 
$$\|F\|_{\CH_q(\Omega)} = \|f_1\|_{H^1_q(\Omega)} + 
\|(\bff_2, f_3)\|_{L_q(\Omega)},\quad
\|V\|_{\CD_q(\Omega)} = \|\zeta\|_{H^1_q(\Omega)}
+ \|(\bv, \vartheta)\|_{H^2_q(\Omega)}
$$
for $F = (f_1, \bff_2, f_3) \in \CH_q(\Omega)$ and 
$V = (\zeta, \bv, \vartheta) \in \CD_q(\Omega)$.
Since $0 < \epsilon < \pi/2$,  
the operator $\CA_q$ generates a $C_0$ analytic semigroup
$\{T(t)\}_{t\geq0}$ on $\CH_q(\Omega)$ possessing the estimate
$$\|T(t)F\|_{\CH_q(\Omega)} \leq Ce^{\gamma t}\|F\|_{\CH_q(\Omega)}
\quad(t > 0)
$$
for some constants $C$ and $\gamma$. 

We now consider the maximal $L_p$-$L_q$ regularity for
Eq. \eqref{linear:1} in the case that  $f_1=\bff_2=f_3=g=0$.
Let 
\begin{equation}\label{p.7.3}
\CE_{p,q}(\Omega) = (\CH_q(\Omega), \CD_q(\Omega))_{1-1/p, p}.
\end{equation}
Notice that $\CE_{p,q}(\Omega) \subset D_{p,q}(\Omega)$,
and that for $(\zeta, \bv, \vartheta) \in \CE_{p,q}(\Omega)$ we have 
\begin{equation}\label{p.7.4} 
\bv|_\Gamma=0 \quad\text{for $2/p + 1/q < 2$}, \quad
(\nabla\vartheta)\cdot\bn|_{\Gamma} = 0
\quad\text{for $2/p + 1/q < 1$}. 
\end{equation}
By real interpolation theory, we have 
\begin{thm}\label{thm:p.1} Let $1 < p, q < \infty$.  Assume that
$\Omega$ is a uniformly $C^2$ domain.  Then, 
for $(\zeta_0, \bv_0, \vartheta_0) \in \CE_{p,q}(\Omega)$, 
$(\zeta, \bv, \vartheta) = 
T(t)(\zeta_0, \bv_0, \vartheta_0)$ satisfies Eq. \eqref{linear:1} 
with $f_1=\bff_2=f_3=g= 0$ and 
possesses the estimate:
\begin{align*}
&\|e^{-\gamma t}(\zeta, \bv, \vartheta)\|_{H^1_p((0, \infty), \CH_q(\Omega))}
+ \|e^{-\gamma t}(\bv, \vartheta)\|_{L_p((0, \infty), H^2_q(\Omega))} 
\leq C\|(\zeta_0, \bv_0, \vartheta_0)\|_{D_{p,q}(\Omega)}.
\end{align*}
\end{thm}
\begin{remark} Theorem \ref{thm:p.1} can be shown employing the same argument as that in the proof of
Theorem 3.9 in Shibata-Shimizu \cite{SS2}, so we may omit the proof. 
\end{remark}

We next consider Eq. \eqref{linear:1} in the case that 
$(\zeta_0, \bv_0, \vartheta_0) = 0$.  Notice that
$g$ is defined on $\BR$ with respect to $t$. We extend
$f_1$, $\bff_2$,  and $f_3$ to $\BR$ as follows:
$$f_{10}(\cdot, t) = \begin{cases} f_1(\cdot, t) \enskip&t \in (0, T), \\
0\enskip&t \not\in(0,T),\end{cases} \quad
\bff_{20}(\cdot, t) = \begin{cases} \bff_2(\cdot, t) \enskip&t \in (0, T), \\
0\enskip&t \not\in(0,T),\end{cases}
\quad
f_{30}(\cdot, t) = \begin{cases} f_3(\cdot, t) \enskip&t \in (0, T), \\
0\enskip&t \not\in(0,T).\end{cases}
$$
We then consider the following equations:
\begin{equation}\label{p.7.5}\pd_tV_1 - A V_1 = (f_{10}, \bff_{20}, f_{30})
\quad\text{in $\Omega\times\BR$}, \quad 
\bv_1=0, \quad (\nabla\vartheta_1)\cdot\bn=g \quad
\text{on $\Gamma \times\BR$},
\end{equation}
where $V_1 = (\zeta_1, \bv_1, \vartheta_1)$. We use the Laplace transform
$\CL$  with respect to $t$ and its inversion formula $\CL^{-1}$, which
are defined by
\begin{align*}
\CL[f](\cdot, \lambda) &= \int^\infty_{-\infty}e^{-\lambda t}f(\cdot, t)\,dt
= \CF[e^{-\gamma t}f](\tau) \quad(\lambda = \gamma + i\tau \in \BC), \\
\CL^{-1}[g](\cdot,t) & =\frac{1}{2\pi}\int^\infty_{-\infty} e^{\lambda t}
g(\cdot, \tau)\,d\tau = e^{\gamma t}\CF^{-1}[g](\cdot, t).
\end{align*}
where $\CF$ and $\CF^{-1}$ denote the Fourier transform with respect to
$t$ and its inverse. Applying the Laplace transform
to Eq. \eqref{p.7.5}, we have 
\begin{equation}\label{p.7.6} \lambda\hat V_1 - A\hat V_1
=(\CL[f_{10}], \CL[\bff_{20}], \CL[f_{30}]) \quad\text{in $\Omega$},
\quad \hat \bv_1=0,\quad (\nabla\hat\vartheta_1)\cdot\bn=
\CL[g] \quad\text{on $\Gamma$}.
\end{equation}
Let $\CS(\lambda) = (\CA(\lambda), \CB_1(\lambda), \CB_2(\lambda))$
be  $\CR$ bounded solution operators given in Theorem \ref{thm:7.1}.
We then have $\hat V_1(\lambda) = \CS(\lambda)\bF_\lambda$, where
we have set 
$$\bF_\lambda = (\CL[f_{10}](\lambda), \CL[\bff_{20}](\lambda),
\CL[f_{30}](\lambda), 
\lambda^{1/2}\CL[g](\lambda), \CL[g](\lambda)).
$$
We now introduce an operator $\Lambda^{1/2}_\gamma$ by 
$$\Lambda_\gamma^{1/2}g = \CL^{-1}[\lambda^{1/2}\CL[g](\lambda)].
$$
Since
$$|(\tau\pd_\tau)^\ell(\lambda^{1/2}/(1+\tau^2)^{1/4})|
\leq C_\gamma
$$
for any $\lambda = \gamma + i\tau \in \BC$ with some constant
$C_\gamma$ depending solely on $\gamma \in \BR$, 
by Bourgain theorem (cf. Lemma \ref{lem:7.3.5}), we have 
\begin{equation}\label{p.7.7}
\|e^{-\gamma t}\Lambda_\gamma^{1/2}g\|_{L_p(\BR, L_q(\Omega))}
\leq C_\gamma\|e^{-\gamma t }g\|_{H^{1/2}_p(\BR, L_q(\Omega))}
\end{equation}
for any $\gamma > 0$. Since $\lambda^{1/2}\CL[g](\lambda)
= \CL[\Lambda_\gamma^{1/2}g](\lambda)$, using Theorem \ref{thm:7.1}
and Weis' operator valued Fourier multiplier theorem \cite{Weis},
we have
\begin{multline*}
\|e^{-\gamma t}(\zeta_1, \bv_1, \vartheta_1)\|_{H^1_p(\BR, \CH_q(\Omega))}
+ \|e^{-\gamma t}(\bv_1 \vartheta_1)\|_{L_p(\BR, H^2_q(\Omega))} \\
\leq r_b(\|e^{-\gamma t}(f_{10}, \bff_{20}, f_{30})
\|_{L_p(\BR, \CH_q(\Omega))} 
+ \|e^{-\gamma t}\Lambda_\gamma^{1/2}g\|_{L_p(\BR, L_q(\Omega))}
+ \|e^{-\gamma t}g\|_{L_p(\BR, H^1_q(\Omega))}),
\end{multline*}
which, combined with \eqref{p.7.7}, leads to 
\begin{equation}\label{p.7.8} \begin{split}
&\|(\zeta_1, \bv_1, \vartheta_1)\|_{H^1_p((0, T), \CH_q(\Omega))}
+ \|(\bv_1, \vartheta_1)\|_{L_p((0, T), H^2_q(\Omega))} \\
&\quad\leq r_b e^{\gamma T}(
\|(f_1, \bff_2, f_3)\|_{L_p((0, T),\CH_q(\Omega))} 
+ C_\gamma \|e^{-\gamma t}g\|_{H^{1/2}_p(\BR, L_q(\Omega))} 
+ \|e^{-\gamma t}g\|_{L_p(\BR, H^1_q(\Omega))}).
\end{split}\end{equation}
Finally, let $V_2$ be a solution of the system:
$$\left\{\begin{aligned}
\pd_t V_2 - AV_2 &= 0 &\qquad&\text{in $\Omega\times(0, \infty)$}, \\
\bv_2=0, \quad(\nabla\vartheta_2)\cdot\bn &= 0
&\qquad&\text{on $\Gamma\times(0, \infty)$}, \\
V_2 &= V_0-V_1|_{t=0}&\qquad&\text{in $\Omega$}.
\end{aligned}\right.$$
By the compatibility condition, $V_0-V_1|_{t=0} \in 
\CD_{p,q}(\Omega)$ provided that $2/p + 1/q \not=1$ and 
$2/p + 1/q \not=1$.  Thus, by Theorem \ref{thm:p.1}
we see that $V_2=(\zeta_2, \bv_2, \vartheta_2)$ 
exists and satisfies the following estimate:
\begin{align*}
&\|e^{-\gamma t}(\zeta_2, \bv_2, \vartheta_2)
\|_{H^1_p((0, \infty), \CH_q(\Omega))} 
+ \|e^{-\gamma t}(\bv_2, \vartheta_2)\|_{L_p((0, \infty), 
H^2_q(\Omega))}\\
&\quad  \leq 
C(\|(\zeta_0-\zeta_1|_{t=0}, 
\bv_0-\bv_1|_{t=0}, \vartheta_0-\vartheta_1|_{t=0})
\|_{D_{p,q}(\Omega)}).
\end{align*}
By real interpolation theorem, we have
$$\|(\bv_1|_{t=0}, \vartheta_1|_{t=0})\|_{B^{2(1-1/p)}_{q,p}(\Omega)}
\leq C(\|e^{-\gamma t}\pd_t(\bv_1, \vartheta_1)
\|_{L_p((0, \infty), L_q(\Omega))}
+ \|e^{-\gamma t}(\bv_1, \vartheta_1)\|_{L_p((0, \infty), H^2_q(\Omega))})
$$
because $e^{-\gamma t}(\bv_1, \vartheta_1)|_{t=0}
= (\bv_1|_{t=0}, \vartheta_1|_{t=0})$.  Putting
$\zeta = \zeta_1 + \zeta_2$, $\bv = \bv_1 + \bv_2$ and 
$\vartheta = \vartheta_1 + \vartheta_2$, we see that
$\zeta$, $\bv$ and $\vartheta$ are required solutios of Eq. \eqref{linear:1}.
The uniqueness follows from the existence of solutions for the dual problem
(cf. Shibata-Shimizu \cite[Proof of Theorem 4.3]{SS2}). 
This completes the proof of Theorem \ref{thm:linear:1}. 
%%%%%%%%%%%%%
\section{Decay Estimate -- proof of Theorem \ref{thm:decay1}}
\label{sec:8}
To prove Theorem \ref{thm:decay1}, we first prove the existence of a 
$C^0$ analytic semigroup associated with Eq. \eqref{newlinear:1} that
is exponentially stable.  For this purpose, we consider the resolvent
problem:
\begin{equation}\label{8.1} \left\{\begin{aligned}
\lambda\zeta + a_{0*}\dv\bv &= f_1&\quad&\text{in $\Omega$}, \\
\lambda\bv - a_{0*}^{-1}(\mu\Delta\bv + \nu\nabla\dv\bv
-a_{1*}\nabla\zeta-a_{2*}\nabla \vartheta) & = \bff_2
&\quad&\text{in $\Omega$}, \\
\lambda\vartheta + a_{3*}^{-1}
(a_{2*}\dv\bv - a_{4*}\Delta\vartheta) & = f_3
&\quad&\text{in $\Omega$}, \\
\bv|_{\Gamma}=0, \quad (\nabla\vartheta)\cdot\bn|_\Gamma
& = 0.
\end{aligned}\right.\end{equation}
We shall prove 
\begin{thm}\label{thm:8.1} Let $1 < q <\infty$ and $0 < \epsilon < \pi/2$. 
Assume that $\Omega$ is a bounded domain in $\BR^N$ $(N \geq 2)$ whose
boundary $\Gamma$ is a compact hypersurface of $C^2$ class.
Assume that $a_{0*}$, $a_{1*}$, $\mu$, $\nu$, $a_{3*}$, and 
$a_{4*}$ are positive constans and that $a_{2*}$ is a 
non-zero constant. Let 
\begin{equation}\label{space:8.1}
\hat \CH_q(\Omega) = \{(f_1, \bff_2, f_3) \in \CH_q(\Omega) \mid 
\int_\Omega f_1\,dx = \int_\Omega f_3\,dx = 0\}.
\end{equation}
Set $\bC_+ = \{\lambda \in \BC \mid {\rm Re}\,\lambda \geq 0$\}.
Then, for any $\lambda \in \bC_+$ and $(f_1, \bff_2, f_3) \in 
\hat H_q(\Omega)$, problem \eqref{8.1} admits a unique solution
$U = (\zeta, \bv, \vartheta) \in D_q(\Omega)\cap \hat \CH_q(\Omega)$
possessing the esitmate:
\begin{equation}\label{8.2} (|\lambda|+1)\|(\zeta, 
\bv, \vartheta)\|_{\CH_q(\Omega)}
+ \|(\bv, \vartheta)\|_{H^2_q(\Omega)} \leq 
C\|(f_1, \bff_2, f_3)\|_{\CH_q(\Omega)}.
\end{equation}
\end{thm}
\begin{proof}
Employing the same argument as that in the proof of Theorem
\ref{thm:7.1}, we can prove the existence of $\CR$-bounded solution 
operators corresponding to Eq. \eqref{8.1}, and so there exists 
$\lambda_0 \geq 1$ such that for any 
$\lambda \in \Sigma_{\epsilon, \lambda_0}$ and $(f_1, \bff_2, f_3) 
\in \CH_q(\Omega)$, problem \eqref{8.1} admits a unique solution 
$(\zeta, \bv, \vartheta) \in D_q(\Omega)$ possessing the estimate
\eqref{8.2}.  Moreover, if $f_1$ and $f_3$ satisfy zero average condition, 
then $\zeta$ and $\vartheta$ also satisfy this
condition in the case that $\lambda\not=0$, what can be easily observed intergating \eqref{8.1}$_1$ and \eqref{8.1}$_3$
and applying the boundary conditions.
Thus,  for $\lambda \in \Sigma_{\epsilon, \lambda_0}$
the solutions obtained above belong to $\hat \CH_q(\Omega)$.

Let $\CB_{\lambda_0} = \{\lambda \in \BC \mid {\rm Re}\,\lambda \geq0,
\enskip |\lambda| \leq \lambda_0\}$.  Our task now is to prove
the unique existence theorem  for $\lambda \in \CB_{\lambda_0}$. We first
consider the case where $\lambda\not=0$.  Inserting the formula
$\zeta = \lambda^{-1}(f_1 - a_{*0}\dv\bv)$ into the second equation
in Eq. \eqref{8.1}, it becomes
$$\lambda\bv - a_{0*}^{-1}\{\mu\Delta \bv 
+ (\nu + \lambda^{-1}a_{1*}a_{0*})\nabla\dv\bv 
-a_{2*}\nabla\vartheta\} = \bff_2-a_{0*}^{-1}a_{1*}\lambda^{-1}
\nabla f_1.
$$
Thus, we consider the following equations:
\begin{equation}\label{8.5}\left\{\begin{aligned}
\lambda\bv - a_{0*}^{-1}\{\mu\Delta \bv 
+ (\nu + \lambda^{-1}a_{1*}a_{0*})\nabla\dv\bv\}
+a_{0*}^{-1}a_{2*}\nabla\vartheta &= \bff_2 &\quad&\text{in $\Omega$}, \\
\lambda\vartheta + a_{2*}a_{3*}^{-1}\dv\bv - 
a_{3*}^{-1}a_{4*}\Delta\vartheta & = f_3
&\quad&\text{in $\Omega$}, \\
\bv|_\Gamma = 0, \quad (\nabla\vartheta)\cdot\bn|_\Gamma
&=0.
\end{aligned}\right.\end{equation}
To solve Eq. \eqref{8.5}, we introduce a new resolvent 
parameter $\tau > 0$ and we consider auxiliary problem:
\begin{equation}\label{8.6}
\tau(\bv, \vartheta)  - \CA_\lambda(\bv, \vartheta)
=(\bg_1, g_2) 
\quad\text{in $\Omega$},
\end{equation}
where we have set 
\begin{align*}
\CA_\lambda(\bv, \vartheta) & 
= (A_{1\lambda}\bv - a_{0*}^{-1}a_{2*}\nabla\vartheta, \,
a_{3*}^{-1}a_{4*}\Delta\vartheta-a_{2*}a_{3*}^{-1}\dv\bv)
\quad \text{for $(\bv, \vartheta) \in 
\CD^1_q(\Omega)\times \CD^2_q(\Omega)$},\\
\CD^1_q(\Omega) &= \{\bv \in H^2_q(\Omega)^N \mid \bv|_\Gamma=0\}, 
\quad \CD_q^2(\Omega) = \{ \vartheta \in H^2_q(\Omega) \mid
(\nabla\vartheta)\cdot\bn|_\Gamma=0\}, \\
A_{1\lambda}\bv & = a_{0*}^{-1}(\mu\Delta\bv+ (\nu+\lambda^{-1}a_{1*}a_{0*})
\nabla\dv\bv) \quad\text{for $\bv \in H^2_q(\Omega)$}. 
\end{align*}
Let 
\begin{align*}
\CA_{1\lambda}\bv &= a_{0*}^{-1}(\mu\Delta\bv + (\nu+\lambda^{-1}a_{1*}a_{0*})
\nabla\dv\bv) \quad\text{for $\bv \in \CD^1_q(\Omega)$}, \\
\CA_2 \vartheta &= a_{3*}^{-1}a_{4*}\Delta\vartheta
\quad\text{for $\vartheta \in \CD^2_q(\Omega)$}.
\end{align*} 
Shibata and Tanaka \cite{ST1} proved that there exists a $\tau_0 > 0$ such that
$(\tau\bI - \CA_{1\lambda})^{-1}$ 
exists as a bounded linear operator from $L_q(\Omega)^N$ into
$\CD_q^1(\Omega)$ for  $\tau \geq \tau_0$ possessing the estimate:
\begin{equation}\label{aux:1}
\tau\|\bw\|_{L_q(\Omega)} + \tau^{1/2}\|\bw\|_{H^1_q(\Omega)}
+\|\bw\|_{H^2_q(\Omega)}
\leq C\|\bg_1\|_{L_q(\Omega)}
\end{equation} 
for any $\tau \geq \tau_0$ and $\bg_1 \in L_q(\Omega)^N$, where 
we have set $\bw = (\tau\bI-\CA_{1\lambda})^{-1}\bg_1$.
And, by Theorem \ref{thm:7.2}
we see that there exists a $\tau_0 > 0$ such that 
$(\tau\bI-\CA_2)^{-1}$ exists as a bounded linear operator from 
$L_q(\Omega)$ into $\CD^2_q(\Omega)$ for  $\tau \geq \tau_0$ possessing
the estimate:
\begin{equation}\label{aux:2}
\tau\|\varphi\|_{L_q(\Omega)} + \tau^{1/2}\|\varphi\|_{H^1_q(\Omega)}
+\|\varphi\|_{H^2_q(\Omega)}
\leq C\|g_2\|_{L_q(\Omega)}
\end{equation}
for any $\tau \geq \tau_0$ and $g_2 \in L_q(\Omega)$, where 
we have set $\varphi = (\tau\bI-A_{2})^{-1}\bg_1$. To solve \eqref{8.6},
we set $(\bv, \vartheta) = ((\tau\bI-\CA_{1\lambda})^{-1}\bg_1, 
(\tau\bI-\CA_2)^{-1}g_2)$.  We then have
\begin{equation}\label{aux:3}
\tau(\bv, \vartheta) - \CA_\lambda(\bv, \vartheta) = (\bg_1, g_2) 
+ \CR_\tau(\bg_1, g_2) 
\end{equation}
where we have set
$$\CR_\tau(\bg_1, g_2)  = (a_{0*}^{-1}a_{2*}\nabla(\tau\bI-\CA_2)^{-1}g_2,\,
a_{2*}a_{3*}^{-1}\dv(\tau\bI - \CA_1)^{-1}\bg_1). 
$$
By \eqref{aux:1} and \eqref{aux:2}, we have 
$$\|\CR_\tau(\bg_1, g_2)\|_{L_q(\Omega)} \leq C\tau^{-1/2}
\|(\bg_1, g_2)\|_{L_q(\Omega)},
$$
and so for large $\tau > 0$, $(\bI - \CR_\tau)^{-1}$ exists as an element in 
$\CL(L_q(\Omega)^{N+1})$ and $\|(\bI + \CR_\tau)^{-1}\|_{\CL(L_q(\Omega)^{N+1})} \leq 2$. Let $(\bI+\CR_\tau)^{-1}(\bg_1, g_2)
= (\bh_{1\tau}, h_{2\tau})$, and then $\bv_\tau = (\tau\bI-\CA_1)\bh_{\tau1} \in 
\CD^1_q(\Omega)$ and $\vartheta_\tau = (\tau-\CA_2)^{-1}h_{2\tau} \in 
\CD_q^2(\Omega)$ are unique solutions of Eq. \eqref{8.6} possessing the 
estimate: 
\begin{equation}\label{aux:4}
\tau\|(\bv_\tau, \vartheta_\tau)\|_{L_q(\Omega)}
+ \|\bv_\tau, \vartheta_\tau)\|_{H^2_q(\Omega)}
\leq C\|(\bg_1, g_2)\|_{L_q(\Omega)}
\end{equation} 
for any large $\tau > 0$. Namely, the resolvent set $\rho(\CA_\lambda)$ of 
$\CA_\lambda$ contains $(\tau_1, \infty)$ for some $\tau_1 > 0$. 
We then write the resolvent operator by $(\tau\bI - \CA_\lambda)^{-1}$
as usual. If we set $(\bv_\tau, \vartheta_\tau) = (\bI-\CA_\lambda)^{-1}
(\bg_1, g_2)$, then $(\bv_\tau, \vartheta_\tau)$ satisfies the estimate
\eqref{aux:4}. Using $(\tau\bI-\CA_\tau)^{-1}$, we write Eq. \eqref{8.5}
as 
\begin{equation}\label{8.10}
(\bv, \vartheta) + (\lambda-\tau)(\tau\bI - \CA_\lambda)^{-1}(\bv, \vartheta)
= (\tau\bI - \CA_\lambda)^{-1}(\bg_1, g_2).
\end{equation}
Since $(\lambda-\tau)(\tau\bI - \CA_\lambda)^{-1}$ is a compact operator
on $L_q(\Omega)^{N+1}$, in view of Riesz-Schauder theory, in particular
Fredholm alternative principle, it is sufficient to prove that
the kernel of $\bI + (\lambda-\tau)(\tau\bI
- \CA_\lambda)^{-1}$ is trivial in order to prove the existence of
$(\bI + (\lambda-\tau)(\tau\bI - \CA_\lambda)^{-1})^{-1}
\in \CL(L_q(\Omega)^{N+1})$. Thus, let $(\bg_1, g_2)$ be an element in
$L_q(\Omega)^{N+1}$ for which 
$$(\bI + (\lambda-\tau)(\tau\bI - \CA_\lambda)^{-1})(\bg_1, g_2) = (0, 0).
$$
Since $(\bg_1, g_2) = (\tau-\lambda)(\tau\bI - 
\CA_\lambda)^{-1}(\bg_1, g_2) \in 
\CD^1_q(\Omega)\times \CD^2_q(\Omega)$, setting $(\bv, \vartheta) 
= (\tau\bI - \CA_\lambda)(\bg_1, g_2)$, we have 
$$(0, 0) = (\tau-\CA_\lambda)(\bv, \vartheta) + (\lambda-\tau)(\bv, \vartheta)=
(\lambda\bI - \CA_\lambda)(\bv, \vartheta),
$$
that is, $(\bv, \vartheta) \in H^2_q(\Omega)^{N+1}$ satisfies the 
homogeneous equations:
\begin{equation}\label{8.11}\left\{\begin{aligned}
a_{0*}\lambda\bv - \mu\Delta \bv 
- (\nu + \lambda^{-1}a_{1*}a_{0*})\nabla\dv\bv 
+a_{2*}\nabla\vartheta &= 0 &\quad&\text{in $\Omega$}, \\
a_{3*}\lambda\vartheta + a_{2*}\dv\bv - a_{4*}\Delta\vartheta & = 0
&\quad&\text{in $\Omega$}, \\
\bv|_\Gamma = 0, \quad (\nabla\vartheta)\cdot\bn|_\Gamma
&=0.
\end{aligned}\right.\end{equation}
To prove $(\bv, \vartheta)=(0,0)$, we first consider the case where
$2 \leq q < \infty$.  Since $(\bv, \vartheta) \in H^2_q(\Omega)^{N+1}
\subset H^2_2(\Omega)^{N+1}$, by \eqref{8.11}
and the divergence theorem of Gau\ss\, we have

\eqh{
0=&a_{0*}\lambda\|\bv\|_{L_2(\Omega)}^2 + \mu\|\nabla\bv\|_{L_2(\Omega)}^2
+(\nu+\lambda^{-1}a_{1*}a_{0*})\|\dv\bv\|_{L_2(\Omega)}^2 \\
&+a_{3*}\lambda\|\vartheta\|_{L_2(\Omega)}^2
 + a_{4*}\|\nabla\vartheta\|_{L_2(\Omega)}^2 
+a_{2*}\{(\nabla\vartheta, \bv)_\Omega- (\bv, \nabla\vartheta)_\Omega
\}.
}
Taking the real part, we have
\begin{align*}
0=&a_{0*}{\rm Re}\,\lambda\|\bv\|_{L_2(\Omega)}^2 
+ \mu\|\nabla\bv\|_{L_2(\Omega)}^2
+(\nu+a_{1*}a_{0*}{\rm Re}\,\lambda^{-1})\|\dv\bv\|_{L_2(\Omega)}^2 \\
&+a_{3*}{\rm Re}\,\lambda\|\vartheta\|_{L_2(\Omega)}^2
 + a_{4*}\|\nabla\vartheta\|_{L_2(\Omega)}^2. 
\end{align*}
Since ${\rm Re}\,\lambda \geq 0$, we have $\nabla(\bv, \vartheta)=(0, 0)$ 
in $\Omega$, that is $\bv$ and $\vartheta$ are constants.  But, 
$\bv|_\Gamma=0$, and so $\bv=0$.  Thus, by the second equation
and boundary condition $(\nabla\vartheta)\cdot\bn|_\Gamma=0$ in 
\eqref{8.11}, we have
$$0= a_{3*}\lambda \int_\Omega \vartheta\,dx - a_{4*}
\int_\Omega\Delta\vartheta\,dx = a_{3*}\lambda \int_\Omega \vartheta\,dx,
$$
and so $\vartheta=0$. Thus, in the case that $2\leq q < \infty$, we see that
Eq. \eqref{8.5} admits a unique solution $(\bv, \vartheta) \in 
\CD^1_q(\Omega)\times \CD^2_q(\Omega)$ possessing the estimate:
\begin{equation}\label{8.13}\|(\bv, \vartheta)\|_{H^2_q(\Omega)} 
\leq C_\lambda\|(\bff_2, f_3)\|_{L_q(\Omega)}
\end{equation}
for some constant $C_\lambda$ depending on $\lambda$.  

We next consider the case $1 < q < 2$.  Let $q^* = q/(q-1)
\in (2, \infty)$.  For any $(\bg_1, g_2) \in L_q(\Omega)^{N+1}$,
let $(\bw, \varphi) \in \CD^1_{q*}(\Omega)\times \CD^2_{q*}(\Omega)$
be a solution of the equation:
\begin{equation}\label{8.12}\left\{\begin{aligned}
\bar\lambda\bw - a_{0*}^{-1}\{\mu\Delta \bw 
+ (\nu + \bar\lambda^{-1}a_{1*}a_{0*})\nabla\dv\bw\}
-a_{0*}^{-1}a_{2*}\nabla\varphi &= \bg_1 &\quad&\text{in $\Omega$}, \\
\lambda\varphi - a_{2*}a_{3*}^{-1}\dv\bw - 
a_{3*}^{-1}a_{4*}\Delta\varphi & =g_2
&\quad&\text{in $\Omega$}, \\
\bv|_\Gamma = 0, \quad (\nabla\vartheta)\cdot\bn|_\Gamma
&=0.
\end{aligned}\right.\end{equation}
Replacing $\lambda$ and $a_{2*}$ by $\bar\lambda$ and 
$-a_{2*}$ in \eqref{8.5}, we can prove
the unique existence of solutions $(\bw, \varphi)
\in \CD^1_q(\Omega)\times\CD^2_q(\Omega)$ of Eq. \eqref{8.12}.
By the divergence theorem of Gau\ss\, we have
\begin{align*}
0& = (a_{0*}\lambda\bv - \mu\Delta\bv-(\nu+\lambda^{-1}a_{1*}a_{0*})
\nabla\dv\bv + a_{2*}\nabla\vartheta, \bw)_\Omega 
%\\
%& 
+(a_{3*}\lambda\vartheta + a_{2*}\dv\bv
- a_{4*}\Delta\vartheta, \varphi)_\Omega\\
& = (\bv, a_{0*}\bar\lambda\bw - \mu\bw - (\nu+\bar\lambda^{-1}a_{1*}a_{0*})
\nabla\dv) - (\vartheta, a_{2*}\dv\bw)_\Omega 
%\\
%& 
+(\vartheta, a_{3*}\bar\lambda\varphi- a_{4*}\Delta\varphi)_\Omega
-(\bv, a_{2*}\nabla\varphi)
\\
& = a_{0*}(\bv, \bg_1)_\Omega + a_{3*}(\vartheta, g_2)_\Omega.
\end{align*}
Thus, the arbitrariness of $(\bg_1, g_2) \in L_{q*}(\Omega)^{N+1}$
yields $(\bv, \vartheta)=(0, 0)$, which leads to the unique existence of 
solutions $(\bv, \vartheta) \in \CD^1_q(\Omega)\times
\CD^2_q(\Omega)$ of Eq. \eqref{8.5} possessing the estimate 
\eqref{8.13}. Thus, we have proved that for 
any $\lambda \in \CB_{\lambda_0}\setminus\{0\}$
and $(f_1, \bff_2, f_3) \in \CH_q(\Omega)$  Eq. \eqref{8.1}
admits a unique solution $(\zeta, \bv, \vartheta)\in \CD_q(\Omega)$
possessing the estimate:
\begin{equation}\label{proof:1}
\|\zeta, \bv, \vartheta)\|_{D_q(\Omega)}
\leq C_\lambda\|(f_1, \bff_2, f_3)\|_{\CH_q(\Omega)}.
\end{equation}

We now consider the case that $\lambda=0$. 
Inserting the relation: $\dv\bv = a_{0*}^{-1}f_1$, we rewrite
\eqref{8.1} as 
\begin{equation}\label{8.14} \left\{\begin{aligned}
\dv\bv &=a_{0*}^{-1} f_1&\quad&\text{in $\Omega$}, \\
 - \mu\Delta\bv +a_{1*}\nabla\zeta & = \bff_2
+\nu a_{0*}^{-1} \nabla f_1-a_{2*}\nabla \vartheta
&\quad&\text{in $\Omega$}, \\
 - a_{4*}\Delta\vartheta & = f_3 - a_{0*}^{-1}a_{2*}f_1
&\quad&\text{in $\Omega$}, \\
\bv|_{\Gamma}=0, \quad (\nabla\vartheta)\cdot\bn|_\Gamma
& = 0.
\end{aligned}\right.\end{equation}
We first consider the Laplace equation:
\begin{equation}
\label{8.16}
-a_{4*}\Delta\vartheta = g_2\quad\text{in $\Omega$}, 
\quad(\nabla\vartheta)\cdot\bn|_{\Gamma}=0,
\end{equation}
and then, for any $g_2 \in L_q(\Omega)$ with $\int_\Omega g_2\,dx=0$, 
problem \eqref{8.16} admits a unique solution $\vartheta \in H^2_q(\Omega)$
with $\int_\Omega \vartheta\,dx = 0$ possessing the estimate:
$\|\vartheta\|_{H^2_q(\Omega)} \leq C\|g_2\|_{L_q(\Omega)}$. 
Therefore the third equation of Eq. \eqref{8.14} admits a unique 
solution $\vartheta \in H^2_q(\Omega)$ satisfying the estimate:
$\|\vartheta\|_{H^2_q(\Omega)} \leq C\|(f_1, f_3)\|_{L_q(\Omega)}$
and 
%the average zero condition: 
$\int_\Omega \vartheta\,dx=0$. 

Finally, setting $g_1 = a_{0*}^{-1}f_1$ and $\bg_2 = 
\bff_2 - \nu a_{0*}^{-1}\nabla f_1 -a_{2*}\nabla\vartheta$, we consider the
Cattabriga problem: 
\begin{equation}\label{8.17} 
- \mu\Delta\bv 
+a_{1*}\nabla\zeta  = \bg_2, 
\quad \dv\bv =g_1
\quad\text{in $\Omega$}, \quad 
\bv|_{\Gamma}=0.
\end{equation}
By Farwig and Sohr \cite{FS}, there exists a $\lambda_0 > 0$ for which
the equation:
$$
\lambda_0\bv- \mu\Delta\bv 
+a_{1*}\nabla\zeta  = \bg_2, 
\quad \dv\bv =g_1
\quad\text{in $\Omega$}, \quad 
\bv|_{\Gamma}=0,
$$
admits a unique solution $(\zeta, \bv) 
\in H^1_q(\Omega)\times H^2_q(\Omega)^N$ with $\int_\Omega \zeta\,dx=0$
for any $(g_1, \bg_2) \in H^1_q(\Omega)\times L_q(\Omega)^N$ with
$\int_\Omega g_2\,dx = 0$.  Thus, by the Fredholm alternative
principle, the uniqueness of solutions of Eq. \eqref{8.17}
yields the unique existence theorem, that is
for any $(g_1, \bg_2) \in H^1_q(\Omega)\times L_q(\Omega)^N$ with
$\int_\Omega g_2\,dx = 0$, problem \eqref{8.17}
admits a unique solution 
$(\zeta, \bv) 
\in H^1_q(\Omega)\times H^2_q(\Omega)^N$ with $\int_\Omega \zeta\,dx=0$
possessing the estimate:
$$\|\zeta\|_{H^1_q(\Omega)} + \|\bv\|_{H^2_q(\Omega)}
\leq C(\|g_1\|_{H^1_q(\Omega)} + \|\bg_2\|_{L_q(\Omega)}).
$$
Therefore the problem of existence for \eqref{8.17} is reduced to showing uniqueness for
the homogeneous problem which is an immediate consequence of the divergence theorem.  

Summing up, we have proved that
for any $(f_1, \bff_2, f_3) \in \hat \CH_q(\Omega)$, 
problem \eqref{8.14} admits a unique solution
$(\zeta, \bv, \vartheta) \in \CD_q(\Omega) \cap \hat \CH_q(\Omega)$
possessing the estimate:
$$\|\zeta\|_{H^1_q(\Omega)} + \|(\bv, \vartheta)\|_{H^2_q(\Omega)}
\leq C(\|f_1\|_{H^1_q(\Omega)} + \|(\bff_2, f_3)\|_{L_q(\Omega)}).
$$
Since the resolvent operator is continuous and 
the set $\CB_{\lambda_0}$ is compact, we can take the constants
$C_\lambda$ in the estimate \eqref{proof:1} independent of $\lambda
\in \CB_{\lambda_0}$.  This completes the proof of Theorem \ref{thm:8.1}.
\end{proof}

We now give a \vskip0.5pc\noindent
{\bf Proof of Theorem \ref{thm:decay1}}.  Let 
\begin{align*}
PU &= \left(\begin{matrix}
-\rho_{0*}\dv\bv \\
a_{0*}^{-1}(\mu\Delta\bv + \nu\nabla\dv\bv- a_{1*}\nabla\zeta
-a_2\nabla\vartheta) \\
-a_{3*}^{-1}(a_{2*}\dv\bv - a_{4*}\Delta\vartheta)
\end{matrix}\right)
\quad \text{for $U = (\zeta, \bv, \vartheta) \in D_q(\Omega)$}, \\
\CP U & = PU \quad\text{for $U = (\zeta, \bv, \vartheta) \in \CD_q(\Omega)
\cap \hat \CH_q(\Omega)$}.
\end{align*}
Here, $\hat\CH_q(\Omega)$ and $\CD_q(\Omega)$ are the spaces given in
\eqref{space:8.1} and \eqref{p.7.0}, respectively. 
Let us consider the Cauchy problem:
\begin{equation}\label{8.19}
\pd_tU- \CP U = 0 \quad\text{for $t>0$}, \quad
U|_{t=0} = U_0=(\zeta_0, \bv_0, \vartheta_0) \in \hat\CH_q(\Omega).
\end{equation}
The resolvent problem corresponding to  \eqref{8.19} is Eq. \eqref{8.1}.  Thus,
by Theorem \ref{thm:8.1}, we see that $\CP$ generates a $C_0$ analytic
semigroup $\{\dot T(t)\}_{t\geq 0}$ that is exponentially stable on
$\hat \CH_q(\Omega)$, that is 
\begin{equation}\label{8.20}
\|\dot T(t)U_0\|_{\CH_q(\Omega)} \leq Ce^{-\gamma_1t}\|U_0\|_{\CH_q(\Omega)}
\end{equation}
for any $U_0 \in \hat \CH_q(\Omega)$ and $t > 0$ with some positive 
constants $C$ and $\gamma_1$. 

Let $\lambda_1 > 0$ be a sufficiently large number and 
let $\gamma > 0$ be a small positive number determined later.
We assume that 
\begin{equation}\label{choice:1}
0 <\gamma < \gamma_1.
\end{equation}
We consider the time-shifted equations:
\begin{equation}\label{8.21}\left\{\begin{aligned}
 \pd_t U_1 + \lambda_1U_1 -PU_1 &= G
&\quad&\text{in $\Omega\times (0, T)$}, \\
BU_1&=(0, g_4)&\quad&\text{on $\Gamma\times(0,T)$}, \\
U_1|_{t=0} &= U_0 &\quad&\text{in $\Omega$}.
\end{aligned}\right.\end{equation}
where $G = (g_1, \bg_2, g_3)$ and
$BU = (\bv, (\nabla\vartheta)\cdot\bn)$. 
Multiplying Eq. \eqref{8.21} by $e^{\gamma t}$, we have 
\begin{equation}\label{8.22}\left\{\begin{aligned}
 \pd_t(e^{\gamma t} U_1) + (\lambda_1-\gamma)
e^{\gamma t}U_1 -P(e^{\gamma t}U_1) &= e^{\gamma t}G
&\quad&\text{in $\Omega\times (0, T)$}, \\
B(e^{\gamma t}U_1) &=(0, e^{\gamma t}g_4)
&\quad&\text{on $\Gamma\times(0,T)$}, \\
e^{\gamma t}U_1|_{t=0} &= U_0 &\quad&\text{in $\Omega$}.
\end{aligned}\right.\end{equation}
Let $G_0$ be the zero extension of $G$ to $\BR$ with respect to $t$, that is 
$G_0(\cdot, t) = G(\cdot, t)$ for $t \in (0, T)$ and
$G_0(\cdot, t) =0$ for $t \not\in (0, T)$. To estimate
$e^{\gamma t}U_1$,  we consider
the equations: 
\begin{equation}\label{8.23}\left\{\begin{aligned}
 \pd_t U_2 + (\lambda_1-\gamma)U_2 -PU_2 &= e^{\gamma t}G_0
&\quad&\text{in $\Omega\times \BR$}, \\
BU_2 &=(0, e^{\gamma t}g_4)
&\quad&\text{on $\Gamma\times\BR$}.
\end{aligned}\right.\end{equation}
Applying the Fourier transform with respect to $t$ to Eq. 
\eqref{8.23},  we have
\begin{equation}\label{8.23*}\left\{\begin{aligned}
(\lambda_1-\gamma+ i\tau)\CF[U_2](\cdot, \tau)
 -P\CF[U_2](\cdot, \tau) &= 
\CF[e^{\gamma t}G_0](\cdot, \tau)
&\quad&\text{in $\Omega$}, \\
B\CF[U_2](\cdot, \tau) &=(0, \CF[e^{\gamma t}g_4](\cdot, \tau))
&\quad&\text{on $\Gamma$}.
\end{aligned}\right.\end{equation}
Let $\CS(\lambda) = (\CA(\lambda), \CB_1(\lambda), \CB_2(\lambda))$ be
the $\CR$-bounded solution operators given in Theorem \ref{thm:7.1}.
If we choose $\lambda_1>0$ so large that $\lambda_1-\gamma \geq \lambda_0$,
then we have $\CF[\hat U_2](\cdot, \tau) = \CS(\lambda_1-\gamma + i\tau)
\bF_{\lambda_1-\gamma+i\tau}$, where 
$$\bF_{\lambda_1-\gamma+i\tau} 
= (\CF[e^{\gamma t}G_0](\cdot, \tau), 
(\lambda_1-\gamma+i\tau)^{1/2}\CF[e^{\gamma t}g_4](\cdot, \tau),
\CF[e^{\gamma t}g_4](\cdot, \tau)).$$
Since $$(\tau\pd_\tau)^\ell(i\tau/\lambda_1-\gamma+i\tau)|
\leq C_{\lambda_1}, \quad
|(\tau\pd_\tau)^\ell((\lambda_1-\gamma+i\tau)^{1/2}/(1+\tau^2)^{1/4})|
\leq C_{\lambda_1}
$$
for $\ell=0,1$ and $\tau \in \BR\setminus\{0\}$, applying Weis' operator
valued Fourier multiplier theorem and Bourgain's theorem (cf.
Lemma \ref{lem:7.3.5}) to 
$$U_1= \CF^{-1}[\CF[U_1](\cdot, \tau)] =
\CF^{-1}[\CS(\lambda_1-\gamma + i\tau)
\bF_{\lambda_1-\gamma+i\tau}],$$
 we have 
\begin{equation}\label{8.24}\begin{split}
&\|\pd_tU_2\|_{L_p(\BR, \CH_q(\Omega))}
+ \|U_2\|_{L_p(\BR, D_q(\Omega))} \\
&\quad \leq C(\|e^{\gamma t}G_0\|_{L_p(\BR, \CH_q(\Omega))}
+ \|e^{\gamma t}g_4\|_{H^{1/2}_p(\BR, L_q(\Omega))}
+ \|e^{\gamma t}g_4\|_{L_p(\BR, H^1_q(\Omega))}) \\
&\quad \leq C(\|e^{\gamma t}G\|_{L_p((0, T), \CH_q(\Omega))}
+ \|e^{\gamma t}g_4\|_{H^{1/2}_p(\BR, L_q(\Omega))}
+ \|e^{\gamma t}g_4\|_{L_p(\BR, H^1_q(\Omega))}).
\end{split}\end{equation}
We next consider the Cauchy problem:
\begin{equation}\label{8.25}\left\{\begin{aligned}
 \pd_t U_3 + (\lambda_1-\gamma)
U_3 -PU_3 &= 0
&\quad&\text{in $\Omega\times (0, \infty)$}, \\
BU_3 &=(0, 0)
&\quad&\text{on $\Gamma\times(0,\infty)$}, \\
U_3|_{t=0} &= U_0-U_2|_{t=0}&\quad&\text{in $\Omega$}.
\end{aligned}\right.\end{equation}
If we choose $\lambda_1>0$ sufficiently large,
by Theorem \ref{thm:7.1} we see that 
there exists a $C^0$ analytic semigroup
$\{T_1(t)\}_{t\geq 0}$ associated with
Eq. \eqref{8.22}, which is exponentially 
stable. Setting $U_3 = T_1(t)(U_0-U_2|_{t=0})$, 
we then see that $U_2$ satisfies Eq. \eqref{8.22}
and the estimate:
\begin{equation}\label{8.26}
\|e^{\gamma t}\pd_tU_3\|_{L_p((0, \infty), \CH_q(\Omega))}
+ \|e^{\gamma t}U_3\|_{L_p((0, \infty), D_q(\Omega))}
\leq C\|U_0-U_2|_{t=0}\|_{D_{p,q}(\Omega)}.
\end{equation}
By the uniqueness of solutions, we have $e^{\gamma t}U_1 
= U_2 + U_3$, and so by \eqref{8.24}, \eqref{8.26},
and real interpolation theorem 
\eqref{real-int:1}
and \eqref{real-int:2}, 
we have
\begin{equation}\label{8.27}\begin{split}
&\|e^{\gamma t}\pd_tU_1\|_{L_p((0, T),  \CH_q(\Omega))}
+ \|e^{\gamma t}U_1\|_{L_p((0, T),  D_q(\Omega))} \\
& \leq C(\|(\zeta_0, \bv_0, \vartheta_0)\|_{D_{p,q}(\Omega)}
+ \|e^{\gamma t}G\|_{L_p((0, T), \CH_q(\Omega))}
+ \|e^{\gamma t}g_4\|_{H^{1/2}_p(\BR, L_q(\Omega))}
+ \|e^{\gamma t}g_4\|_{L_p(\BR, H^1_q(\Omega))}).
\end{split}\end{equation}
We next consider the equations:
\begin{equation}\label{8.29} \pd_tV - PV = -\lambda_0U_1 
\quad\text{in $\Omega\times(0, T)$}, \quad
BV|_\Gamma=0, \quad V|_{t=0} =0.\ \quad\text{in $\Omega$}.
\end{equation}
Let $U_1 = (\zeta_1, \bv_1, \vartheta_1)$ and set 
\begin{equation}\label{8.33}\tilde U_1(x, t) = (\zeta_1(x, t) 
- \frac{1}{|\Omega|}\int_\Omega\zeta_1(y, t)\,dy, \bv_1(x, t),
\vartheta_1(x, t)-\frac{1}{|\Omega|}\int_\Omega\vartheta_1(y, t)\,dy).
\end{equation}
Then $\tilde U(\cdot, t) \in \hat \CH_q(\Omega)$ for any $t \in (0, T)$.
We consider the equations: 
\begin{equation}\label{8.30} \pd_t\tilde V - P\tilde V = -\lambda_0\tilde U_1 
\quad\text{in $\Omega\times(0, T)$}, \quad
B\tilde V|_\Gamma=0, \quad \tilde V|_{t=0} =0.\ \quad\text{in $\Omega$}.
\end{equation}
In view of \eqref{8.19}, by the Duhamel principle we have
$$\tilde  V = \int^t_0 \dot T(t-s) \tilde U_1(\cdot, s)\,ds. $$
Moreover, by \eqref{8.20} we have
\begin{equation}\label{8.31}
\|e^{\gamma t}\tilde V\|_{L_p((0, T), \CH_q(\Omega))}
\leq C(\gamma_1-\gamma)^{-1/p}
\|e^{\gamma t}\tilde U_1\|_{L_p((0, T), \CH_q(\Omega))}.
\end{equation}
Here, we choose $\gamma$ as \eqref{choice:1}. In fact, by \eqref{8.20} 
and H\"older's inequality with exponent $p' = p/(p-1)$ we have
\begin{align*}
e^{\gamma t}\|\tilde V(\cdot, t)\|_{L_q(\Omega)}
& \leq C\int^t_0e^{\gamma t}e^{-\gamma_1(t-s)}
\|\tilde U_1(\cdot, s)\|_{\CH_q(\Omega)}\,ds
= C\int^t_0e^{-\gamma_1(t-s)}e^{\gamma s}
\|\tilde U_1(\cdot, s)\|_{\CH_q(\Omega)}\,ds\\
& \leq \Bigl(\int^t_0e^{-(\gamma_1-\gamma)(t-s)}\,ds\Bigr)^{1/{p'}}
\Bigl(\int^t_0 e^{-(\gamma_1-\gamma)(t-s)}
(e^{\gamma s}\|\tilde U(\cdot, s)\|_{\CH_q(\Omega)})^p\,ds
\Bigr)^{1/p},
\end{align*}
and so by the change of integration order we have 
\begin{align*}
\int^T_0(e^{\gamma t}\|\tilde V(\cdot, t)\|_{L_q(\Omega)})^p\,dt
&\leq C^p(\gamma_1-\gamma)^{-p/{p'}}\int^T_0
(e^{\gamma s}\|\tilde U(\cdot, s)\|_{\CH_q(\Omega)})^p\,ds
\int^T_se^{-(\gamma_1-\gamma)(t-s)}\,dt
\\
&= C^p(\gamma_1-\gamma)^{-p}\int^T_0
(e^{\gamma s}\|\tilde U(\cdot, s)\|_{\CH_q(\Omega)})^p\,ds.
\end{align*}
Thus, we have \eqref{8.31}.

Since $\tilde V$ satisfies the shifted equations: 
$$\pd_t\tilde V + \lambda_0\tilde V-P\tilde V = -\lambda\tilde U_1
+\lambda \tilde V \quad\text{in $\Omega\times(0, T)$}, 
\quad B\tilde V|_\Gamma = 0, \quad \tilde V|_{t=0} = 0,
$$
we have
$$
\|e^{\gamma t}\pd_t\tilde V\|_{L_p((0, T), \CH_q(\Omega))}
+ \|e^{\gamma t}\tilde V\|_{L_p((0, T), D_q(\Omega))}
\leq C(\|e^{\gamma t}\tilde U_1\|_{L_p((0, T), \CH_q(\Omega))}
+ \|e^{\gamma t}\tilde V\|_{L_p((0, T), \CH_q(\Omega))}),
$$
which, combined with \eqref{8.27} and \eqref{8.31}, 
leads to 
\begin{equation}\label{8.32}\begin{split}
&\|e^{\gamma t}\pd_t\tilde V\|_{L_p((0, T),  \CH_q(\Omega))}
+ \|e^{\gamma t}\tilde V\|_{L_p((0, T),  D_q(\Omega))} \\
& \leq C(\|(\zeta_0, \bv_0, \vartheta_0)\|_{D_{p,q}(\Omega)}
+ \|e^{\gamma t}G\|_{L_p((0, T), \CH_q(\Omega))}
+ \|e^{\gamma t}g_4\|_{H^{1/2}_p(\BR, L_q(\Omega))}
+ \|e^{\gamma t}g_4\|_{L_p(\BR, H^1_q(\Omega))}).
\end{split}\end{equation}
In view of \eqref{8.33}, we define $V$ by  
$$V = \tilde V - (\frac{1}{|\Omega|}\int^t_0\int_\Omega\zeta_1(x, s)\,dx, 0,
\frac{1}{|\Omega|}\int^t_0\int_\Omega \vartheta_1(x, s)\,dxds),$$
and then, $V$ satisfies Eq.\eqref{8.29}.  Moreover, 
setting $V = (\zeta_2, \bv_2, \vartheta_2)$, by \eqref{8.32}
and \eqref{8.27} we have 
\begin{equation}\label{8.34}
\begin{split}
&\|e^{\gamma t}\pd_t(\zeta_2, \bv_2, \vartheta_2)
\|_{L_p((0, T),  \CH_q(\Omega))}
+ \|e^{\gamma t}\nabla\zeta_2\|_{L_p((0, T),  H^1_q(\Omega))} 
+ \|e^{\gamma t}\bv_2\|_{L_p((0, T), H^2_q(\Omega))} \\
&\quad + \|e^{\gamma t}\nabla \vartheta_2\|_{L_p((0, T), H^1_q(\Omega))}
\leq C(\|(\zeta_0, \bv_0, \vartheta_0)\|_{D_{p,q}(\Omega)}
+ \|e^{\gamma t}G\|_{L_p((0, T), \CH_q(\Omega))} \\
&\qquad + \|e^{\gamma t}g_4\|_{H^{1/2}_p(\BR, L_q(\Omega))}
+ \|e^{\gamma t}g_4\|_{L_p(\BR, H^1_q(\Omega))}).
\end{split}\end{equation}
Let $(\zeta, \bv, \vartheta) = U_1 + V$, and then 
$(\zeta, \bv, \vartheta)$ is a unique solution of \eqref{newlinear:1}.
Moreover, by \eqref{8.34} and \eqref{8.27}  $(\zeta, \bv, \vartheta)$
satisfies the decay estimate:
\begin{align*}
&\|e^{\gamma t}\pd_t(\zeta, \bv, \vartheta)
\|_{L_p((0, T),  \CH_q(\Omega))}
+ \|e^{\gamma t}\nabla\zeta\|_{L_p((0, T),  H^1_q(\Omega))} 
+ \|e^{\gamma t}\bv\|_{L_p((0, T), H^2_q(\Omega))} 
 + \|e^{\gamma t}\nabla \vartheta\|_{L_p((0, T), H^1_q(\Omega))}\\
&\leq C(\|(\zeta_0, \bv_0, \vartheta_0)\|_{D_{p,q}(\Omega)}
+\|e^{\gamma t}(g_1, \bg_2, g_3)\|_{L_p((0, T), \CH_q(\Omega))} 
+ \|e^{\gamma t}g_4\|_{H^{1/2}_p(\BR, L_q(\Omega))}
+ \|e^{\gamma t}g_4\|_{L_p(\BR, H^1_q(\Omega))}).
\end{align*}
This completes the proof of Theorem \ref{thm:decay1}.


\footnotesize

\begin{thebibliography}{10}

\bibitem{Amann2} 
H.~Amann.
\newblock {\it Linear and Quasilinear Parabolic Problems}.
\newblock Vol. I. Birkh\"auser, Basel, 1995.

\bibitem{BE}
J.~Bebernes, D.~Eberly.
\newblock {\em Mathematical problems from combustion theory}, volume~83 of {\em
  Applied Mathematical Sciences}.
\newblock Springer-Verlag, New York, 1989.


\bibitem{B2010}
D.~Bothe.
\newblock On the {M}axwell-{S}tefan approach to multicomponent diffusion.
\newblock In {\em Parabolic problems}, volume~80 of {\em Progr. Nonlinear
  Differential Equations Appl.}, pages 81--93. Birkh{\"a}user/Springer Basel
  AG, Basel, 2011.
  
\bibitem{BD2015}  
D. Bothe, W. Dreyer.   
\newblock Continuum thermodynamics of chemically reacting fluid mixtures.
\newblock {\em Acta Mech.}, 226:1757--1805, 2015.
  
\bibitem{BP2017}
D. Bothe, J Pr\"{u}ss.
\newblock Modeling and analysis of reactive multi-component two-phase flows with mass transfer and phase transition--the isothermal incompressible case. 
\newblock {\emph Discrete Contin. Dyn. Syst. Ser. S} 10, no. 4, 673--696, 2017.

\bibitem{Bourgain} 
J.~Bourgain.
\newblock {\it Vector-valued singular 
integrals and the $H^1$-BMO duality}, 
\newblock In: Probability Theory and 
Harmonic Analysis, D.~Borkholder (ed.) {\it Marcel Dekker, 
New York}, 1--19, 1986. 

\bibitem{BGS12}
L.~Boudin, B.~Grec, and F.~Salvarani.
\newblock A mathematical and numerical analysis of the Maxwell-Stefan diffusion
  equations.
\newblock {\em Discrete Contin. Dyn. Syst. Ser. B}, 17(5):1427--1440, 2012.


\bibitem{BH15}
M. Bulicek, J. Havrda.
\newblock On existence of weak solutions to a model describing compressible
mixtures with thermal diffusion cross effects. 
\newblock {\emph Z. Angew. Math. Mech.} 95, 589--619, 2015.


\bibitem{CJ13}
X.~Chen, A.~J{\"u}ngel.
\newblock Analysis of an incompressible {N}avier-{S}tokes-{M}axwell-{S}tefan
  system.
\newblock {\em Comm. Math. Phys.}, 340 (2), pp. 471--497, 2015.



\bibitem{DHP} 
R.~Denk, M.~Hieber and J.~Pr\"u\ss.
\newblock {\em $\CR$-boundedness, Fourier multipliers and problems of 
elliptic and parabolic type}. 
\newblock Memoirs of AMS. Vol 166. no. 788,  2003.

\bibitem{DDGG}
W. Dreyer, P-\'E. Druet, P. Gajewski, and C. Guhlke.
\newblock Existence of weak solutions for improved Nernst--Planck--Poisson models of compressible reacting electrolytes.
\newblock preprint WIAS, 2016.


\bibitem{ES1} 
Y.~Enomoto, Y.~Shibata.
\newblock {\it On the $\CR$-sectoriality and the initial 
boundary value problem for the viscous compressible 
fluid flow}.
\newblock Funkcial Ekvac., {56}(3), 441--505, 2013.

\bibitem{EBS} 
Y.~Enomoto, L.~von Below, and Y.~Shibata.
\newblock {\it On some free boundary problem for a compressible barotropic viscous fluid flow}. 
\newblock Ann. Univ. Ferrara Sez. VII Sci. Mat., {60}(1), 55--89, 2014.

\bibitem{FS} 
R.~Farwig, H.~Sohr.
\newblock {\it Generalized  resolvent estimates for the Stokes system in bounded and unbounded domains}.
\newblock J. Math. Soc. Japan, {46} (4), 607--643, 1994.


\bibitem{FTP}
E.~Feireisl, H.~Petzeltov{{\'a}}, and K.~Trivisa.
\newblock Multicomponent reactive flows: global-in-time existence for large
  data.
\newblock {\em Commun. Pure Appl. Anal.}, 7(5):1017--1047, 2008.


\bibitem{FGM05b}
J.~Frehse, S.~Goj, and J.~M{{\'a}}lek.
\newblock A uniqueness result for a model for mixtures in the absence of
  external forces and interaction momentum.
\newblock {\em Appl. Math.}, 50(6):527--541, 2005.


\bibitem{Gio90}
V.~Giovangigli.
\newblock Mass conservation and singular multicomonent diffusion algorithms.
\newblock {\em IMPACT Comput. Sci. Eng.}, 2(2):73--97, 1990.

\bibitem{Gio91}
V.~Giovangigli.
\newblock Convergent iterative methods for multicomponent diffusion.
\newblock {\em Impact Comput. Sci. Engrg.}, 3(3):244--276, 1991.

\bibitem{VG}
V.~Giovangigli.
\newblock {\em Multicomponent flow modeling}.
\newblock Modeling and Simulation in Science, Engineering and Technology.
  Birkh{\"a}user Boston Inc., Boston, MA, 1999.

\bibitem{GPZ}
V. Giovangigli, M. Pokorn\'y, and E. Zatorska.
\newblock On the steady flow of reactive gaseous mixture. 
\newblock {\em Analysis (Berlin)} 35, no. 4, 319--341, 2015.


\bibitem{HMPW13}
M.~Herberg, M.~Meyries, J.~Pr{\"u}ss, and M.~Wilke.
\newblock Reaction-diffusion systems of Maxwell-Stefan type with reversible
  mass-action kinetics.
\newblock {\em Nonlinear Anal. }159, 264--284, 2017.

\bibitem{Jungel}
A. J{{\"u}}ngel.
\newblock {\em  Entropy Methods for Diffusive Partial Differential Equations, SpringerBriefs in Mathematics}.
\newblock Springer 2016.

\bibitem{JS13}
A.~J{{\"u}}ngel, I.V. Stelzer.
\newblock Existence analysis of {M}axwell-{S}tefan systems for multicomponent
  mixtures.
\newblock {\em SIAM J. Math. Anal.}, 45(4):2421--2440, 2013.

\bibitem{K84}
S. Kawashima.
\newblock {\em Systems of Hyperbolic-Parabolic Composite Type, with
Application to the Equations of Magnetohydrodynamics.}
\newblock Doctoral Thesis, Kyoto University, 1984.

\bibitem{KS88} S. Kawashima and Y. Shizuta. 
\newblock On the Normal Form of the Symmetric
Hyperbolic-Parabolic Systems Associated with the Conservation Laws.
\newblock {\em Tohoku Math. J.}, 40, pp. 449--464, 1988.

\bibitem{KMP12}
N.~A. Kucher, A.~E. Mamontov, and D.~A. Prokudin.
\newblock Stationary solutions to the equations of the dynamics of mixtures of
  viscous compressible heat-conducting fluids.
\newblock {\em Sibirsk. Mat. Zh. Sib. Math. J.}, 53(6):1075--1088, 2012.


\bibitem{MT13}
M.~Marion, R.~Temam.
\newblock Global existence for fully nonlinear reaction-diffusion systems
  describing multicomponent reactive flows.
\newblock {\em J. Math. Pures Appl.} (9), 104 (1), pp. 102--138, 2015.

\bibitem{MaNi} 
A.~Matusumura, T. Nishida. 
\newblock Initial-boundary value problems for the equations of motion of compressible viscous and heat-conductive fluids. 
\newblock {\it Comm. Math. Phys.} 89 no. 4, 445--464, 1983.
 

\bibitem{MS12} 
M.~Meyries, R.~Schnaubelt. 
\newblock
{\it Interpolation, embeddings and traces of aisotropic fractional 
Sobolev spaces with temporal weights}, \newblock
J. Func. Anal. {\bf 262}, 1200--1229, 2012.

\bibitem{MPZ}
P.~B. Mucha, M.~Pokorn{{\'y}}, and E.~Zatorska.
\newblock Approximate solutions to model of two-component reactive flow.
\newblock {\em Discrete Contin. Dyn. Syst. Ser. S}, 7, no.5 , 1079--1099, 2014.

\bibitem{MPZ1}
P.~B. Mucha, M.~Pokorn{{\'y}}, and E.~Zatorska.
\newblock Chemically reacting mixtures in terms of degenerated parabolic
  setting.
\newblock {\em J. Math. Phys.}, 54(071501), 2013.

\bibitem{MPZ2}
P.~B. Mucha, M.~Pokorn{{\'y}}, and E.~Zatorska.
\newblock Heat-conducting, compressible mixtures with multicomponent diffusion: construction of a weak solution. 
\newblock {\em SIAM J. Math. Anal.} 47, no. 5, 3747--3797, 2015.


\bibitem{Murata} 
M.~Murata.
\newblock {\it On a maximal 
$L_p$-$L_q$ approach to the compressible viscous fluid flow
with slip boundary condition}.
\newblock Nonlinear Analysis, {106}, 86--109, 2014.


\bibitem{MS16}
 M.~Murata,  Y.~Shibata. 
 \newblock
{\it On the global well-posedness for the compressible 
Navier-Stokes equations with slip boundary condition},
J. Differential Equtions {\bf 260} (7), 5761--5795, 2016.


\bibitem{PP1}
T. Piasecki, M. Pokorn{{\'y}}.
\newblock Weak and variational entropy solutions to the system describing steady flow of a compressible reactive mixture. 
\newblock {\em Nonlinear Anal.} 159, 365--392, 2017.

\bibitem{PP2}
T. Piasecki, M. Pokorn{{\'y}}.
\newblock On steady solutions to a model of chemically reacting heat conducting compressible mixture with slip boundary conditions
\newblock Preprint, arXiv:1709.06886, 2017.

\bibitem{SSchade} 
K.~Schade, Y.~Shibata.
 \newblock
{\it On strong dynamics of compressible Nematic Liquid
Crystals}, SIAM J. Math. Anal. {\bf 47}(5), 3963--3992, 2015.

\bibitem{S16} 
Y.~Shibata.
 \newblock
{\it On the global well-posedness of some free boundary problem for a compressible barotropic viscous fluid flow.}
Recent advances in partial differential equations and applications, 341â356,
Contemp. Math., 666, Amer. Math. Soc., Providence, RI, 2016.

\bibitem{S17} 
Y.~Shibata.
 \newblock
{\it On the local wellposedness of free boundary problem
for the Navier-Stokes equations in an exterior domain},
Submitted. 

\bibitem{SS1} \newblock Y.~Shibata and S.~Shimizu.
\newblock {\it On some free boundary problem 
for the Navier-Stokes equtions}, \newblock
Diff. Int. Eqns., {\bf 20}, 241--276, 2007.

\bibitem{SS2} 
Y.~Shibata, S. Shimizu.
{\it On the $L_p$-$L_q$ maximal regularity of the Neumann problem for 
the Stokes equations in a bounded domain}, 
J. Reine Angew. Mat., {\bf 615}, 157--209, 2008.

\bibitem{ST1} 
Y.~Shibata, K.~Tanaka. 
\newblock
{\it On a resolvent problem for the linearized
system from the dynamical system describing the compressible
viscous fluid motion}, Math. Mech. Appl. Sci.
{\bf 27}, 1579--1606, 2004.

\bibitem{St1}
G.~Str\"ohmer. 
\newblock
{\it About a certain class of parabolic-hyperbolic 
systems of differential equation}, Analysis  {\bf 9}, 1--39, 1989. 

\bibitem{Tanabe} 
H.~Tanabe. 
\newblock
Functional analytic methods for partial diffeential 
equations, \newblock \emph{Monographs and textbooks in pure and
applied mathematics}, Vol 204, Marchel Dekker, Inc. New York, Basel, 
1997.

\bibitem{T2000}
S.~R. Turns.
\newblock An introduction to combustion: Concepts and applications.
\newblock {\em McGraw-Hill Series in Mechanical Engineering}, 2000.

\bibitem{Wal62}
L.~Waldmann, E.~Tr{{\"u}}benbacher.
\newblock Formale kinetische {T}heorie von {G}asgemischen aus anregbaren
  {M}olek{\"u}len.
\newblock {\em Z. Naturforsch}, 17a:363--376, 1962.

\bibitem{Weis} 
L.~Weis. 
\newblock {\it Operator-valued Fourier multiplier theorems and maximal $L_p$-regularity}. 
\newblock Math. Ann. { 319}, 735--758, 2001.

\bibitem{WZ} W. M. Zajaczkowski.
\newblock {\it On nonstationary motion of a compressible barotropic viscous fluid with boundary slip condition.}
\newblock {\em J. Appl. Anal.} {\bf 4}(2):167--204, 1998.

\bibitem{EZ}
E.~Zatorska.
\newblock On a steady flow of multicomponent, compressible, chemically reacting
  gas.
\newblock {\em Nonlinearity}, 24:3267--3278, 2011.

\bibitem{EZ2}
E.~Zatorska.
\newblock On the flow of chemically reacting gaseous mixture.
\newblock {\em J. Differential Equations}, 253(12):3471--3500, 2012.

\bibitem{EZ3}
E.~Zatorska.
\newblock Mixtures: sequential stability of variational entropy solutions.
\newblock {\em J. Math. Fluid Mech.} 17, no. 3, 437--461, 2015.

\bibitem{XiXie}
X.Xi, B. Xie.
\newblock Global existence of weak solutions for the multicomponent reaction flows.
\newblock {\emph Journal of Mathematical Analysis and Applications}
Volume 441, Issue 2, 15 September 2016, Pages 801--814.

\end{thebibliography}
\end{document}